\documentclass[11pt,reqno]{amsart}
\usepackage{amssymb}
\usepackage{amsmath}
\usepackage{mathrsfs}
\usepackage{amssymb, amsmath, amsfonts}

\usepackage{mathabx}

\usepackage{color}

\usepackage[hidelinks]{hyperref}
\allowdisplaybreaks
\usepackage[numbers,sort&compress]{natbib}

\makeatletter
\def\tank#1{\protected@xdef\@thanks{\@thanks
 \protect\footnotetext[0]{#1}}}
\def\bigfoot{

 \@footnotetext}
\makeatother

\newcommand{\ea}{\end{array}}

\allowdisplaybreaks
\numberwithin{equation}{section}

\allowdisplaybreaks

\newtheorem{theorem}{Theorem}[section]
\newtheorem{lemma}{Lemma}[section]
\newtheorem{proposition}[theorem]{Proposition}

\newtheorem{condition}[theorem]{Condition}

\newtheorem{definition}[theorem]{Definition}

\newtheorem{remark}{Remark}

\newtheorem{rem}{Remark}[section]

\def\beq{\begin{equation}}
\def\nneq{\end{equation}}

\def\bthm{\begin{theorem}}
\def\nthm{\end{theorem}}

\def\blem{\begin{lemma}}
\def\nlem{\end{lemma}}
\def\bprf{\begin{proof}}
\def\nprf{\end{proof}}
\def\bprop{\begin{prop}}
\def\nprop{\end{prop}}
\def\brmk{\begin{rem}}
\def\nrmk{\end{rem}}

\def\bexa{\begin{exa}}
\def\nexa{\end{exa}}
\def\bcor{\begin{cor}}
\def\ncor{\end{cor}}

\def\RR{\mathbb{R}}

\def\EE{\mathbb{E}}

\def\cA{\mathcal{A}}

\def\cF{\mathcal{F}}
\def\cC{\mathcal{C}}
\def\cD{\mathcal{D}}
\def\cH{\mathcal{H}}
\def\cI{\mathcal{I}}

\def\cZ{\mathcal{Z}}
\def\cN{\mathcal{N}}

\newcommand\HHH{\mathfrak H}

\newcommand\HH{\mathcal H}

\def\ls{{\lesssim}}

\def\ee{{\mathbb E}}
\def\e{{\varepsilon}}

\def\FF{\mathcal {F}}

\newcommand{\ep}{\varepsilon}

\newcommand{\1}{{\bf 1}}

\newcommand{\Blc}{\Big(}
\newcommand{\Brc}{\Big)}

\newcommand{\Blk}{\Big[}
\newcommand{\Brk}{\Big]}
\newcommand{\lc}{\left(}
\newcommand{\rc}{\right)}
\newcommand{\lk}{\left[}
\newcommand{\rk}{\right]}
\newcommand{\lt}{\left }
\newcommand{\rt}{\right}

\title[LDP for stochastic    heat equation  with general rough noise]{A large deviation principle for the stochastic  heat equation  with general rough noise}

\author[R. Li]{Ruinan  Li}
\address[]{Ruinan  Li, School of Statistics and Information,  Shanghai University of International Business and Economics,  Shanghai, 201620,  China. }
\email{ruinanli@amss.ac.cn}

\author[R. Wang]{Ran Wang*}\thanks{*Corresponding author. E-mail: rwang@whu.edu.cn}
\address[]{Ran Wang, School of Mathematics and Statistics,  Wuhan University,  Wuhan, 430072,
China.}
\email{rwang@whu.edu.cn}

\author[B. Zhang]{Beibei Zhang}
\address[]{Beibei Zhang, School of Mathematics and Statistics,  Wuhan University,  Wuhan, 430072,
China.}
\email{zhangbb@whu.edu.cn}

\date{}
\begin{document}
\maketitle
\noindent{\bf Abstract}
 We study  Freidlin-Wentzell's large  deviation principle for  one  dimensional nonlinear stochastic heat equation  driven by a Gaussian noise:
$$
  \frac{\partial u^{\e}(t,x)}{\partial t}=\frac{\partial^2 u^{\e}(t,x)}{\partial x^2}+\sqrt{\e}\sigma(t, x, u^{\e}(t,x))\dot{W}(t,x),\quad   t> 0,\,
   x\in\RR,
$$
 where $\dot W$ is white in time and fractional in space with Hurst parameter $H\in(\frac 14,\frac 12)$.  Recently, Hu and Wang   {\cite{HW2019}} studied the well-posedness of this equation without  the technical  condition of $\sigma(0)=0$ which was previously assumed in Hu et al.   {\cite{HHLNT2017}}. We adopt a new sufficient condition proposed by Matoussi et al.   {\cite{MSZ}} for the weak convergence criterion of the large deviation principle.


\vskip0.3cm

\maketitle

\tableofcontents
\section{Introduction}

In this paper, we consider the following nonlinear  stochastic  heat  equation (SHE, for short) driven by  a  Gaussian noise which is white in time and fractional in space:
\begin{equation}\label{SHE}
  \frac{\partial u^{\e}(t,x)}{\partial t}=\frac{\partial^2 u^{\e}(t,x)}{\partial x^2}+\sqrt{\e}\sigma(t,x,u^{\e}(t,x))\dot{W}(t,x),\quad   t>0,\,
   x\in\RR.
\end{equation}
Here  $\e>0$, $u^{\e}(0,\cdot)\equiv1$, $W(t,x)$ is a centered Gaussian field with   {the}  covariance given by
\begin{equation}\label{CovW}
  \ee[W(t,x)W(s,y)]=\frac 12 \left(s\wedge t\right)\left( |x|^{2H}+|y|^{2H}-|x-y|^{2H} \right),
\end{equation}
for some $\frac 14<H<\frac 12$,  and $\dot{W}(t,x)=\frac{\partial^2 W}{\partial t\partial x}$.  The covariance of the noise $\dot{W}$ is given by
$$\ee[\dot{W}(t,x)\dot{W}(s,y)]=\delta_0 ( t-s)\Lambda\left(x-y\right),$$
where $\Lambda$ is a distribution, whose Fourier transform is the measure $\mu(d\xi)=c_{1,H}|\xi|^{1-2H}d\xi$ with
 \begin{align}
   c_{1,H}=&\,\frac{1}{2\pi}\Gamma(2H+1)\sin(\pi H).   \label{e.c1}
   \end{align}
$\Lambda$ can be formally written as $\Lambda(x-y)=H(2H-1)\times|x-y|^{2H-2}$. In addition, the measure $\mu$ satisfies the integrability condition $\int_{\RR}\frac{\mu(d\xi)}{1+|\xi|^2}<\infty$. However, the corresponding covariance $\Lambda$ is not locally integrable and fails to be non-negative when $H\in(\frac 14,\frac 12)$. It does not satisfy the classical Dalang's condition in \cite{Dalang1999} (there $\Lambda$ is given by a non-negative locally integrable function). See \cite{BJQ2015,HHLNT2017}  for  more details. For this reason, the approaches used in references \cite{Dalang1999, DQ, DaPrato1992,PZ1997,PZ2000} can not handle such rough covariance.

Recently, many authors studied the existence and uniqueness for the solutions  of  stochastic partial differential equations (SPDEs, for short) driven by Gaussian noises   {which are} rough in space, see, e.g., \cite{HHLNT2017, HHLNT2018, HW2019, LiuHuWang, SongSX2020}. We refer to \cite{hu19} and \cite{Song2018} for surveys. When the diffusion coefficient is  an affine function $\sigma(x)=ax+b$, Balan et al. \cite{BJQ2015} proved the existence and uniqueness of the mild solution to SHE \eqref{SHE}  by using   {the} technique of Fourier analysis, and they
established the H\"older continuity of the solution in \cite{BJQ2016}.

For the general nonlinear coefficient $\sigma$, which has a Lipschitz derivative and satisfies $\sigma(0)=0$, the problem of the  existence and uniqueness of the solution was proved by Hu et al. \cite{HHLNT2017}. Under this condition,  the large deviations, the moderate deviations   and the transportation inequalities  were studied in \cite{HNZ2018}, \cite{LiuJunfeng} and \cite{DaiyinLiruinan}, respectively.

In \cite{HW2019}, Hu and Wang  removed the technical and unusual condition of $\sigma(0) = 0$ and they proved the  well-posedness of the solution to Eq.\,\eqref{SHE}   under  Condition $(\mathbf{H})$ in Subsection 2.3 below.  Without the condition of $\sigma(0) = 0$, the solution is no longer in the space $\mathcal{Z}^p_T$ (see \cite{HHLNT2017} or \eqref{ZNorm} in Subsection 2.3 of this paper with $\lambda(x)\equiv1$).  To relax the restriction, Hu and Wang \cite{HW2019} introduced a decay weight (as the spatial variable $x$ goes to
infinity) to enlarge the solution space $\mathcal{Z}^p_{T}$ to a weighted space $\mathcal{Z}^p_{\lambda,T}$ for some suitable power decay function $\lambda(x)$, see Subsection 2.3 below for details.

The aim of this paper is to establish a large deviation principle (LDP, for short) for the solution $u^\e$ of \eqref{SHE} as $\e\rightarrow0$. A useful  approach of investigating the LDP is the well-known weak convergence method (see, e.g., \cite{BCD2013, BD, BD2019,  BDM2008, BDM2011, DE1997}), which is mainly based on the variational representation formula for measurable functionals of Brownian motion. For some relevant LDP results by using the weak convergence method, we refer to
  \cite{L, RZ2008, XuZhang2009} and the references therein. In particular, Liu et al. \cite{LTZ} and Xiong and Zhai \cite{XZ2018} proved the LDP for a large class of SPDEs with locally   {monotone} coefficients driven by Brownian motions or L\'evy noises, respectively. However, the frameworks of \cite{LTZ} and  \cite{XZ2018} cannot be applied to SHE \eqref{SHE} because of the spatial rough noise with $H\in (\frac{1}{4}, \frac{1}{2})$.

    Comparing with the   {LDP result}  for SHE \eqref{SHE}  under the assumption of $\sigma(0)=0$ in Hu et al. \cite{HNZ2018},   our result   and its proof  are in  the weighted space $\mathcal{Z}^p_{\lambda,T}$.   Notice that  the introduction of the weight brings many difficulties. See  Hu and Wang \cite{HW2019} for the excellent analysis.  In this paper, we adopt  a new sufficient condition for the LDP (see Condition \ref{cond1} below) which was  proposed by Matoussi et al.  \cite{MSZ}. This method was successfully applied to the study of   {LDPs} for SPDEs, see, e.g., \cite{DongWuZhang2020, HHL, HLL2021,  LSZ2020, WanfZhang2021, WuZhai2020}.

The rest of this paper is organized as follows. In Section 2, we first present the  framework introduced in \cite{HW2019} and recall   the weak convergence method  given by \cite{BDM2008, MSZ}, then we formulate the main result of the present paper.
In Section 3, the associated skeleton equation is  studied. Sections 4 and 5 are devoted to verifying the two conditions for the weak convergence criterion. Finally, we give some useful lemmas in Appendix.

We adopt the following notations throughout this paper. We always use $C_\alpha$  to denote a constant  dependent on the parameter $\alpha$,  which may change from line  to
 line. $A\lesssim B$ ($A\gtrsim B$, resp.) means that $A\leq CB$ ($A\geq CB$, resp.) for some positive   {generic constant $C$, the value of which might change from line to line},  and $A\simeq B$ means that  $A\lesssim B$ and $A\gtrsim B$.

\section{Preliminaries and main result}\label{lemmas}
In this section, we first give some preliminaries of SHE \eqref{SHE}, then we state the weak convergence criterion for the LDP in \cite{BDM2008, MSZ} and give the main result of this paper.
\subsection{Covariance structure and stochastic integration}
Recall  some notations from \cite{HHLNT2017} and \cite{HW2019}.   Denote by  $ \cD=\cD(\RR) $  the space of   {real-valued infinitely differentiable} functions with compact support on $\mathbb{R}$, and by $\cD'$  the   dual of $\cD$ with respect to the $L^2(\RR, dx)$,   { where $L^2(\RR, dx)$ is the Hilbert space of square integrable functions with respect to Lebesgue measure on $\RR$}. The Fourier transform of a function $f\in\cD$ is defined as
 $$
  \cF f(\xi):=\int_{\RR} e^{-i\xi x}f(x) dx.
$$ 

 Let   $(\Omega,\cF,\mathbb P)$ be a complete probability space and   { $\mathcal{D}(\mathbb{R}_{+}\times \mathbb{R})$ be the space of real-valued infinitely differentiable functions with compact support on $\mathbb{R}_{+}\times \mathbb{R}$.}
The noise $\dot W$ is  a zero-mean Gaussian family $\{W(\phi), \phi\in\cD(\RR_{+}\times \RR)\}$  with the covariance structure given by
 \begin{equation}\label{CovStru}
   \EE\lk W(\phi)W(\psi)\rk=c_{1,H}\int_{\RR_{+}\times \RR} \cF \phi(s,\xi) \overline{\cF \psi(s,\xi)}\cdot |\xi|^{1-2H}d\xi ds,
 \end{equation}
 where $H\in (\frac14 , \frac 12)$, $c_{1,H}$ is given in \eqref{e.c1},
 and $\cF \phi(s,\xi)$ is the Fourier transform with respect to the spatial variable $x$ of the function $\phi(s,x)$.
 Then \eqref{CovStru}  defines   a Hilbert scalar product on $\cD(\RR_{+}\times \RR) $.
Denote $\HHH$ the Hilbert space obtained by completing $\cD(\RR_{+}\times \RR) $ with respect to this scalar  product.
\begin{proposition}  {(}\cite[Proposition 2.1]{HW2019}, \cite[Theorem 3.1]{PT2000}  {)}\label{hSpaceProp}
 The  space $\HHH$ is a Hilbert space equipped with the scalar  product
\begin{equation}
\begin{split}
     \langle\phi,\psi\rangle_{\HHH}
     :=&\,c_{1,H}\int_{\RR_{+}}\left(\int_ {\RR} \cF \phi(t,\xi) \overline{\cF \psi(t,\xi)}\cdot |\xi|^{1-2H}d\xi \right)dt
     \\
     =&\,c_{2,H} \int_{\RR_+}\left(\int_{\RR^2}[\phi(t, x+y)-\phi(x)]\cdot[\psi(t, x+y)-\psi(x)]\cdot |y|^{2H-2} dxdy\right)dt,
   \end{split}
   \end{equation}
   where
   \begin{align}
   c_{2,H} :=\left(\frac 12-H\right)^{\frac{1}{2}}H^{\frac{1}{2}} \lk\Gamma\Blc H+\frac 12\Brc\rk^{-1}\left(\int_{0}^{\infty} \Blk(1+t)^{H-\frac 12}-t^{H-\frac 12}\Brk^2 dt+\frac{1}{2H}\right)^{\frac{1}{2}}.  \label{e.c3}
   \end{align}
 The space  $\cD(\RR_{+}\times \RR)$ is dense in $\HHH$.
 \end{proposition}

  {For any $t\ge0$, let $\cF_t$ be the filtration generated by $W$, that is
$$
 \cF_t:=\sigma\big\{W(\phi): \phi\in \cD([0,t]\times\RR)\big\},
$$
where $\cD([0,t]\times\RR)$ is  the space of real-valued infinitely differentiable functions on $[0,t]\times\RR$.}

\begin{definition}\label{ElemP}
An elementary process $g$ is a process given by
   \[
    g(t,x)=\sum_{i=1}^{n}\sum_{j=1}^{m} X_{i,j}\1_{(a_i,b_i]}(t)\1_{(h_j,l_j]}(x),
   \]
   where $n$ and $m$ are finite positive integers, $0\leq a_1<b_1<\cdots<a_n<b_n<\infty$, $h_j<l_j$ and $X_{i,j}$ are $\cF_{a_i}$-measurable random variables for $i=1,\dots,n$, $j=1,\dots,m$. The stochastic integral of   an elementary  process with respect to $W$ is defined as
   \begin{equation}\label{SI_Ele}
     \begin{split}
      \int_{\RR_{+}}\int_{\RR} g(t,x)W(dt,dx)
     =&\, \sum_{i=1}^{n}\sum_{j=1}^{m}  X_{i,j}W(\1_{(a_i,b_i]}\otimes \1_{(h_j,l_j]}) \\
      =&\, \sum_{i=1}^{n}\sum_{j=1}^{m}  X_{i,j}\Big[W(b_i,l_j)-W(a_i,l_j)-W(b_i,h_j)+W(a_i,h_j)\Big].
     \end{split}
   \end{equation}
 \end{definition}

  {
Hu et al. \cite[Proposition 2.3]{HHLNT2017}  extent the the notion of integral with respect to $W$ to a broad class of adapted processes in the following way.}
\begin{proposition}(\cite[Proposition 2.3]{HHLNT2017})\label{prop 2.3}
   Let $\Lambda_{H}$ be the space of predictable processes $g$ defined on $\RR_{+}\times\RR$ such that almost surely $g\in\HHH$ and $\EE[\|g\|_{\HHH}^2]<\infty$.  Then, we have that:
   \begin{itemize}
       \item[(i).]
   the space of the elementary processes
    defined in Definition \ref{ElemP} is dense in $\Lambda_{H}$;
      \item[(ii).]
 for any  $g\in\Lambda_{H}$, the stochastic integral $\int_{\RR_{+}}\int_{\RR} g(s,x)W(ds,dx)$ is defined  as the $L^2(\Omega)$-limit of  Riemann sums along    elementary processes approximating $g$
in $\Lambda_H$, and we have
    \begin{equation}\label{Isometry}
     \EE\lk\lc\int_{\RR_{+}}\int_{\RR} g(s,x)W(ds,dx)\rc^2\rk=\EE\lk\|g\|_{\HHH}^2\rk.
    \end{equation}
\end{itemize}
\end{proposition}

  {Let $\mathcal H$  be
	  the Hilbert space obtained by completing $\mathcal D(\RR) $ with respect to the following scalar  product:
	  \begin{equation}\label{eq H product}
\begin{split}
     \langle\phi,\psi\rangle_{\mathcal H}
     =&\,c_{1,H}\int_{\RR} \cF \phi(\xi) \overline{\cF \psi(\xi)}  \cdot |\xi|^{1-2H}d\xi\\
     =&\, c_{2,H}\int_{\RR^2}[\phi(x+y)-\phi(x)]\cdot [\psi(x+y)-\psi(x)]\cdot |y|^{2H-2} dxdy, \ \, \forall \phi,\psi\in \mathcal D(\mathbb R).
        \end{split}
   \end{equation}
 By Proposition \ref{prop 2.3}, for any orthonormal basis $\{e_k\}_{k\ge1}$ of the Hilbert space $\HH$, the family of processes
\begin{equation}\label{eq int 10}
\left\{B_t^k:=\int_0^t\int_{\mathbb{R}}e_k(y)W(ds,dy)\right\}_{k\ge1}
\end{equation}
is a sequence of independent standard Wiener processes and the process
$
B_t:=\sum_{k\ge1}B_t^ke_k
$
is a cylindrical Brownian motion on $\HH$.
It is well-known that (see \cite{DaPrato1992} or \cite{DQ}) for any $\HH$-valued predictable process
 $g\in L^2(\Omega\times[0,T];\HH)$,
we can define the stochastic integral with respect to the cylindrical Wiener process $B$ as follows:					
\begin{equation}\label{eq int2}
\int_0^T g(s)dB_s:=\sum_{k\ge1}\int_0^T \langle g(s),e_k\rangle_{\HH} dB_s^k.
\end{equation}
  Note that the above series converges in $L^2(\Omega, \FF,\mathbb P)$ and the sum does not depend on
the selected orthonormal basis. Moreover, each summand, in the above series, is a classical It\^o
integral with respect to a standard Brownian motion.}

 Let $(B,\| \cdot \|_B)$ be a Banach space  with the norm $\| \cdot \|_B$.  Let   $H\in(\frac{1}{4},\frac{1}{2})$ be a fixed number.
 For  any  function $f:\RR\rightarrow B$,  denote
 \begin{equation}\label{NBNorm}
   \cN_{\frac{1}{2}-H}^{B}f(x):=\lt(\int_{\RR}\|f(x+h)-f(x)\|_B^2\cdot |h|^{2H-2}dh\rt)^{\frac 12},
 \end{equation}
 if the above quantity is finite.
 When $B=\RR$, we abbreviate the notation $\cN_{\frac{1}{2}-H}^{\RR}f$  as  $\cN_{\frac{1}{2}-H}f$.
 As in \cite{HHLNT2017}, when $B=L^p(\Omega)$, we  denote $ \cN_{\frac{1}{2}-H}^{B}$ by $\cN_{\frac{1}{2}-H,\,p}$, that is,
 \begin{equation}\label{NpNorm}
   \cN_{\frac{1}{2}-H,\,p}f( x):=\lt(\int_{\RR}\|f(x+h)-f(x)\|^2_{L^p(\Omega)} \cdot |h|^{2H-2}dh\rt)^{\frac 12}.
 \end{equation}

The following    Burkholder-Davis-Gundy's     inequality is well-known
(see,  e.g.,  \cite{HHLNT2017}).

 \begin{proposition}\label{HBDG}  {(}\cite[Proposition 3.2]{HHLNT2017}  {)}
   Let $W$ be the Gaussian noise with the covariance \eqref{CovStru}, and let    $f\in\Lambda_H$  be a predictable random field. Then, we have   {that} for any $p\geq2$,
   \begin{equation}
     \begin{split}
        \lt\|\int_{0}^{t}\int_{\RR}f(s,y) W(ds,dy)\rt\|_{L^p(\Omega)}
     \leq \sqrt{4p}c_{H}\lt(\int_{0}^{t}\int_{\RR}\lk\cN_{\frac 12-H,\,p}f(s,y)\rk^2dyds\rt)^{\frac 12},\label{e.bdg}
     \end{split}
   \end{equation}
   where $c_{H}$ is a constant depending only on $H$ and
   $\cN_{\frac 12-H,\,p}f(s,y)$ denotes the application of $\cN_{\frac 12-H,\,p} $  to the spatial variable $y$.
 \end{proposition}

\subsection{Stochastic heat equation}

Let $\mathcal{C}([0, T]\times\RR)$ be the space of all continuous real-valued functions on $[0,T]\times\mathbb{R}$, equipped with the following metric:
\begin{equation}
  d_{\mathcal C}(u,v):=\sum_{n=1}^{\infty}\frac{1}{2^n} \max_{0\le t\le T,|x|\leq n}(|u(t,x)-v(t,x)|\wedge 1).\label{e.4.metric}
 \end{equation}

Define the weight  function
   \begin{equation}\label{lamd}
     \lambda(x):=c_H(1+|x|^2)^{H-1}\,,
 \end{equation}
where $c_H$ is a constant such that $\int_{\RR} \lambda(x) dx=1$.

For any $p\geq 2$ and $\frac 14<H<\frac 12$, we introduce a norm $\|\cdot\|_{\cZ_{\lambda,T}^p}$ for a random field $v=\{v(t,x)\}_{(t,x)\in[0,T]\times\RR}$ as follows:
 \begin{equation}\label{ZNorm}
   \|v\|_{\cZ_{\lambda,T}^p} :=\sup_{t\in[0,T]} \|v(t,\cdot)\|_{L^p_{\lambda}(\Omega\times\RR)}+\sup_{t\in[0,T]}\cN^*_{\frac 12-H,\,p}v(t),
 \end{equation}
where
  \begin{equation}\label{LpNormOmegatimesRR}
  \|v(t,\cdot)\|_{L^p_{\lambda}(\Omega\times\RR)}:=\lc\int_{\RR} \EE\lk|v(t,x)|^p\rk \lambda(x)dx\rc^{\frac 1p},
 \end{equation}
 and
 \begin{equation}\label{N*pNorm}
   \cN^*_{\frac 12-H,\,p}v(t):=\lc\int_{\RR}\|v(t,\cdot)-v(t,\cdot+h)\|^2_{L^p_{\lambda}(\Omega\times\RR)}\cdot |h|^{2H-2}dh\rc^{\frac 12}.
 \end{equation}
 Denote $\cZ_{\lambda,T}^p$  the space of  all random fields $v=\{v(t,x)\}_{(t,x)\in[0,T]\times\RR}$ such that $\|v\|_{\cZ_{\lambda,T}^p}$ is finite.

For the well-posedness of the solution and the large deviation principle, we assume the following conditions.
\begin{enumerate}
  \item[\textbf{(H)}]   Assume that $\sigma(t, x, u)\in\cC^{0,1,1}([0, T]\times\RR^2)$ (the space of all continuous functions $\sigma$, with continuous partial derivatives $\sigma'_x$, $\sigma'_u$ and $\sigma''_{xu}$), and there exists a constant $C>0$ such that

      \begin{align}\label{lineargrowthuniformlycond}
      \sup_{t\in[0,T],\,x\in\RR}\left|\sigma(t, x, u)\right|\leq C(1+|u|),\,\,\,\forall u\in\RR;
      \end{align}

      \begin{align}\label{Lipschitzianuniformlycondu}
      \sup_{t\in[0,T],\,x\in\RR}\left|\sigma(t, x, u)-\sigma(t, x, v)\right|\leq C|u-v|,\,\,\,\forall u, v\in\RR;
      \end{align}

  \begin{align}
  \sup_{t\in[0,T], \,x\in\RR,\, u\in\RR} |\sigma'_u(t,x,u)| &\leq C\,; \label{DuSigam}
  \end{align}

\begin{align}
  \sup_{t\in[0,T], \,x\in\RR} |\sigma'_x(t,x,0)| &\leq C\,; \label{DuySigam}
  \end{align}

  \begin{align}
  \sup_{t\in[0,T], \,x\in\RR, \, u\in\RR} |\sigma''_{xu}(t,x,u)| &\leq C\,. \label{DuxSigam}
  \end{align}
Moreover,
there exist some constants $p_0>\frac{6}{4H-1}$ and $C>0$ such that
  \begin{equation}\label{DuSigamAdd}
   \sup_{t\in[0,T],  \,x\in\RR}\lambda^{-\frac{1}{p_0}}(x)\left| \sigma'_u(t,x,u_1)-\sigma'_u(t,x,u_2)\right| \le C  |u_1-u_2| , \,\,\,\forall u_1,u_2\in \RR,
  \end{equation}
   where $\lambda(x)$ is the weight  function defined by \eqref{lamd}.
\end{enumerate}

\begin{remark} By the monotonicity of $\lambda(x)$, we know that if \eqref{DuSigamAdd} holds for some $p_0>\frac{6}{4H-1}$, then it also holds for all  $p\ge p_0$.
\end{remark}

 Let  $p_{t}(x):=\frac{1}{\sqrt{4\pi t}}\exp\lc-\frac{x^2}{4t}\rc$ be the heat kernel  associated with
the  Laplacian operator $\Delta$.
Recall the following   definition of  the solution to SHE \eqref{SHE} from \cite{HW2019}.
\begin{definition}  {(}\cite[Definition 1.4]{HW2019}  {)}\label{DefMildSol}
  Given the initial value $u_0(\cdot)\equiv1$, a real-valued adapted  stochastic process $u^\e$ is called a   {mild}  solution to \eqref{SHE},  if for all $t\ge 0$ and $x\in\RR$,
   \begin{equation}\label{MildSol}
     u^\e(t,x)=\,1+\sqrt{\e}\int_{0}^{t}\int_{\RR} p_{t-s}(x-y)\sigma(s,y,u^\e(s,y))W(ds,dy), \ \  \text{a.s.}.
   \end{equation}

\end{definition}

The following theorem  follows from \cite{HW2019}.
\begin{theorem}  {(}\cite[Theorem 1.6]{HW2019}  {)}\label{ModUniq}
Assume that  $\sigma$ satisfies the  hypothesis $(\mathbf{H})$. Then \eqref{SHE} admits a unique   {mild} solution in $\mathcal C([0, T]\times\RR)$ almost surely.
\end{theorem}

\subsection{A general criterion for the large deviation principle}\label{A Criteria for Large Deviations}

Let $\{u^\e\}_{\e>0} $ be a family of random variables defined on a probability space
$(\Omega,\mathcal{F},\mathbb{P})$
and taking values in a Polish space $E$.
 \begin{definition}\label{Dfn-Rate function}
     A function $I: E\rightarrow[0,\infty]$ is called a rate function on
       $E$,
       if for each $M<\infty$ the level set $\{y\in E:I(y)\leq M\}$ is a compact subset of $E$.

    \end{definition}
    \begin{definition}  \label{d:LDP}
       Let $I$ be a rate function on $E$. The sequence
       $\{u^\e\}_{\e>0}$
       is said to satisfy  a large deviation principle on $E$ with the rate function $I$, if the following two
       conditions
       hold:
\begin{itemize}
\item[(a).]   for each closed subset $F$ of $E$,
              $$
                \limsup_{\e\rightarrow 0}\e \log\mathbb{P}(u^\e\in F)\leq  {-\inf_{y\in F}}I(y);
              $$

        \item[(b).]  for each open subset $G$ of $E$,
              $$
                \liminf_{\e\rightarrow 0}\e \log\mathbb{P}(u^\e\in G)\geq-\inf_{y\in G}I(y).
              $$
              \end{itemize}
    \end{definition}

 Set   { $\mathbb{V}=C([0,T];\mathbb R^{\infty})$}
 and let $\mathbb U$ be a Polish space. Let $\{\Gamma^{\varepsilon}\}_{\varepsilon >0}$ be a family of measurable maps from $\mathbb V$ to $\mathbb U$.
We   recall a criterion for  the LDP of the family $Z^{\varepsilon}=\Gamma^{\varepsilon}( W)$ as $\varepsilon \rightarrow 0$,   {where and throughout this section $W$  is the Gaussian
process identified a sequence of independent, standard, real valued Brownian motions,  by using the representation formulae of \eqref{eq int 10}.}
 
Define the following space of stochastic processes:
\begin{equation}\label{eq: space}
\mathcal L_2:=\left\{\phi: \Omega\times [0,T]\rightarrow \mathcal H \text{  is predictable and  } \int_0^T\|\phi(s)\|_{\mathcal H}^2ds<\infty, \ \ \mathbb{P}\text{-a.s.} \right\},
\end{equation} and  denote $L^2([0,T];\mathcal H)$ the space of square integrable $\mathcal H$-valued functions on $[0,T]$. For each   {$N\geq1$}, let
\begin{equation}\label{eq: space SN}
S^N=\left\{g\in L^2([0,T];\mathcal H): L_T(g)\le N  \right\},
\end{equation}
 where
\begin{equation}\label{eq: LTg}
L_T(g):=\frac12\int_0^T\|g(s)\|_{\mathcal{H}}^2ds,
\end{equation}
 and $S^N$ is equipped with the topology of the weak convergence in $L^2([0,T];\mathcal H)$. Set $\mathbb S=\bigcup_{N\ge1} S^N$, and
$$
\mathcal U^N=\left\{g\in \mathcal L_2: g(\omega)\in S^N, \,\,\mathbb{P}  {\text{-a.s.}} \right\}.
$$
\begin{condition}\label{Aa} There exists a measurable mapping $\Gamma^0:\mathbb V\rightarrow\mathbb U$ such that the following two items hold:
\begin{itemize}
\item[(a).]  for every $N<+\infty$, the set $K_N=\left\{\Gamma^0\left(\int_0^{\cdot}   g(s)ds\right): g\in S^N\right\}$ is a compact subset of $\mathbb U$;
\item[(b).] for every $N<+\infty$ and  $\{g^\e\}_{\e>0}\subset \mathcal U^N$, if $g^\e$ converges to $g$ as $S^N$-valued random elements in distribution, then  $\Gamma^\e\left(W+\frac{1}{\sqrt\e} \int_0^{\cdot}  g^{\e}(s)ds\right)$ converges in distribution to $\Gamma^0\left(\int_0^{\cdot}  g(s)ds\right)$.
 \end{itemize}
  \end{condition}
  Let
   $I:\mathbb U\rightarrow [0,\infty]$ be defined  by
\beq\label{rate function}
I(\phi):=\inf_{\left\{g\in \mathbb{S};\,\phi=\Gamma^0\left(\int_0^{\cdot}g(s)ds\right)\right\}}  \{L_T(g)\},\  \phi\in\mathbb U,
\nneq
with the convention $\inf\emptyset=\infty$.

The following result is due to Budhiraja et al. \cite{BDM2008}.
\begin{theorem}(\cite[Theorem 6]{BDM2008})\label{thm MSZ} For  any $\e>0$, let $X^{\e}=\Gamma^\e(W)$ and suppose that Condition \ref{Aa} holds. Then,  the family $\{X^{\e}\}_{\e>0}$ satisfies an LDP with the rate function  $I$ defined by \eqref{rate function}.
\end{theorem}

Recently, a new sufficient condition (Condition \ref{cond1} below) for verifying the assumptions in Condition \ref{Aa} for LDPs was proposed by Matoussi et al. \cite{MSZ}. It turns out that this new sufficient condition seems to be more suitable for establishing the LDP for SPDEs.
  \begin{condition}\label{cond1} There exists a measurable mapping $\Gamma^0:\mathbb V\rightarrow\mathbb U$ such that the following two items hold:
\begin{itemize}
\item[(a).] for every $N<+\infty$ and any family $\{g^n\}_{n\ge1}\subset S^N$ that converges to some element $g$ in $S^N$ as $n\rightarrow\infty$, $\Gamma^0\left(\int_0^{\cdot} g^n(s)ds\right)$ converges to  $\Gamma^0\left(\int_0^{\cdot}   g(s)ds\right)$ in the space $\mathbb U$;
   \item[(b).] for every $N<+\infty$, $\{g^\e\}_{\e>0}\subset \mathcal U^N$  and $\delta>0$,
    $$\lim_{\e\rightarrow 0}\mathbb P\big(\rho(Y^{\e}, Z^{\e})>\delta\big)=0, $$
     where  $Y^{\e}=\Gamma^\e\left(W+\frac{1}{\sqrt\e} \int_0^{\cdot}  g^{\e}(s)ds\right), Z^{\e}=\Gamma^0\left(\int_0^{\cdot}  g^{\e}(s)ds\right)$ and $\rho(\cdot, \cdot)$ stands for the metric in the space $\mathbb U$.
 \end{itemize}
  \end{condition}
\subsection{Main result}
To state our main result and give its proof, we need to introduce a map $\Gamma^0$ appeared in Condition \ref{Aa} (or Condition \ref{cond1}). Given $g\in\mathbb{S}$,  consider the following deterministic integral equation (the skeleton equation):
\begin{equation}\label{eq skeleton}
u^g(t,x)=1+\int_0^t\langle p_{t-s}(x-\cdot)\sigma(s, \cdot, u^g(s, \cdot)), g(s,\cdot)\rangle_{\mathcal H}ds, \ \ \ t\ge0,\,\, x\in \mathbb R.
 \end{equation}
 By Proposition \ref{thm solu skeleton} below, Eq.\,\eqref{eq skeleton} admits a unique solution $u^g\in \mathcal{C}([0,T]\times\mathbb{R})$.
 For any $g\in\mathbb{S}$, we define
 \begin{align}\label{eq Gamma0}
 \Gamma^0\left(\int_0^{\cdot}  g(s)ds\right):=u^g(\cdot).
 \end{align}

The following is the main result of this paper.
\begin{theorem}\label{thm LDP}  Assume  that    the hypothesis  $(\mathbf{H})$ holds. Then, the family $\{u^{\e}\}_{\e>0}$ in Eq.\,\eqref{SHE} satisfies an LDP in the space  $\cC([0, T]\times\RR)$ with the rate function $I$ given by
\begin{equation}\label{eq rate}
I(\phi):=\inf_{\left\{g\in \mathbb S;\, \phi=\Gamma^0\left(\int_0^{\cdot}  g(s)ds\right)\right\} }\left\{ \frac12 \int_0^{T}\|g(s)\|_{\mathcal H}^2ds \right\}.
\end{equation}
\end{theorem}
 \begin{proof}
According to Theorem \ref{thm MSZ} and \cite[Theorem 3.2]{MSZ}, it suffices to show that the conditions (a) and (b) in Condition \ref{cond1} are satisfied. Condition (a) will be proved in  Proposition \ref{thm continuity skeleton}, and Condition (b) will be verified
   in Proposition \ref{proposition-4-3}.
The proof is complete.
\end{proof}

\section{Skeleton equation}

In this section, we study the well-posedness of  the  skeleton equation \eqref{eq skeleton}.

 For $p\geq 2$, $H\in(\frac 14,\frac 12)$, recall the space $\mathcal Z_{\lambda, T}^p$  with the norm \eqref{ZNorm}.
The space of all non-random functions in $\mathcal Z_{\lambda, T}^p$ is denoted by $Z_{\lambda, T}^p$,  with the following norm:
\begin{equation}\label{YNorm}
   \|v\|_{Z_{\lambda,T}^p} :=\sup_{t\in[0,T]} \|v(t,\cdot)\|_{L^p_{\lambda}(\RR)}+\sup_{t\in[0,T]}\cN_{\frac 12-H,\, p}^*v(t),
 \end{equation}
 where
\begin{align}\label{LPlambdaRNorm}
  \|v(t,\cdot)\|_{L^p_{\lambda}(\RR)}:=\lc\int_{\RR} |v(t,x)|^p \lambda(x)dx\rc^{\frac 1p},
\end{align}
and
 \begin{equation}\label{N*pNorm}
   \cN^*_{\frac 12-H,\,p}v(t):=\lc\int_{\RR}\|v(t,\cdot)-v(t,\cdot+h)\|^2_{L^p_{\lambda}(\RR)}\cdot |h|^{2H-2}dh\rc^{\frac 12}.
 \end{equation}

\begin{proposition}\label{thm solu skeleton}
 Assume that  $\sigma$ satisfies the hypothesis $(\mathbf{H})$. Then, Eq.\,\eqref{eq skeleton} admits a unique solution $u^g$ in $\mathcal{C}([0, T]\times\mathbb R)$.
 In addition, $\displaystyle\sup_{g\in S^N}\|u^g\|_{Z^p_{\lambda,T}}<\infty$ for any $N\geq 1$ and any  $p\geq 2$.
  \end{proposition}

 Due to the complexity of the space $\mathcal H$, it is difficult to prove Proposition \ref{thm solu skeleton} by using Picard's iteration directly. We use the approximation method by introducing a new Hilbert space $\mathcal H_{\e}$ as follows.

 For every fixed $\varepsilon>0$, let
\begin{equation}\label{eq f e}
f_{\e}(x):=\frac{1}{2\pi} \int_{\mathbb R} e^{i\xi x}e^{-\e |\xi|^2}|\xi|^{1-2H}  d\xi.
	\end{equation}
For any $\phi, \psi \in \mathcal{D}(\RR)$, we define
	\begin{equation}\label{eq product H e}
    \begin{split}
     \langle\phi,\psi\rangle_{\mathcal H_{\e}}
     :=&\,c_{1,H}\int_{\RR} \cF \phi(\xi) \overline{\cF \psi(\xi)} e^{-\e|\xi|^2} |\xi|^{1-2H}d\xi\\
     =&\,c_{1, H}\int_{\mathbb{R}^2}\phi(x)\psi(y)f_{\e}(x-y)dxdy,
   \end{split}
   \end{equation}
   where  $c_{1,H}$ is given by  \eqref{e.c1}.    Let $\mathcal H_{\e}$  be
	  the Hilbert space obtained by completing $\mathcal D(\RR) $ with respect to the scalar  product given by  \eqref{eq product H e}.
	For any $0\le \e_1<\e_2$, we have  that for any $\phi\in \mathcal H_{\e_1},$
	 \begin{align}\label{eq H compare}
  \|\phi\|_{\mathcal H_{\e_1}}\ge \|\phi\|_{\mathcal H_{\e_2}},
	 \end{align}
and  we have by the dominated convergence theorem,
$$
\lim_{\e\rightarrow0}\, \langle \phi, \psi \rangle_{\HH_{\e}}=\langle \phi, \psi \rangle_{\HH} \ \ \text{ for any } \phi, \psi\in \mathcal H.
$$

For any $g\in \mathbb{S}$, let
\begin{align}\label{eq u h e}
		u_{\ep}^{g}(t,x)=1 + \int_0^t \langle  p_{t-s}(x-\cdot)\sigma(s,\cdot, u^g_{\ep}(s,\cdot)), g(s,\cdot) \rangle_{\HH_\e}ds.
\end{align}

 Since $|\xi|^{1-2H}e^{-\varepsilon |\xi|^2}$ is in $L^1(\mathbb R)$, $|f_{\varepsilon}|$ is bounded.  Due to the regularity in space, the existence and uniqueness of the solution $u^{g}_{\ep}$ to Eq.\,\eqref{eq u h e} is well-known, see, e.g., \cite[Section 4]{MCS}.

For any $t \ge0, x, y, h\in \mathbb R$, let
\begin{align}\label{TechDt}
D_{t}(x,h):=p_{t}(x+h)-p_{t}(x),   
\end{align}
and
\begin{align}\label{TechBoxt}
\Box_{t}(x,y,h):=p_{t}(x+y+h)-p_{t}(x+y)-p_{t}(x+h)+p_{t}(x).
\end{align}

The following lemma asserts that the approximate solution $u_\ep^g$  is   {bounded in the space $Z_{\lambda,T}^p$ uniformly over $\e>0$.}
 \begin{lemma}\label{UniBExist}
Let $H\in(\frac 14,\frac 12)$ and $g\in \mathbb{S}$.
Assume that   $\sigma$ satisfies the hypothesis $(\mathbf{H})$.
Then,  for any $p\ge2$,
   \begin{equation}\label{ReguBdd}
      { \sup_{\ep>0}\|u_\ep^g\|_{Z_{\lambda,T}^p}
     <\infty.}
   \end{equation}
 \end{lemma}
	  \begin{proof}
	  We use the similar argument   as  in \cite{HW2019} replacing the stochastic integral by the  deterministic integral.
 We first define  Picard's iteration sequence. For any $t\ge0, x\in \mathbb R$,  let
$$
    u_\ep^{g, 0}(t,x)=1,
$$
   and recursively  for $n=0, 1, 2, \cdots$,
  \begin{equation}\label{4-38}
    u_\ep^{g, n+1}(t,x)=1+\int_{0}^{t}\langle p_{t-s}(x-\cdot)\sigma(s,\cdot,u_\ep^{g, n}(s,\cdot)), g(s,\cdot) \rangle_{\HH_\e}ds.
  \end{equation}
   Since $g\in  \mathbb{S}$,   there exists some constant $N>0$ such that $\int_0^T \|g(s,\cdot)\|_{\HH}^2ds\le N$. By \eqref{eq H compare}, we know that
  \begin{align}\label{g-interal-bound}
  \int_0^T \|g(s,\cdot)\|_{\HH_{\e}}^2ds\le N \,\,\,\emph{for}\,\,\emph{any}\,\,\e>0.
  \end{align}
   {In Step 1, we prove  the convergence of $u_\ep^{g, n}(t,\cdot)$ in $L^p_{\lambda}(\mathbb R)$ for any $t\in[0,T]$. In Steps 2 and 3,  we   give  some quantitative   estimates  for  $\|u_\ep^{g,n}(t)\|_{L^p_{\lambda}(\RR)}$  and $\mathcal N^*_{\frac12-H,\, p} u_{\ep}^{g,n}(t)$ for each fixed $\ep>0$. Step 4 is devoted to  proving that  $u_\ep^g$ is   bounded  in  $\left(Z_{\lambda,T}^p,\, \|\cdot\|_{Z_{\lambda,T}^p}\right)$   {uniformly over    }$\ep>0$.}

{\bf Step 1.} In this step,  we will  prove that for any fixed $\e>0$,  as    $n$ goes to  infinity, the sequence $u_\ep^{g, n}(t,\cdot)$ converges to $u_\ep^{g}(t,\cdot)$   {in $L_{\lambda}^{p}(\mathbb R)$.}

By the  Cauchy-Schwarz  inequality, \eqref{Lipschitzianuniformlycondu}, the boundedness of $f_ {\e}$ and Jensen's inequality respect to $p_{t-s}(x-y)dy$, we have that,  for any $t\in [0,T]$ and $x\in \RR$,
     \begin{align*}
      & |u_\ep^{g, n+1}(t,x)-u_\ep^{g,n}(t,x)|^2 \\
     =&\, \left|\int_{0}^{t}\left\langle p_{t-s}(x-\cdot)\big[\sigma(s,\cdot,u_\ep^{g, n}(s,\cdot))-\sigma(s,\cdot,u_\ep^{g, n-1}(s,\cdot))\big], g(s,\cdot)\right\rangle_{\HH_\e}ds\right|^2  \\
     \le  &\,   \int_0^t \|g(s,\cdot)\|_{\HH_\e}^2ds\cdot \int_{0}^{t}\left\|p_{t-s}(x-\cdot)\left[\sigma(s,\cdot,u_\ep^{g, n}(s,\cdot))-\sigma(s,\cdot,u_\ep^{g, n-1}(s,\cdot))\right]\right\|_{\HH_\e}^2ds\\
     \leq&\,  c_{1,H} N    \int_0^t \int_{\RR}\int_{\RR}  p_{t-s}(x-y)\left[\sigma(s,y,u_\ep^{g, n}(s,y)-\sigma(s,y,u_\ep^{g, n-1}(s,y))\right] \\
     &\qquad\qquad\qquad\quad \cdot p_{t-s}(x-z)\left[\sigma(s,z,u_\ep^{g, n}(s,z))-\sigma(s,z,u_\ep^{g, n-1}(s,z))\right]f_{\e}(y-z)dydzds\\
     \le&\, c_{1,H}  N\|f_{\e}\|_{\infty} \cdot\int_0^t\int_{\RR } p_{t-s}(x-y)\left[\sigma(s,y,u_\ep^{g, n}(s,y)-\sigma(s,y,u_\ep^{g, n-1}(s,y))\right]^2dyds.
   \end{align*}
  Integrating  with respect to the spatial variable with the weight $\lambda(x)$ and invoking  \eqref{Lipschitzianuniformlycondu}, Jensen's inequality respect to $p_{t-s}(x-y)dyds$ and an application of  Lemma \ref{TechLemma1} yield that for any $p\ge2$,  $t\in [0,T]$,
  \begin{align}\label{u-g-n-l-estimate}
    & \|u_\ep^{g, n+1}(t,\cdot)-u_\ep^{g, n}(t,\cdot)\|_{L^p_{\lambda}(\RR)}^p\notag\\
     =&\, \int_{\RR}  \left| u_\ep^{g, n+1}(t,x)-u_\ep^{g, n}(t,x)\right|^p  \lambda(x)dx\notag \\
     \leq&\,C_{\e,N}\int_{\RR}\left|\int_{0}^{t}\int_{\RR}p_{t-s}(x-y)\left|u_\ep^{g, n}(s,y)-u_\ep^{g, n-1}(s,y)\right|^2dyds\right|^{\frac{p}{2}}\lambda(x)dx\notag\\
     \leq&\,C_{\ep,N,p,T}\int_{\RR}\left[\int_{0}^{t}\int_{\RR}p_{t-s}(x-y)\left|u_\ep^{g, n}(s,y)-u_\ep^{g, n-1}(s,y)\right|^pdyds\right]\lambda(x)dx\\
       {=}&\,C_{\ep,N,p,T}\int_{0}^{t}\int_{\RR}\frac{1}{\lambda(y)}\int_{\RR}p_{t-s}(x-y)\lambda(x)dx\left|u_\ep^{g, n}(s,y)-u_\ep^{g, n-1}(s,y)\right|^p\lambda(y)dyds\notag\\
     \le &\,C_{\ep,N,p,T} \int_{0}^{t}\|u_\ep^{g, n}(s,\cdot)-u_\ep^{g, n-1}(s,\cdot)\|_{L^p_{\lambda}( \RR)}^p ds\notag\\
     \le &\,C_{\ep,N,p,T}^n  \frac{T^n}{n!} \sup_{s\in[0,T]}\|u_\ep^{g, 1}(s,\cdot)-u_\ep^{g, 0}(s,\cdot)\|_{L^p_{\lambda}( \RR)}^p.\notag
   \end{align}
 Then, (\ref{u-g-n-l-estimate}) implies that
 $$
   \sup_{n\geq1} \sup\limits_{t\in[0,T]}\|u_{\ep}^{g, n}(t,\cdot)\|_ {L^p_{\lambda}(\RR)}<\infty, \quad \hbox{ for each $\ep>0$},
  $$
  and that   {$\{u_{\ep}^{g, n}(t,\cdot)\}_{n\geq1}$  is a Cauchy sequence in $L^p_{\lambda}(\RR)$ for any $t\in[0,T]$. Hence, for any fixed $t\in [0,T]$,  there exists  $u_{\ep}^{g}(t,\cdot)\in L^p_{\lambda}( \RR)$  such that  
  \begin{align}\label{eq u gn}
  u_{\ep}^{g, n}(t,\cdot)\rightarrow u_{\ep}^{g}(t,\cdot)  \text{ in }   L^p_{\lambda}(\RR),   \text{ as  } n\rightarrow\infty
  \end{align}
  }

{\bf Step 2.} In this step, we will give   { a quantitative estimate  for  $\|u_\ep^{g,n}(t,\cdot )\|_{L^p_{\lambda}(\RR)}$ for any fixed $\e>0$}.
 By the Cauchy-Schwarz inequality, \eqref{eq H product} and \eqref{eq product H e}, we have
   \begin{align}\label{uepLp}
      |u_\ep^{g,n+1}(t,x)|^p
     \lesssim\,& 1
       +\lc\int_{0}^{t}\left\| p_{t-s}(x-\cdot)\sigma(s,\cdot,u_\ep^{g,n}(s,\cdot))\right\|_{\HH_\e}^2ds\rc^{\frac p2}  \nonumber\\
        \lesssim\,&1
       +\lc\int_{0}^{t}\left\| p_{t-s}(x-\cdot)\sigma(s,\cdot,u_\ep^{g,n}(s,\cdot))\right\|_{\HH}^2ds\rc^{\frac p2}  \nonumber\\
     \simeq\,& 1
       + \bigg(\int_{0}^{t}\int_{\RR^2}  \Big|p_{t-s}(x-y-h)\sigma(s,y+h,u_\ep^{g,n}(s,y+h))\\
       &\qquad\qquad\qquad-p_{t-s}(x-y)\sigma(s,y,u_\ep^{g,n}(s,y))\Big|^2 \cdot|h|^{2H-2}dhdyds\bigg)^{\frac p2} \nonumber\\
       \lesssim\,& 1+\cA_1(t,x)+\cA_2(t,x)+\cA_3(t,x), \nonumber
   \end{align}
where
   \begin{align*}
     \cA_1(t,x) :=\,&\bigg(\int_{0}^{t}\int_{\RR^2} p^2_{t-s}(x-y-h) \cdot  \big|\sigma(s,y+h,u_\ep^{g,n}(s,y+h))
    -\sigma(s,y,u_\ep^{g,n}(s,y+h))\big|^2\\ &\qquad\quad\cdot|h|^{2H-2}dhdyds\bigg)^{\frac p2} ;\\
     \cA_2(t,x) :=\,&\bigg(\int_{0}^{t}\int_{\RR^2} p^2_{t-s}(x-y-h) \cdot  \big|\sigma(s,y,u_\ep^{g,n}(s,y+h))-\sigma(s,y,u_\ep^{g,n}(s,y))\big|^2\\
     &\qquad\quad\cdot|h|^{2H-2}dhdyds\bigg)^{\frac p2};\\
     \cA_3(t,x) :=\,&\bigg( \int_{0}^{t}\int_{\RR^2}|D_{t-s}(x-y,h)|^2\cdot\left|\sigma(s,y,u_\ep^{g,n}(s,y))\right|^2
             \cdot|h|^{2H-2} dhdyds\bigg)^{\frac p2}.
   \end{align*}
   If $|h|>1$, then    we have by   \eqref{lineargrowthuniformlycond}
   \begin{equation}\label{DuSigamy}
  \begin{split}
   &\left|\sigma(s,y+h,u_\ep^{g,n}(s,y))-\sigma(s,y,u_\ep^{g,n}(s,y)) \right|^2\\
   \lesssim\,&\left|\sigma(s,y+h,u_\ep^{g,n}(s,y)) \right|^2+\left|\sigma(s,y,u_\ep^{g,n}(s,y)) \right|^2\\
    \lesssim\,&1+\left|u_\ep^{g,n}(s,y) \right|^2.
      \end{split}
     \end{equation}
If $|h|\leq1$, then  by   \eqref{DuySigam} and \eqref{DuxSigam}  there exists some $\zeta\in(0,1)$ such that
   \begin{equation}\label{DuxSigamy}
  \begin{split}
   &\left|\sigma(s,y+h,u_\ep^{g,n}(s,y))-\sigma(s,y,u_\ep^{g,n}(s,y)) \right|^2\\
   \lesssim\,&\left|\sigma(s,y+h,0)-\sigma(s,y,0)\right|^2+\left|\int_0^{u_\ep^{g,n}(s,y)}[\sigma'_{\xi}(s,y+h,\xi)-\sigma'_{\xi}(s,y,\xi)d\xi \right|^2\\
    \lesssim\,&\left|\sigma'_x(s,y+\zeta h,0)\right|^2\cdot |h|^2+\left|u_\ep^{g,n}(s,y) \right|^2\cdot |h|^2\\
    \lesssim\,& \left(1+\left|u_\ep^{g,n}(s,y) \right|^2\right)\cdot |h|^2.
   \end{split}
     \end{equation}
   By a change of variable, \eqref{DuSigamy} and \eqref{DuxSigamy}, we have
   \begin{align*}
     \cA_1(t,x) \lesssim \bigg(\int_{0}^{t}\int_{\RR}p^2_{t-s}(x-y)\cdot \left(1+  \big|u_\ep^{g,n}(s,y)\big|^2\right)dyds\bigg)^{\frac p2} .
    \end{align*}
   By a change of variable, Minkowski's inequality,   Lemma \ref{TechLemma1}  and Jensen's inequality with respect to
 \begin{equation}\label{scale}
   p^2_{t-s}(y)dy\simeq (t-s)^{-\frac{1}{2}}p_{\frac{t-s}{2}}(y)dy,
   \end{equation}
   we have
   \begin{equation}\label{D1Bdd}
   \begin{split}
 &\lc\int_\RR \cA_1(t,x) \lambda(x)dx \rc^{\frac 2p}\\
 \lesssim\,&
 \int_{0}^{t}\int_{\RR}p^2_{t-s}(y)\cdot \left(\int_{\RR}\left(1+\big|u_\ep^{g,n}(s,x)\big|^p\right)\lambda(x-y)dx \right)^{\frac 2p}dyds\\
 \lesssim\,&\int_{0}^{t}(t-s)^{-\frac 12}\cdot \left(\int_{\RR^2}p_{\frac{t-s}{2}}(y)\left(1+\big|u_\ep^{g,n}(s,x)\big|^p\right)\lambda(x-y)dxdy \right)^{\frac 2p}ds\\
 \lesssim\,&\int_{0}^{t} (t-s)^{-\frac 12}\cdot \left(1+\|u_{\ep}^{g,n}(s,\cdot)\|^2_{L^p_{\lambda}(\RR)}\right)ds.
   \end{split}
   \end{equation}
 By \eqref{Lipschitzianuniformlycondu}, a change of variable, Minkowski's inequality, Jensen's inequality and Lemma \ref{TechLemma1}, we have
   \begin{align}\label{D1Bdd2}
 &\lc\int_\RR \cA_2(t,x) \lambda(x)dx \rc^{\frac 2p}\notag\\
 \lesssim\,&\lc\int_\RR \bigg(\int_{0}^{t}\int_{\RR} p^2_{t-s}(x-y)\cdot  \big|u_\ep^{g,n}(s,y+h)-u_\ep^{g,n}(s,y)\big|^2\cdot|h|^{2H-2}dhdyds\bigg)^{\frac p2}\lambda(x)dx \rc^{\frac 2p}\notag\\
 \lesssim\,&\int_{0}^{t}\int_{\RR^2}p^2_{t-s}(y)\cdot \left(\int_{\RR}\big|u_\ep^{g,n}(s,x+h)-u_\ep^{g,n}(s,x)\big|^p\lambda(x-y)dx \right)^{\frac 2p}\cdot|h|^{2H-2}dhdyds\\
 \lesssim\,&\int_{0}^{t}\int_{\RR}(t-s)^{-\frac 12}\cdot \left(\int_{\RR^2}p_{\frac{t-s}{2}}(y)\big|u_\ep^{g,n}(s,x+h)-u_\ep^{g,n}(s,x)\big|^p\lambda(x-y)dydx \right)^{\frac 2p}\cdot|h|^{2H-2}dhds\notag\\
 \lesssim \,&\int_{0}^{t}(t-s)^{-\frac 12}\cdot \lt[\cN^*_{\frac 12-H,\,p}u_\ep^{g,n}(s)\rt]^2 ds.\notag
   \end{align}
   By \eqref{lineargrowthuniformlycond}, a change of variable, Minkowski's inequality, Jensen's inequality   {with respect to $(t-s)^{1-H}\left|D_{t-s}(y,h)\right|^2 |h|^{2H-2}dydh$} and Lemma \ref{TechLemma5}, we have
   \begin{align}\label{D1Bdd3}
 &\lc\int_\RR \cA_3(t,x) \lambda(x)dx \rc^{\frac 2p} \notag\\
 \lesssim\,&\int_{0}^{t}\int_{\RR^2}|D_{t-s}(y,h)|^2\cdot \left(1+\int_{\RR}\big|u_\ep^{g,n}(s,x)\big|^p\lambda(x-y)dx \right)^{\frac 2p}\cdot|h|^{2H-2}dhdyds\notag\\
 \lesssim\,&\int_{0}^{t}(t-s)^{H-1}\cdot \Bigg(\int_{\RR^3}(t-s)^{1-H}\cdot \big|D_{t-s}(y,h)\big|^2\\
 &\, \ \,\,\,\, \quad\,\,\,\, \quad\,\,\,\, \quad \,\,\,\, \quad \cdot \left(1+\big|u_\ep^{g,n}(s,x)\big|^p\right)\lambda(x-y)\cdot|h|^{2H-2}dhdydx \Bigg)^{\frac 2p}ds \notag\\
 \lesssim\,&\int_{0}^{t}(t-s)^{H-1}\cdot \left(1+\|u_{\ep}^{g,n}(s,\cdot)\|^2_{L^p_{\lambda}( \RR)} \right) ds.\notag
   \end{align}
Thus, \eqref{uepLp}, \eqref{D1Bdd}, \eqref{D1Bdd2} and  \eqref{D1Bdd3} yield that
\begin{align}\label{ueptcdot-350}
  \|u_{\ep}^{g,n+1}(t,\cdot)\|_{L^p_{\lambda}( \RR)}^2
   =\, &\lc \int_\RR   |u_\ep^{g,n+1}(t,x)|^p \lambda(x) dx \rc^{\frac 2p}\nonumber\\
 \ls\,&  1+ \int_{0}^{t} \left((t-s)^{H-1}+(t-s)^{-\frac{1}{2}}\right)\cdot \|u_{\ep}^{g,n}(s,\cdot)\|^2_{L^p_{\lambda}( \RR)}  ds \\
 &+\int_{0}^{t}(t-s)^{-1/2} \lt[\cN^*_{\frac 12-H,\,p}u_\ep^{g,n}(s)\rt]^2  ds.\nonumber
     \label{e.4.42}
\end{align}

{\bf Step 3.} This step is devoted to estimating $\mathcal N^*_{\frac12-H,\, p} u_{\ep}^{g,n}(t)$.
Using the Cauchy-Schwarz inequality and \eqref{eq H compare}, we have
   \begin{align*}
      & \left|u_\ep^{g,n+1}(t,x)-u_\ep^{g,n+1}(t,x+h)\right|^p  \\
     \lesssim\,&
      \lc\int_{0}^{t}\left\|D_{t-s}(x-\cdot,h)\sigma(s,\cdot,u_\ep^{g,n}(s,\cdot))\right\|_{\HH_\e}^2ds\rc^{\frac p2}  \nonumber\\
     \lesssim\,&
      \lc\int_{0}^{t}\left\|D_{t-s}(x-\cdot,h)\sigma(s,\cdot,u_\ep^{g,n}(s,\cdot))\right\|_{\HH}^2ds\rc^{\frac p2}  \nonumber\\
     \simeq\,&
       \bigg( \int_{0}^{t}\int_{\RR^2} \Big|D_{t-s}(x-y-z,h)\sigma(s,y+z,u_\ep^{g,n}(s,y+z)) \\
       &\qquad\qquad -D_{t-s}(x-z,h)\sigma(s,z,u_\ep^{g,n}(s,z))\Big|^2\cdot|y|^{2H-2} dzdyds\bigg)^{\frac p2} \\
     \lesssim\,& \cI_1(t,x,h)+\cI_2(t,x,h)+\cI_3(t,x,h),
   \end{align*}
   where
   \begin{align*}
     \cI_1(t,x,h):=\,&  \bigg(\int_{0}^{t}\int_{\RR^2}\left| D_{t-s}(x-y-z,h)\right|^2 \\
     &\quad \qquad\quad\cdot\big|\sigma(s,y+z,u_\ep^{g,n}(s,y+z))-\sigma(s,z,u_\ep^{g,n}(s,y+z)) \big|^2\cdot |y|^{2H-2}dzdyds\bigg)^{\frac p2};\\
     \cI_2(t,x,h):=\,&  \bigg(\int_{0}^{t}\int_{\RR^2}\left| D_{t-s}(x-y-z,h)\right|^2 \\
     &\quad\qquad\quad\cdot\big|\sigma(s,z,u_\ep^{g,n}(s,y+z))-\sigma(s,z,u_\ep^{g,n}(s,z)) \big|^2\cdot |y|^{2H-2}dzdyds\bigg)^{\frac p2};\\
     \cI_3(t,x,h):=\,& \bigg(\int_{0}^{t}\int_{\RR^2} \left|\Box_{t-s}(x-z,y,h)\right|^2\cdot  \big|\sigma(s,z,u_\ep^{g,n}(s,z))\big|^2\cdot |y|^{2H-2} dzdyds\bigg)^{\frac p2}.
   \end{align*}
   Therefore, by \eqref{N*pNorm}, we have
   \begin{equation}\label{sum-I-1-3-445}
   \begin{split}
     \lk\cN^*_{\frac 12-H,\,p}u_\ep^{g,n+1}(t)\rk^2=\,&\int_{\RR}\lc \int_\RR \left|u_\ep^{g,n+1}(t,x)-u_\ep^{g,n+1}(t,x+h)\right|^p \lambda(x) dx\rc^{\frac 2p} \cdot |h|^{2H-2} dh\\
     \lesssim\,&\sum_{j=1}^{3}\int_{\RR}\lc \int_\RR \cI_j(t,x,h) \lambda(x) dx\rc^{\frac 2p} \cdot |h|^{2H-2} dh.
   \end{split}
   \end{equation}
 
For the first term of (\ref{sum-I-1-3-445}), by \eqref{DuSigamy}, \eqref{DuxSigamy} and a change of variable, we have
\begin{align}\label{I11Bdd}
  \cI_1(t,x,h)\lesssim\,&  \bigg(\int_{0}^{t}\int_{\RR}\int_{|y|>1}\left| D_{t-s}(x-y-z,h)\right|^2\cdot
     \left(1+\big|u_\ep^{g,n}(s,y+z)\big|^2\right)\cdot |y|^{2H-2}dzdyds\bigg)^{\frac p2}\notag\\
     &+\,\bigg(\int_{0}^{t}\int_{\RR}\int_{|y|\leq1}\left| D_{t-s}(x-y-z,h)\right|^2\cdot
     \left( 1+\big|u_\ep^{g,n}(s,y+z)\big|^2\right) \cdot |y|^{2H}dzdyds\bigg)^{\frac p2}\notag\\
     =\,&\bigg(\int_{0}^{t}\int_{\RR}\int_{|y|>1}\left| D_{t-s}(x-z,h)\right|^2\cdot  \left(1+\big|u_\ep^{g,n}(s,z)\big|^2\right) \cdot |y|^{2H-2}dzdyds\bigg)^{\frac p2}\\
     &+ \bigg(\int_{0}^{t}\int_{\RR}\int_{|y|\leq1}\left| D_{t-s}(x-z,h)\right|^2 \cdot \left(1+\big|u_\ep^{g,n}(s,z)\big|^2\right)\cdot |y|^{2H}dzdyds\bigg)^{\frac p2}\notag\\
      \lesssim\,&\bigg(\int_{0}^{t}\int_{\RR}\left| p_{t-s}(z-h)-p_{t-s}(z)\right|^2\cdot \left(1+\big|u_\ep^{g,n}(s,x+z)\big|^2\right) dzds\bigg)^{\frac p2}.\notag
  \end{align}
By (\ref{I11Bdd}), a change of variable, Minkowski's inequality, Lemma \ref{TechLemma5} and Jensen's inequality with respect to $(t-s)^{1-H}\left|D_{t-s}(z,h)\right|^2 |h|^{2H-2}dzdh$, we have
   \begin{align}\label{I1Bdd}
      &\int_{\RR}\lt| \int_\RR \cI_1(t,x,h) \lambda(x) dx\rt|^{\frac 2p}\cdot |h|^{2H-2} dh \notag\\
      \lesssim\,&\int_{\RR}\bigg| \int_\RR \bigg(\int_{0}^{t}\int_{\RR}\left| D_{t-s}(z,h)\right|^2 \left(1+\big|u_\ep^{g,n}(s,x) \big|^2\right) dzds\bigg)^{\frac p2}\lambda(x-z) dx\bigg|^{\frac 2p} \cdot|h|^{2H-2} dh \notag\\
      \lesssim\,&\int_{0}^{t}\int_{\RR^2}\left| D_{t-s}(z,h)\right|^2\cdot |h|^{2H-2}\left(\int_\RR\left(1+\big|u_\ep^{g,n}(s,x) \big|^p\right) \lambda(x-z) dx\right)^{\frac{2}{p}}
      dzdhds\\
      \lesssim\,& \int_{0}^{t}(t-s)^{H-1}\cdot \bigg(\int_{\RR^3} (t-s)^{1-H}\cdot \left|D_{t-s}(z,h)\right|^2\cdot |h|^{2H-2}\cdot \left(1+\big|u_\ep^{g,n}(s,x)\big|^p \right) \lambda(x-z) dxdzdh\bigg)^{\frac 2p}ds \notag\\
      \lesssim\,& \int_{0}^{t} (t-s)^{H-1}\cdot \left(1+\left\|u_\ep^{g,n}(s,\cdot)\right\|_{L^p_{\lambda}( \RR)}^{2}\right) ds.\notag
   \end{align}
   For the second term of (\ref{sum-I-1-3-445}), by \eqref{Lipschitzianuniformlycondu} and a change of variable, we have  
\begin{align}\label{I21Bdd}
  \cI_2(t,x,h)\lesssim\,&  \bigg(\int_{0}^{t}\int_{\RR^2}\left| D_{t-s}(x-y-z,h)\right|^2\cdot\big|u_\ep^{g,n}(s,y+z)-u_\ep^{g,n}(s,z) \big|^2\cdot |y|^{2H-2}dzdyds\bigg)^{\frac p2}\notag\\
    =\,&\bigg(\int_{0}^{t}\int_{\RR^2}\left| D_{t-s}(x-z,h)\right|^2\cdot\big|u_\ep^{g,n}(s,z)-u_\ep^{g,n}(s,z-y) \big|^2\cdot |y|^{2H-2}dzdyds\bigg)^{\frac p2}\notag\\
    =\,&\bigg(\int_{0}^{t}\int_{\RR^2}\left| D_{t-s}(x-z,h)\right|^2\cdot\big|u_\ep^{g,n}(s,y+z)-u_\ep^{g,n}(s,z) \big|^2\cdot |y|^{2H-2}dzdyds\bigg)^{\frac p2}\\
      =\,&\bigg(\int_{0}^{t}\int_{\RR^2}\left| p_{t-s}(z-h)-p_{t-s}(z)\right|^2 \notag\\
     & \, \,\, \qquad\quad\cdot\big|u_\ep^{g,n}(s,x+y+z)-u_\ep^{g,n}(s,x+z) \big|^2\cdot|y|^{2H-2}dzdyds\bigg)^{\frac p2}.\notag
  \end{align}
By (\ref{I21Bdd}), a change of variable, Minkowski's inequality, Jensen's inequality and Lemma \ref{TechLemma5}, we have
   \begin{align}\label{I2Bdd}
      &\int_{\RR}\lt| \int_\RR \cI_2(t,x,h) \lambda(x) dx\rt|^{\frac 2p} \cdot|h|^{2H-2} dh \notag\\
      \lesssim\,&\int_{\RR}\bigg| \int_\RR \bigg(\int_{0}^{t}\int_{\RR^2}\left| D_{t-s}(z,h)\right|^2\cdot \big|u_\ep^{g,n}(s,x+y)-u_\ep^{g,n}(s,x) \big|^2\cdot |y|^{2H-2}dzdyds\bigg)^{\frac p2}\notag\\
     &\qquad\cdot \lambda(x-z) dx\bigg|^{\frac 2p} \cdot|h|^{2H-2} dh \notag\\
      \lesssim\,&\int_{0}^{t}\int_{\RR^3}\left| D_{t-s}(z,h)\right|^2 \cdot |y|^{2H-2}\cdot |h|^{2H-2}\\
      &\qquad\cdot \left(\int_\RR\big|u_\ep^{g,n}(s,x+y)-u_\ep^{g,n}(s,x) \big|^p\lambda(x-z) dx\right)^{\frac{2}{p}}
      dzdydhds\notag\\
      \lesssim\,& \int_{0}^{t}\int_\RR (t-s)^{H-1} \cdot \bigg(\int_{\RR^3} (t-s)^{1-H}\cdot \left|D_{t-s}(z,h)\right|^2\cdot |h|^{2H-2} \notag\\
      &\qquad\cdot \big|u_\ep^{g,n}(s,x+y)-u_\ep^{g,n}(s,x)\big|^p  \lambda(x-z) dxdzdh\bigg)^{\frac 2p}\cdot |y|^{2H-2} dyds \notag\\
      \lesssim\,& \int_{0}^{t} (t-s)^{H-1}\cdot \lk\cN^*_{\frac 12-H,\,p}u_\ep^{g,n}(s)\rk^2 ds.\notag
     \end{align}
For the last term of (\ref{sum-I-1-3-445}), by \eqref{lineargrowthuniformlycond} and a change of variable, we have
\begin{align*}
     \cI_3(t,x,h)\lesssim\,&  \bigg(\int_{0}^{t}\int_{\RR^2} \left|\Box_{t-s}(x-z,y,h)\right|^2\cdot \left(1+ \big|u_\ep^{g,n}(s,z)\big|^2\right)\cdot|y|^{2H-2} dzdyds\bigg)^{\frac p2}\\
     =\,&\bigg(\int_{0}^{t}\int_{\RR^2} \left|\Box_{t-s}(z,y,h)\right|^2\cdot \left(1+ \big|u_\ep^{g,n}(s,x-z)\big|^2\right)\cdot|y|^{2H-2} dzdyds\bigg)^{\frac p2}\\
    \leq\,&\lc \int_{0}^{t}\int_{\RR^2}|\Box_{t-s}(-z,y,h)|^2\cdot|u_\ep^{g,n}(s,x+z)-u_\ep^{g,n}(s,x)|^2 \cdot|y|^{2H-2}dzdyds\rc^{\frac p2}\\
  &+\lc \int_{0}^{t}\int_{\RR^2} |\Box_{t-s}(-z,y,h)|^2\cdot \left( 1+\big|u_\ep^{g,n}(s,x)\big|^2\right)\cdot |y|^{2H-2}dzdyds\rc^{\frac p2} \\
     =:\,&\cI_{31}(t,x,h) +\cI_{32}(t,x,h).
   \end{align*}
By a change of variable, Minkowski's inequality and  Lemma \ref{TechLemma4}, we have
   \begin{align}\label{I31Bdd}
      &\int_{\RR}\lt| \int_\RR \cI_{31}(t,x,h) \lambda(x) dx\rt|^{\frac 2p} \cdot|h|^{2H-2} dh\notag \\
      \lesssim\,&\int_{0}^{t}\int_{\RR^3}|\Box_{t-s}(z,y,h)|^2\cdot\left(\int_\RR|u_\ep^{g,n}(s,x+z)-u_\ep^{g,n}(s,x)|^p\lambda(x)dx\right)^{\frac{2}{p}}\\
      &\, \qquad \cdot  |h|^{2H-2} \cdot |y|^{2H-2}dhdzdyds\notag \\
       \lesssim\,&\int_{0}^{t} (t-s)^{H-1}\cdot \lk\cN^*_{\frac 12-H,\,p}u_\ep^{g,n}(s)\rk^2 ds.\notag
     \end{align}
   By a change of variable, Minkowski's inequality and Lemma \ref{TechLemma2},    we have
   \begin{equation}\label{I322Bdd}
   \begin{split}
      &\int_{\RR}\lt| \int_\RR \cI_{32}(t,x,h) \lambda(x) dx\rt|^{\frac 2p} \cdot|h|^{2H-2} dh\\
     \lesssim\,&\int^t_0\int_{\RR^3}|\Box_{t-s}(z,y,h)|^2\cdot\left(1+\int_\RR\left|u_\ep^{g,n}(s,x)\right|^p\lambda(x)dx\right)^{\frac{2}{p}}\cdot |h|^{2H-2}  \cdot|y|^{2H-2}dhdzdyds\\
      \lesssim\,&\int_{0}^{t} (t-s)^{2H-\frac 32}\cdot \lc 1+\|u_\ep^{g,n}(s,\cdot)\|_{L^p_{\lambda}( \RR)}^{2}\rc ds.
   \end{split}
   \end{equation}
   Thus, by \eqref{sum-I-1-3-445}, \eqref{I1Bdd}, \eqref{I2Bdd}, \eqref{I31Bdd} and \eqref{I322Bdd}, we have
   \begin{equation}\label{e.4.48}
   \begin{split}
     \lk\cN^*_{\frac 12-H,\,p}u_\ep^{g,n+1}(t)\rk^2
     \lesssim\,&1+\int_{0}^{t} \left((t-s)^{2H-\frac 32}+(t-s)^{H-1}\right)\cdot  \left\|u_\ep^{g,n}(s,\cdot)\right\|_{L^p_{\lambda}( \RR)}^{2} ds\,\\
     &+\int_{0}^{t}(t-s)^{H-1}\cdot   \lk\cN^*_{\frac 12-H,\,p}u_\ep^{g,n}(s)\rk^2 ds.
   \end{split}
   \end{equation}

  {{\bf Step 4.}}
Let
$$
    \Psi^n_\ep(t):=\|u_\ep^{g,n}(t,\cdot)\|_{L^p_{\lambda}(\RR)}^{2}+\lk\cN^*_{\frac 12-H,\,p}u_\ep^{g,n}(t)\rk^2.
  $$
Putting  \eqref{ueptcdot-350} and  \eqref{e.4.48} together, there exists a constant $C_{T,p,H,N}>0$ such that
   \begin{align*}
     \Psi^{n+1}_\ep(t)\leq C_{T,p,H,N}\left(1 +\int_{0}^{t} (t-s)^{2H-\frac 32}\Psi^n_\ep(s) ds\right).
  \end{align*}
By   { the extension of Gr\"onwall's lemma (\cite[lemma 15]{Dalang1999})}, we have
\begin{align}\label{eq cg}
    \sup_{n\geq1}\sup_{t\in[0,T]}\Psi^n_\ep(t)\leq\, C_g,
\end{align}
where $C_g$ is a constant independent of   $\e$.

For any fixed $\e>0$,  
  by  \eqref{eq u gn} and \eqref{eq cg}, we have that, for any $t\in [0,T]$,
\begin{equation}\label{eq F11}
\begin{split}
\|u^{g}_\ep(t,\cdot)\|_{L^p_\lambda(\RR)}
= \,\lim_{n\rightarrow\infty}\left(\int_\RR\left[|u^{g,n}_\ep(t,x)|^p\right]\lambda(x)dx\right)^{\frac{1}{p}}
\leq \,  C_g.
\end{split}
\end{equation}
  For any fixed $t$ and $h$, we have  by     \eqref{eq u gn},
  $$\|u^{g,n}_\e(t,\cdot+h)-u^{g,n}_\e(t,\cdot)\|_{L^p_{\lambda}(\RR)}\rightarrow\|u^{g}_\e(t,\cdot+h)-u^{g}_\e(t,\cdot)\|_{L^p_{\lambda}(\RR)},  \text{ as } n\rightarrow\infty.$$
By Fatou's lemma and \eqref{eq cg}, we have
\begin{align}\label{eq F12}
       \cN^*_{\frac 12-H,\,p}u_\ep^{g}(t)=&\int_{\RR}\|u^g_\e(t,\cdot+h)-u^g_\e(t,\cdot)\|^2_{L^p_{\lambda}(\RR)}|h|^{2H-2}dh\\\notag
      \leq&\,\liminf_{n\rightarrow\infty}
      \int_{\RR}\|u^{g,n}_\e(t,\cdot+h)-u^{g,n}_\e(t,\cdot)\|^2_{L^p_{\lambda}(\RR)}|h|^{2H-2}dh\leq C_g.
  \end{align}
By \eqref{eq F11} and \eqref{eq F12}, we obtain that
  $$\sup_{\e>0}\|u_\ep^{g}\|_{Z^p_{\lambda,T}}\leq\sup_{\e>0}\sup_{t\in[0,T]}\|u^{g}_\ep(t,\cdot)\|_{L^p_\lambda(\RR)}+
\sup_{\e>0}\sup_{t\in[0,T]}\cN^*_{\frac 12-H,\,p}u_\ep^{g}(t)<\infty.$$

  The proof is complete.
\end{proof}

 \begin{lemma}\label{TimeSpaceRegBdd}
Let $u_\ep^g$ be the approximate  process defined by \eqref{eq u h e}.
   \begin{itemize}
     \item[(i).] If {\bf $p>\frac{6}{4H-1}$}, then there exists a constant $C_{T,p,H,N}>0$ such that
     \begin{equation}\label{359359}
      \sup\limits_{t\in[0,T],\,x\in \RR}\lambda^{\frac{1}{p}}(x)\cdot  \cN_{\frac 12-H}u_\ep^g(t,x) \leq C_{T,p,H,N}\left(1+\|u_\ep^g\|_{Z_{\lambda,T}^p}\right).
     \end{equation}
     \item[(ii).] If    $p>\frac{3}{H}$  and  
     $0<\gamma<\frac H2-\frac{3}{2p}$, then there exists a constant $C_{T ,p,H,N,\gamma}>0$ such that
     \begin{equation}\label{Hol1}
       \sup\limits_{\substack{t,\,t+h\in[0,T], \\ x\in \RR}}\lambda^{\frac{1}{p}}(x) \cdot   |u_\ep^g(t+h,x)-u_\ep^g(t,x)| \leq C_{T ,p,H,N,\gamma}|h|^{\gamma}\cdot\left(1+\|u_\ep^g\|_{Z_{\lambda,T}^p}\right).
     \end{equation}
     \item[(iii).] If  $p>\frac{3}{H}$  and
   $0<\gamma<H-\frac{3}{p}$, then there exists a constant $C_{T ,p,H,N,\gamma}>0$ such that
     \begin{equation}\label{Lemma-4.3-iii}
       \sup\limits_{\substack{t\in[0,T], \\ x,\,y\in \RR}}\frac{|u_\ep^g(t,x)-u_\ep^g(t,y)|}{\lambda^{-\frac{1}{p}}(x)+\lambda^{-\frac{1}{p}}(y)}  \leq C_{T,p,H,N,\gamma}|x-y|^{\gamma}\cdot\left(1+\|u_\ep^g\|_{Z_{\lambda,T}^p}\right).
     \end{equation}
   \end{itemize}
 \end{lemma}
\begin{proof}
   (i). By \eqref{eq u h e}, we have
\begin{align*} u^g_{\ep}(t,x+h)-u^g_{\ep} (t,x)
=\,&\int_0^t \left\langle  p_{t-s}(x+h-\cdot)\sigma(s,\cdot,u_{\ep}^g(s,\cdot)),g(s,\cdot) \right\rangle_{\mathcal H_{\ep}} ds\\
&-\int_0^t \left\langle  p_{t-s}(x-\cdot)\sigma(s,\cdot,u_{\ep}^g(s,\cdot)),g(s,\cdot) \right\rangle_{\mathcal H_{\ep}} ds\notag\\
=:&\,\Phi(t,x+h)-\Phi(t,x).\notag
\end{align*}
  Applying \eqref{eq product H e} and the Fubini theorem,
we have that, for any $\alpha\in(0,1)$,
 \begin{align}\label{461-transform}
  \Phi(t,x)=&\,\int_{0}^{t}\langle p_{t-s}(x-\cdot)\sigma(s,\cdot,u_\ep^{g}(s,\cdot)), g(s,\cdot) \rangle_{\HH_\e}ds \notag \\
  \simeq&\,\int_{0}^{t}\int_{\RR^2}p_{t-s}(x-y)\sigma(s,y,u_\ep^{g}(s,y))g(s,\tilde{y})f_{\e}(y-\tilde{y})dyd\tilde{y}ds \notag \\
  =&\,\frac{\sin(\pi\alpha)}{\pi}\int_{0}^{t}\int_{\RR^2}p_{t-s}(x-y)\left[\int^t_s(t-r)^{\alpha-1}(r-s)^{-\alpha}dr\right]\notag\\
  &\qquad\qquad\qquad\cdot\sigma(s,y,u_\ep^{g}(s,y))g(s,\tilde{y})f_{\e}(y-\tilde{y})dyd\tilde{y}ds \notag \\
  =&\,\frac{\sin(\pi\alpha)}{\pi}\int_{0}^{t}\int_{\RR^2}\left(\int_{\RR}p_{t-r}(x-z)p_{r-s}(z-y)dz\right)\left(\int^t_s(t-r)^{\alpha-1}(r-s)^{-\alpha}dr\right)\notag\\
  &\qquad\qquad\qquad\cdot\sigma(s,y,u_\ep^{g}(s,y))g(s,\tilde{y})f_{\e}(y-\tilde{y})dyd\tilde{y}ds\notag \\
  \simeq&\,\int_{0}^{t}\int_{\RR}(t-r)^{\alpha-1}\cdot p_{t-r}(x-z)\\
  &\qquad\cdot\left(\int^r_0(r-s)^{-\alpha}\cdot \int_{\RR^2} p_{r-s}(z-y)\sigma(s,y,u_\ep^{g}(s,y)) g(s,\tilde{y})f_{\e}(y-\tilde{y})dyd\tilde{y}ds\right)dzdr\notag \\
  \simeq&\,\int_{0}^{t}\int_{\RR}(t-r)^{\alpha-1} \cdot p_{t-r}(x-z)\notag\\
  &\qquad\cdot\left(\int^r_0(r-s)^{-\alpha}\langle p_{r-s}(z-\cdot)\sigma(s,\cdot,u_\ep^{g}(s,\cdot)),g(s,\cdot)\rangle_{\HH_\e} ds\right)dzdr\notag \\
  =&\,\int_{0}^{t}\int_{\RR}(t-r)^{\alpha-1}\cdot p_{t-r}(x-z)J_\alpha(r,z)dzdr, \notag
 \end{align}
where
\begin{align*}
J_\alpha(r,z):=\int^r_0(r-s)^{-\alpha}\langle p_{r-s}(z-\cdot)\sigma(s,\cdot,u_\ep^{g}(s,\cdot)),g(s,\cdot)\rangle_{\HH_\e} ds.
\end{align*}
Set
 \begin{align*}
 \Delta_hJ_\alpha(t,x):=J_\alpha(t,x+h)-J_\alpha(t,x).
 \end{align*}
Applying a change of variable, we have
 \begin{align*}
 &\Phi(t,x+h)-\Phi(t,x)\notag \\
 \simeq&\,\int_{0}^{t}\int_{\RR}(t-r)^{\alpha-1}\cdot p_{t-r}(x+h-z)\cdot J_\alpha(r,z)dzdr-\int_{0}^{t}\int_{\RR}(t-r)^{\alpha-1} \cdot p_{t-r}(x-z)\cdot J_\alpha(r,z)dzdr\notag \\
  =&\,\int_{0}^{t}\int_{\RR}(t-r)^{\alpha-1}\cdot p_{t-r}(x-z)\cdot J_\alpha(r,z+h)dzdr-\int_{0}^{t}\int_{\RR}(t-r)^{\alpha-1}\cdot p_{t-r}(x-z)\cdot J_\alpha(r,z)dzdr\notag \\
 =&\,\int_{0}^{t}\int_{\RR}(t-r)^{\alpha-1}\cdot p_{t-r}(x-z)\cdot \Delta_h J_\alpha(r,z)dzdr.
 \end{align*}
 Invoking Minkowski's inequality and   {H\"older's} inequality with $\frac{1}{p}+\frac{1}{q}=1$, we have
 \begin{align}\label{468}
  &\int_\RR\left|\Phi(t,x+h)-\Phi(t,x)\right|^2\cdot|h|^{2H-2}dh\notag \\
 \simeq&\,\int_\RR\left|\int_{0}^{t} \int_\RR(t-r)^{\alpha-1} \cdot p_{t-r}(x-z)\Delta_hJ_\alpha(r,z)dzdr\right|^2\cdot|h|^{2H-2}dh\notag \\
  \lesssim&\,\left(\int_{0}^{t} \int_\RR(t-r)^{\alpha-1}\cdot  p_{t-r}(x-z)\left[\int_\RR|\Delta_hJ_\alpha(r,z)|^2\cdot|h|^{2H-2}dh\right]^{\frac{1}{2}}dzdr\right)^2\notag \\
   \lesssim&\,\left(\int_{0}^{t} \int_\RR(t-r)^{q(\alpha-1)} \cdot p^q_{t-r}(x-z)\lambda^{-\frac{q}{p}}(z)dzdr\right)^{\frac{2}{q}} \\
   &\cdot\left(\int_{0}^{T}\int_\RR\left[\int_\RR|\Delta_hJ_\alpha(r,z)|^2\cdot|h|^{2H-2}dh\right]^{\frac{p}{2}}\lambda(z)dzdr\right)^{\frac{2}{p}}\notag \\
   \lesssim&\,\lambda^{-\frac{2}{p}}(x)\left(\int_{0}^{t} (t-r)^{q(\alpha-\frac{3}{2}+\frac{1}{2q})}dr\right)^{\frac{2}{q}}\notag\\ &\cdot\left(\int_{0}^{T}\int_\RR\left[\int_\RR\left|\Delta_hJ_\alpha(r,z)\right|^2\cdot|h|^{2H-2}dh\right]^{\frac{p}{2}}\lambda(z)dzdr\right)^{\frac{2}{p}},\notag
 \end{align}
  where in the last step  we have used   Lemma \ref{TechLemma1}  and the following fact:   {\begin{align}\label{pqt-s12simeq}
  p^q_{t-r}(x-z)\simeq(t-r)^{\frac{1-q}{2}}p_{\frac{t-r}{q}}(x-z).
  \end{align}}
    If   $q(\alpha-\frac{3}{2}+\frac{1}{2q})>-1$, namely,    {if}  $\alpha>\frac{3}{2p}$,
 then
 \begin{equation}\label{phi}
 \begin{split}
  &\sup_{t\in[0,T],x\in\RR}\lambda^{\frac{1}{p}}(x)\left(\int_\RR|\Phi(t,x+h)-\Phi(t,x)|^2\cdot|h|^{2H-2}dh\right)^{\frac{1}{2}}\\
  \lesssim&\,\left(\int_{0}^{T}\int_\RR\left[\int_\RR|\Delta_hJ_\alpha(r,z)|^2\cdot|h|^{2H-2}dh\right]^{\frac{p}{2}}\lambda(z)dzdr\right)^{\frac{1}{p}}.
 \end{split}
 \end{equation}
Thus, to prove part $(i)$ we only need to prove that there exists some positive constant $C$, independent of $r\in[0,T]$, such that for $\frac{3}{2p}<\alpha<1$,
\begin{align}\label{4-21}
 \int_\RR\left[\int_\RR|\Delta_hJ_\alpha(r,z)|^2\cdot|h|^{2H-2}dh\right]^{\frac{p}{2}}\lambda(z)dz
\leq C\left(1+\left\|u_\ep^g\right\|^p_{Z_{\lambda,T}^p}\right).
 \end{align}

Now, it remains to show (\ref{4-21}).
Recall $D_t(x,h)$ defined by (\ref{TechDt}). Applying the Cauchy-Schwarz inequality,
 (\ref{eq H product}), (\ref{eq product H e}) and \eqref{g-interal-bound}, we have
\begin{align*}
  \Delta_hJ_\alpha(r,z)  \leq&\,\int^r_0(r-s)^{-\alpha}\cdot \big\|D_{r-s}(z-\cdot,h)\sigma(s,\cdot,u_\ep^{g}(s,\cdot))\big\|_{\HH_\e}\cdot\big\|g(s,\cdot)\big\|_{\HH_\e}ds\notag \\
  \lesssim&\,\left(\int^r_0(r-s)^{-2\alpha}\cdot \big\|D_{r-s}(z-\cdot,h)\sigma(s,\cdot,u_\ep^{g}(s,\cdot))\big\|^2_{\HH_\e}ds\right)^{\frac{1}{2}}
  \cdot\left(\int_0^r\big\|g(s,\cdot)\big\|_{\HH_\e}ds\right)^{\frac{1}{2}}\notag \\
  \lesssim&\,\left(\int^r_0(r-s)^{-2\alpha}\cdot\big\|D_{r-s}(z-\cdot,h)\sigma(s,\cdot,u_\ep^{g}(s,\cdot))\big\|^2_{\HH_\e}ds\right)^{\frac{1}{2}}\notag \\
  \lesssim&\,\bigg(\int^r_0\int_{\RR^2}(r-s)^{-2\alpha}\cdot \big|D_{r-s}(z-y-l,h)\sigma(s,y+l,u_\ep^{g}(s,y+l))\\
  &\qquad\qquad\qquad\qquad\quad-D_{r-s}(z-y,h)\sigma(s,y,u_\ep^{g}(s,y))\big|^2 \cdot|l|^{2H-2}dldyds\bigg)^{\frac{1}{2}}.
 \end{align*}

 Set
 \begin{align*}
 \mathcal{J}_1(r,z,h):=&\,\Bigg(\int^r_0\int_{\RR^2}(r-s)^{-2\alpha}\cdot\big|D_{r-s}(z-y-l,h)\big|^2\\
 &\qquad\qquad\cdot\big|\sigma(s,y+l,u_\ep^{g}(s,y+l))-\sigma(s,y,u_\ep^{g}(s,y+l))\big|^2\cdot|l|^{2H-2}dldyds\Bigg)^{\frac{p}{2}};\\
 \mathcal{J}_2(r,z,h):=&\,\Bigg(\int^r_0\int_{\RR^2}(r-s)^{-2\alpha}\cdot \big|D_{r-s}(z-y-l,h)\big|^2\\
 &\qquad\qquad\cdot\big|\sigma(s,y,u_\ep^{g}(s,y+l))-\sigma(s,y,u_\ep^{g}(s,y))\big|^2\cdot|l|^{2H-2}dldyds\Bigg)^{\frac{p}{2}};\\
 \mathcal{J}_3(r,z,h):=&\,\left(\int^r_0\int_{\RR^2}(r-s)^{-2\alpha}\cdot \big|\Box_{r-s}(z-y,l,h)\big|^2\cdot\big|\sigma(s,y,u_\ep^{g}(s,y))\big|^2
 \cdot|l|^{2H-2}dldyds\right)^{\frac{p}{2}}.
 \end{align*}
 By Minkowski's inequality, we have
 \begin{align}\label{mathcalJxiaoyudengyu}
\mathcal{J}
  \lesssim&\sum_{i=1}^{3}\left[\int_{\RR}\left(\int_{\RR}\mathcal{J}_i(r,z,h)\lambda(z)dz\right)^{\frac{2}{p}}|h|^{2H-2}dh\right]^{\frac{p}{2}}.
 \end{align}
 By the same technique as that in Step 2 of the proof for Lemma \ref{UniBExist}, we have
 \begin{align}
 \int_{\RR}\left(\int_{\RR}\mathcal{J}_1(r,z,h)\lambda(z)dz\right)^{\frac{2}{p}}|h|^{2H-2}dh
 \lesssim& \int_{0}^{r} (t-s)^{-2\alpha+H-1}\cdot \left(1+\|u_\ep^g(s,\cdot)\|_{L^p_{\lambda}( \RR)}^{2}\right) ds;\label{BoundMathcalJ1}\\
 \int_{\RR}\left(\int_{\RR}\mathcal{J}_2(r,z,h)\lambda(z)dz\right)^{\frac{2}{p}}|h|^{2H-2}dh
 \lesssim&\int_{0}^{r} (t-s)^{-2\alpha+H-1}\cdot \lk\cN^*_{\frac 12-H,\,p}u_\ep^g(s)\rk^2ds;\label{BoundMathcalJ2}
  \end{align}
   \begin{equation}\label{BoundMathcalJ3}
   \begin{split}
 \int_{\RR}\left(\int_{\RR}\mathcal{J}_3(r,z,h)\lambda(z)dz\right)^{\frac{2}{p}}|h|^{2H-2}dh
 \lesssim &\,\int_{0}^{r} (t-s)^{-2\alpha+2H-\frac 32}\cdot \left(1+ \|u_\ep^g(s,\cdot)\|_{L^p_{\lambda}( \RR)}^{2}\right) ds\\
 &+\int_{0}^{r} (t-s)^{-2\alpha+H-1} \cdot \lk\cN^*_{\frac 12-H,\,p}u_\ep^g(s)\rk^2ds.
  \end{split}
  \end{equation}
 By \eqref{mathcalJxiaoyudengyu}, \eqref{BoundMathcalJ1}, \eqref{BoundMathcalJ2}, \eqref{BoundMathcalJ3} and by using the same method as that in  the proof of (4.27) in \cite{HW2019}, we have
   \begin{equation}\label{bound-I}
   \begin{split}
  &\int_\RR\left[\int_\RR|\Delta_hJ_\alpha(r,z)|^2\cdot|h|^{2H-2}dh\right]^{\frac{p}{2}}\lambda(z)dz\\
  \le\,  & C\left(\int^r_0(r-s)^{-2\alpha+2H-\frac{3}{2}}+(r-s)^{-2\alpha+H-1}ds\right)^{\frac{p}{2}}\left(1+\|u_\ep^g\|^p_{Z_{\lambda,T}^p}\right).
  \end{split}
   \end{equation}
If $-2\alpha+2H-\frac{3}{2}>-1$ and $-2\alpha+H-1>-1$, namely, if $\alpha<H-\frac{1}{4}$,  then we see that
(\ref{4-21}) follows from  (\ref{bound-I}). This condition on $\alpha$ is combined with $\alpha>\frac{3}{2p}$ to become $\frac{3}{2p}<\alpha<H-\frac{1}{4}$. Therefore, we have proved that for any  $p>\frac{6}{4H-1}$,    \eqref{359359} holds.

   (ii). By  \eqref{eq u h e} and \eqref{461-transform}, we have
\begin{equation}\label{3.405.7}
\begin{split}
 &\Phi(t+h,x)-\Phi(t,x)  \\
\simeq&\,\int_{0}^{t}\int_{\RR}(t+h-r)^{\alpha-1} p_{t+h-r}(x-z)J_\alpha(r,z)dzdr \\
&-\int_{0}^{t}\int_{\RR}(t-r)^{\alpha-1} p_{t-r}(x-z)J_\alpha(r,z)dzdr \\
&+\int_{t}^{t+h}\int_{\RR}(t+h-r)^{\alpha-1} p_{t+h-r}(x-z)J_\alpha(r,z)dzdr \\
=&\,\int_{0}^{t}\int_{\RR}(t+h-r)^{\alpha-1}\left( p_{t+h-r}(x-z)-p_{t-r}(x-z)\right)J_\alpha(r,z)dzdr \\
&+\int_{0}^{t}\int_{\RR}\left((t+h-r)^{\alpha-1}-(t-r)^{\alpha-1}\right) p_{t-r}(x-z)J_\alpha(r,z)dzdr \\
&+\int_{t}^{t+h}\int_{\RR}(t+h-r)^{\alpha-1} p_{t+h-r}(x-z)J_\alpha(r,z)dzdr \\
=:&\,\mathcal{K}_1(t,h,x)+\mathcal{K}_2(t,h,x)+\mathcal{K}_3(t,h,x).
\end{split}
\end{equation}

In the following, we give estimates for $\mathcal{K}_i(t,h,x)$, $i=1,2,3$, respectively.
 By H\"older's inequality, we have
\begin{equation}\label{J-1-First}
\begin{split}
\mathcal{K}_1(t,h,x)=&\,\int_{0}^{t}\int_{\RR}(t+h-r)^{\alpha-1}\left( p_{t+h-r}(x-z)-p_{t-r}(x-z)\right)\lambda^{-\frac{1}{p}}(z)J_\alpha(r,z)\lambda^{\frac{1}{p}}(z)dzdr\\
\leq&\,\left(\int_{0}^{t}\int_{\RR}(t+h-r)^{q(\alpha-1)}\big| p_{t+h-r}(x-z)-p_{t-r}(x-z)\big|^q\lambda^{-\frac{q}{p}}(z)dzdr\right)^{\frac{1}{q}}\\
&\, \, \cdot\left(\int_{0}^{t}\int_{\RR}\big|J_\alpha(r,z)\big|^p\lambda(z)dzdr\right)^{\frac{1}{p}}\\
=:&\,\left|\mathcal{K}_{11}(t,h,x)\right|^{\frac{1}{q}}\cdot \left(\int^t_0\big\|J_\alpha(r,\cdot)\big\|
^p_{L^p_\lambda(\mathbb{R})}dr\right)^{\frac{1}{p}}.
\end{split}
\end{equation}
Applying the Cauchy-Schwarz inequality,
 (\ref{eq H product}), (\ref{eq product H e}) and \eqref{g-interal-bound}, we have
\begin{equation}\label{J-12-First}
\begin{split}
&\big\|J_\alpha(r,\cdot)\big\|^p_{L^p_\lambda(\mathbb{R})}\\
=&\,\int_{\RR}\bigg|\int^r_0(r-s)^{-\alpha}\cdot \langle p_{r-s}(z-\cdot)\sigma(s,\cdot,u^g_{\ep}(s,\cdot)),g(s,\cdot) \rangle _{\HH_\e}ds\bigg|^p\lambda(z)dz\\
\lesssim &\,\int_{\RR}\bigg|\int^r_0(r-s)^{-2\alpha}\cdot \left\|p_{r-s}(z-\cdot)\sigma(s,\cdot,u^g_{\ep}(s,\cdot))\right\|^2_{\HH}ds\bigg|^{\frac{p}{2}}
\lambda(z)dz\\
\simeq&\,\int_{\RR}\bigg|\int^r_0\int_{\RR^2}(r-s)^{-2\alpha}\cdot \Big(p_{r-s}(z-y-l)\sigma(s,y+l,u^g_{\ep}(s,y+l))\\
&\qquad\qquad\qquad\qquad\qquad-p_{r-s}(z-y)\sigma(s,y,u^g_{\ep}(s,y))\Big)^2
\cdot|l|^{2H-2}dydlds\bigg|^{\frac{p}{2}}\lambda(z)dz\\
\lesssim &\,\int_{\RR}[\mathcal{B}_1(r,z)+\mathcal{B}_2(r,z)+\mathcal{B}_3(r,z)]\lambda(z)dz,
\end{split}
\end{equation}
where
\begin{align*}
\mathcal{B}_1(r,z):=&\,\Bigg(\int^r_0\int_{\RR^2}(r-s)^{-2\alpha}\cdot p^2_{r-s}(z-y-l) \\
&\qquad \quad \cdot \big|\sigma(s,y+l,u^g_{\ep}(s,y+l)   -\sigma(s,y,u^g_{\ep}(s,y+l))\big|^2\cdot|l|^{2H-2}dydlds\Bigg)^{\frac{p}{2}};\\
\mathcal{B}_2(r,z):=&\,\Bigg(\int^r_0\int_{\RR^2}(r-s)^{-2\alpha}\cdot p^2_{r-s}(z-y-l)\\
&\qquad \quad \cdot \big|\sigma(s,y,u^g_{\ep}(s,y+l)   -\sigma(s,y,u^g_{\ep}(s,y))\big|^2\cdot|l|^{2H-2}dydlds\Bigg)^{\frac{p}{2}};\\
\mathcal{B}_3(r,z):=&\,\left(\int^r_0\int_{\RR^2}(r-s)^{-2\alpha}\cdot \big|D_{r-s}(z-y,l)\big|^2\cdot \big|\sigma(s,y,u^g_{\ep}(s,y))\big|^2\cdot |l|^{2H-2}dydlds\right)^{\frac{p}{2}}.
\end{align*}
By using the same technique as that in Step 3 of the proof for Lemma \ref{UniBExist}, we have
\begin{align}\label{D-1-bound}
\int_{\RR}\mathcal{B}_1(r,z)\lambda(z)dz
\lesssim &\,
\left(\int^r_0(r-s)^{-2\alpha-\frac{1}{2}}\cdot\left(1+\left\|u^g_{\ep}(s,\cdot)\right\|_{L^p_\lambda(\mathbb{R})}^2\right) ds\right)^{\frac{p}{2}};\\
  \label{D-2-bound}
\int_{\RR}\mathcal{B}_2(r,z)\lambda(z)dz
\lesssim&\,\left(\int^r_0(r-s)^{-2\alpha-\frac{1}{2}}\cdot \left(N^*_{\frac 12-H,\,p}u^g_{\ep}(s)\right)^2 ds\right)^{\frac{p}{2}};\\
  \label{D-3-bound}
\int_{\RR}\mathcal{B}_3(r,z)\lambda(z)dz
\lesssim &\,
\left(\int^r_0(r-s)^{-2\alpha+H-1}\cdot\left(1+\left\|u^g_{\ep}(s,\cdot)\right\|_{L^p_\lambda(\mathbb{R})}^2\right)ds\right)^{\frac{p}{2}}.
\end{align}
Putting \eqref{J-12-First}, \eqref{D-1-bound},  \eqref{D-2-bound} and \eqref{D-3-bound} together, we have
\begin{align}\label{{J-12-Second}}
\big\|J_\alpha(r,\cdot)\big\|^p_{L^p_\lambda(\mathbb{R})}
\lesssim& \,\left(\int^r_0(r-s)^{-2\alpha-\frac{1}{2}}+(r-s)^{-2\alpha+H-1}ds\right)^{\frac{p}{2}}\left(1+\|u^g_{\ep}\big\|^p_{\mathcal{Z}^p_{\lambda,T}}\right).
\end{align}
  {If $-2\alpha-\frac{1}{2}>-1$ and $-2\alpha+H-1>-1$, namely,} if $\alpha<\frac{H}{2}$, then we have
\begin{align}\label{J-12-Third}
\big\|J_\alpha(r,\cdot)\big\|^p_{L^p_\lambda(\mathbb{R})}
\lesssim\,1+\|u^g_{\ep}\big\|^p_{{Z}^p_{\lambda,T}}.
\end{align}
By     {\eqref{pqt-s12simeq}}, \eqref{6-155} and Lemma \ref{TechLemma1},  for any $\gamma\in(0,1)$, we have
\begin{align*}
\big|\mathcal{K}_{11}(t,h,x)\big|\lesssim &\,|h|^{q\gamma}\int_{0}^{t}\int_{\RR}(t-r)^{q(\alpha-1-\gamma)+\frac{1-q}{2}}p_{\frac{2(t+h-r)}{q\gamma}}(x-z)\lambda^{-\frac{q}{p}}(z)dzdr\notag\\
&+|h|^{q\gamma}\int_{0}^{t}\int_{\RR}(t-r)^{q(\alpha-1-\gamma)+\frac{1-q}{2}}p_{\frac{2(t-r)}{q\gamma}}(x-z)\lambda^{-\frac{q}{p}}(z)dzdr\notag\\
\lesssim &\,|h|^{q\gamma}\lambda^{-\frac{q}{p}}(x)\int_{0}^{t}(t-r)^{q(\alpha-1-\gamma)+\frac{1-q}{2}}dr.
\end{align*}
  {Hence, if $\alpha<\frac{H}{2}$ and $q(\alpha-1-\gamma)+\frac{1-q}{2}>-1$,  namely,} if   $0<\gamma<\frac{H}{2}-\frac{3}{2p}$, then we have
\begin{align}\label{J-11-First}
\big|\mathcal{K}_{11}(t,h,x)\big|^{\frac{1}{q}}\lesssim &\, |h|^{\gamma}\lambda^{-\frac{1}{p}}(x).
\end{align}
 If $0<\gamma<\frac{H}{2}-\frac{3}{2p}$, then  by putting \eqref{J-1-First},  \eqref{J-12-Third}   and \eqref{J-11-First} together, we have
\begin{align}\label{J-1-Third}
       \sup\limits_{\substack{t,\,t+h\in[0,T], \\ x\in \RR}}\lambda^{\frac{1}{p}}(x) \cdot  \big|\mathcal{K}_1(t,h,x)\big| \leq \, C_{T,p,H,N,\gamma}\cdot|h|^{\gamma}\cdot\left(1+\|u_\ep^g\|_{Z_{\lambda,T}^p}\right).
\end{align}
Applying H\"older's inequality, \eqref{pqt-s12simeq} and Lemma \ref{TechLemma1}, we have
\begin{equation}\label{J-2-First}
\begin{split}
\mathcal{K}_2(t,h,x)
\leq&\,\left(\int^t_0\int_{\RR}\big|(t+h-r)^{\alpha-1}-(t-r)^{\alpha-1}\big|^q\cdot p^q_{t-r}(x-z)\lambda^{-\frac{q}{p}}(z)dzdr\right)^{\frac{1}{q}}\\
&\,\, \,\,\cdot\left(\int^t_0\int_{\RR}\big|J_\alpha(r,z)\big|^p\lambda(z)dzdr\right)^{\frac{1}{p}}\\
\lesssim&\,\left(\int^t_0\big|(t+h-r)^{\alpha-1}-(t-r)^{\alpha-1}\big|^q(t-r)^{\frac{1-q}{2}}\lambda^{-\frac{q}{p}}(x)dr\right)^{\frac{1}{q}}\\
&\,\,\,\, \cdot\left(\int^T_0\big\|J_\alpha(r,\cdot)\big\|^p_{L^p_\lambda(\mathbb{R})}dr\right)^{\frac{1}{p}}.
\end{split}
\end{equation}
 By   \eqref{J-12-Third}  and \eqref{6-154}, for any  $\gamma\in(0,1)$, if $\alpha<\frac{H}{2}$, then we have
\begin{align*}
       &\sup\limits_{\substack{t,\,t+h\in[0,T], \\ x\in \RR}}\lambda^{\frac{1}{p}}(x) \cdot  \big|\mathcal{K}_2(t,h,x)\big|\\ \lesssim\, & |h|^{\gamma}\cdot\sup_{t\in[0,T]}\left(\int^t_0(t-r)^{q(\alpha-1-\gamma)+\frac{1-q}{2}}dr\right)^{\frac{1}{q}}\cdot \left(1+\|u_\ep^g\|_{Z_{\lambda,T}^p}\right).\notag
\end{align*}
  {If $\alpha<\frac{H}{2}$ and $q(\alpha-1-\gamma)+\frac{1-q}{2}>-1$, namely,} if   $0<\gamma<\frac{H}{2}-\frac{3}{2p}$, then there exists a positive constant $C_{T,p,H,N,\gamma}$ such that
\begin{align}\label{J-2-Third}
       \sup\limits_{\substack{t,\,t+h\in[0,T], \\ x\in \RR}}\lambda^{\frac{1}{p}}(x) \cdot  \big|\mathcal{K}_2(t,h,x)\big| \leq C_{T,p,H,N,\gamma} |h|^{\gamma}\cdot \left(1+\|u_\ep^g\|_{Z_{\lambda,T}^p}\right).
\end{align}
By H\"older's inequality and Lemma \ref{TechLemma1}, we have
\begin{equation}\label{J-3-First}
\begin{split}
\mathcal{K}_3(t,h,x)
\leq&\,\left(\int^{t+h}_t(t+h-r)^{q(\alpha-1)}(t+h-r)^{\frac{1-q}{2}}\lambda^{-\frac{q}{p}}(x)dr\right)^{\frac{1}{q}}\\
&\,\,\, \cdot\left(\int^T_0\big\|J_\alpha(r,\cdot)\big\|^p_{L^p_\lambda(\mathbb{R})}dr\right)^{\frac{1}{p}}\\
\simeq&\,\lambda^{-\frac{1}{p}}(x)\cdot |h|^{\alpha-\frac{3}{2p}}\cdot\left(\int^T_0\big\|J_\alpha(r,\cdot)\big\|^p_{L^p_\lambda(\mathbb{R})}dr\right)^{\frac{1}{p}}.
\end{split}
\end{equation}
  {Let $\gamma=\alpha-\frac{3}{2p}$. If $\frac{3}{2p}<\alpha<\frac{H}{2}$, namely, if  $0<\gamma<\frac{H}{2}-\frac{3}{2p}$, then \eqref{J-12-Third} yields that}
\begin{equation}\label{J-3-Second}
\begin{split}
       \sup\limits_{\substack{t,\,t+h\in[0,T], \\ x\in \RR}}\lambda^{\frac{1}{p}}(x)\cdot   \big|\mathcal{K}_3(t,h,x)\big|
        \leq C_{T,p,H,N,\gamma}|h|^{\gamma}\cdot\left(1+\|u_\ep^g\|_{Z_{\lambda,T}^p}\right).
\end{split}
\end{equation}
If $p>\frac{3}{H}$ and $0<\gamma<\frac{H}{2}-\frac{3}{2p}$, then by putting \eqref{3.405.7}, \eqref{J-1-Third}, \eqref{J-2-Third} and \eqref{J-3-Second} together, we have
\begin{align}\label{phi-estimate}
 \sup\limits_{\substack{t,\,t+h\in[0,T], \\ x\in \RR}}\lambda^{\frac{1}{p}}(x)\cdot   \big|\Phi(t+h,x)-\Phi(t,x)\big| \leq \, C_{T,p,H,N,\gamma} |h|^{\gamma}\cdot\left(1+\|u_\ep^g\|_{Z_{\lambda,T}^p}\right).
\end{align}

(iii). Notice that
\begin{align*}
&|\Phi(t,x)-\Phi(t,y)|\\
\simeq&\,\int_{0}^{t}\int_{\RR}(t-r)^{\alpha-1}\cdot \left(p_{t-r}(x-z)-p_{t-r}(y-z)\right)J_\alpha(r,z)dzdr\notag \\
\lesssim&\, \left(\int_{0}^{t}\int_{\RR}(t-r)^{q(\alpha-1)}\cdot \big|p_{t-r}(x-z)-p_{t-r}(y-z)\big|^q\lambda^{-\frac{q}{p}}(z)dzdr\right)^{\frac{1}{q}} \cdot\left(\int_{0}^{t}\int_{\RR}\big|J_\alpha(r,z)\big|^p\lambda(z)dzdr\right)^{\frac{1}{p}}\notag \\
=&\,\left(\int_{0}^{t}\int_{\RR}(t-r)^{q(\alpha-1)}\cdot \big|p_{t-r}(x-z)-p_{t-r}(y-z)\big|^q\lambda^{-\frac{q}{p}}(z)dzdr\right)^{\frac{1}{q}} \cdot\left(\int^t_0\big\|J_\alpha(r,\cdot)\big\|^p_{L^p_\lambda(\mathbb{R})}dr\right)^{\frac{1}{p}}.
 \end{align*}
For any $\gamma\in(0,1)$, by using the same method as that in the proof of    (4.35) in \cite{HW2019} and (\ref{J-12-Third}), we have
\begin{equation}\label{481-x-y}
\begin{split}
&\big|\Phi(t,x)-\Phi(t,y)\big|\\
\lesssim&\,\big|x-y\big|^\gamma\cdot\left[\lambda^{-\frac{1}{p}}(x)+\lambda^{-\frac{1}{p}}(y)\right]
\cdot\left(\int_{0}^{t}(t-r)^{(\alpha q-\frac{3q}{2}+\frac{1}{2})-\frac{\gamma q}{2}}dr\right)^{\frac{1}{q}} \cdot\left(1+\|u^g_{\ep}\big\|_{{Z}^p_{\lambda,T}}\right).
 \end{split}
 \end{equation}
   {If $\alpha<\frac{H}{2}$ and $(\alpha q-\frac{3q}{2}+\frac{1}{2})-\frac{\gamma q}{2}>-1$, namely,} if $p>\frac{3}{H}$ and $0<\gamma<H-\frac{3}{p}$, then we get the desired result \eqref{Lemma-4.3-iii}.

 The proof is complete.
 \end{proof}
  {
\begin{lemma}\label{lemma3.3limits} Assume that     $u_{n}\in Z^p_{\lambda,T}$, $n\geq 1$,  for some $p\ge2$. If  $u_n\rightarrow u$ in $(\cC([0, T]\times\RR),d_{\cC})$ as $n\rightarrow \infty$, then we have the following results:
\begin{itemize}
     \item[(i).]     $u$ is also in $Z^p_{\lambda,T}$;
  \item[(ii).] for any fixed $t\in [0,T]$,
  \begin{align}\label{eq H1 finite}
  \int^t_0\int_{\mathbb R}\|p_{t-s}(x-\cdot)\sigma(s,\cdot,u(s,\cdot))\|^2_{\cH} \lambda(x) dsdx<\infty.
  \end{align}
  Hence, for almost all  $x\in \mathbb R$,
    \begin{align}\label{eq H2 finite}
     \int^t_0\|p_{t-s}(x-\cdot)\sigma(s,\cdot,u(s,\cdot))\|^2_{\cH}  ds<\infty.
      \end{align}
  \end{itemize}
\end{lemma}}
\begin{proof}
  {(i). This result  is due to  \cite[Lemma 4.6]{HW2019}  by replacing $\|\cdot\|_{L^p_{\lambda}(\Omega)\times\RR}$ by $\|\cdot\|_{L^p_{\lambda}(\RR)}$, the proof is omitted here.}

  {(ii). }Recall  $D_t(x,h)$ given by  \eqref{TechDt}.
By \eqref{eq H product} and a change of variable, we have
\begin{equation}\label{p_tsxcdottriangleg_n_mathcalH}
\begin{split}
&\|p_{t-s}(x-\cdot)\sigma(s,\cdot,u(s,\cdot))\|^2_{\mathcal{H}}\\
=&\,c_{2, H}\int_{\RR^2}\bigg[p_{t-s}(x-y-h)\sigma(s,y+h,u(s,y+h))
-p_{t-s}(x-y)\sigma(s,y,u(s,y))\bigg]^2\cdot|h|^{2H-2}dhdy\\
\lesssim&\,\int_{\RR^2}p^2_{t-s}(x-y)\left|\sigma(s,y,u(s,y+h))-\sigma(s,y,u(s,y))\right|^2\cdot|h|^{2H-2}dhdy\\
 &+\int_{\RR^2}\left|D_{t-s}(x-y,h)\right|^2\left|\sigma(s,y+h,u(s,y))\right|^2\cdot|h|^{2H-2}dhdy\\
 &+\int_{\RR^2}p^2_{t-s}(x-y)\left|\sigma(s,y+h,u(s,y))-\sigma(s,y,u(s,y))\right|^2\cdot|h|^{2H-2}dhdy\\
  =:&\,F_1(t,s,x)+F_2(t,s,x)+F_3(t,s,x).
\end{split}
\end{equation}
By \eqref{Lipschitzianuniformlycondu}, \eqref{scale}, a change of variable and  Lemma \ref{TechLemma1}, we have that, for any fixed $t\in [0,T]$,
\begin{equation}\label{eq F1}
    \begin{split}
&\int^t_0\int_{\RR}F_1(t,s,x)\lambda(x)dxds\\
\lesssim&\,\int^t_0\int_{\RR^3}p^2_{t-s}(x-y)\cdot\big|u(s, y+h)-u(s, y)\big|^2\cdot|h|^{2H-2}\lambda(x)dhdydxds\\
\simeq&\,\int^t_0\int_{\RR^2}(t-s)^{-\frac{1}{2}}\cdot\left(\int_{\RR}p_{\frac{t-s}{2}}(y-x)\lambda(x)dx\right)\\
&\qquad\quad \cdot\big|u(s, y+h)-u(s, y)\big|^2\cdot|h|^{2H-2}dhdyds\\
\lesssim&\,\int^t_0\int_{\RR^2}(t-s)^{-\frac{1}{2}}\lambda(y)\big|u(s, y+h)-u(s, y)\big|^2\cdot|h|^{2H-2}dhdyds\\
\lesssim&\,\int^t_0(t-s)^{-\frac{1}{2}}ds\cdot\sup_{s\in[0,T]}\left[\mathcal{N}^*_{\frac{1}{2}-H,2}u(s)\right]^2
<\infty.
\end{split}
\end{equation}

By \eqref{lineargrowthuniformlycond}, 
a change of variable and Lemma \ref{TechLemma2}, we have  that, for any fixed $t\in [0,T]$,
\begin{align}\label{eq F2}
&\int^t_0\int_{\RR}F_2(t,s,x)\lambda(x)dxds\notag\\
\lesssim&\,\int^t_0\int_{\RR^3}|D_{t-s}(x-y,h)|^2\cdot\left(1+\big|u(s, y)\big|^2\right)\cdot|h|^{2H-2}\lambda(x)dhdydxds\notag\\
=&\,\int^t_0\int_{\RR}\left(\int_{\RR^2}|D_{t-s}(y,h)|^2\cdot|h|^{2H-2}\lambda(x-y)dhdy\right)\cdot\left(1+\big|u(s, x)\big|^2\right) dxds\\
\lesssim&\,  \int^t_0\int_{\RR}(t-s)^{H-1}\left(1+\big|u(s, x)\big|^2\right)\lambda(x)dxds\notag\\
\leq&\,\int^t_0(t-s)^{H-1}ds\cdot \sup_{s\in[0,T]}\left(1+\|u(s,\cdot)\|^2_{L^2_\lambda(\RR)}\right)<\infty.\notag
\end{align}

By \eqref{DuSigamy}, \eqref{DuxSigamy}, \eqref{scale},  
a change of variable and  Lemma \ref{TechLemma1}, we have  that, for any fixed $t\in [0,T]$,
\begin{equation}\label{eq F3}
\begin{split}
&\int^t_0\int_{\RR}F_3(t,s,x)\lambda(x)dxds\\
\lesssim&\,\int_0^t\int_{\RR^3}p^2_{t-s}(x-y)\left(1+\left|u(s,y)\right|^2\right)\cdot\left(\textbf{1}_{\{|h|\leq1\}}|h|^2+\textbf{1}_{\{|h|>1\}}\right)|h|^{2H-2}\lambda(x)dhdydxds\\
\lesssim&\,\int_0^t\int_{\RR^2}p^2_{t-s}(x-y)\left(1+\left|u(s,y)\right|^2\right)\lambda(x)dydxds \\
\simeq&\,\int_0^t\int_{\RR}(t-s)^{-\frac{1}{2}}\cdot\left(\int_{\RR}p_{\frac{t-s}{2}}(y-x)\lambda(x)dx\right)\cdot\left(1+\left|u(s,y)\right|^2\right)dyds\\
 \lesssim&\,\int_0^t\int_{\RR}(t-s)^{-\frac{1}{2}}\lambda(y)\left(1+\left|u(s,y)\right|^2\right)dyds\\
 \leq&\,\int^t_0(t-s)^{-\frac{1}{2}}ds\cdot \sup_{s\in[0,T]}\left(1+\left\|u(s,\cdot)\right\|^2_{L^2_\lambda(\RR)}\right)<\infty.
 \end{split}
 \end{equation}
Putting  \eqref{p_tsxcdottriangleg_n_mathcalH}-\eqref{eq F3} together, we obtain \eqref{eq H1 finite}. In particular,
 \eqref{eq H2 finite} holds  for almost all $x\in \mathbb R$.

 The proof is complete.
\end{proof}

 \begin{proof}[Proof of Proposition \ref{thm solu skeleton}] ({\bf Existence})
  {By Lemma \ref{TimeSpaceRegBdd} (ii) and (iii),
we have that for any   $R>0$ and $\gamma>0$, there exists $\theta>0$ such that
\begin{align*}
       \max\limits_{\substack{|t-s|+|x-y|\leq \theta, \\ 0\leq t,s \leq T;\,    -R\le x,y \le R }}|u^{g}_{\e}(t,x)-u^{g}_{\e}(s,y)|  \leq \gamma.
     \end{align*}
Combining this with the fact of $u^{g}_{\e}(0,\cdot)\equiv1$, we know that the family $\{u^{g}_{\e}\}_{\e>0}$ is relatively compact on the space $({\mathcal C}([0,T]\times\RR), d_{\mathcal C})$ by the Arzel\`a-Ascoli theorem.}
Thus, there is a subsequence $\ep_n\downarrow 0$ such that    $u^g_{\ep_n}$ converges to a function $ u^g$ in $(\cC([0, T]\times\RR),d_{\cC})$. 

By using the Cauchy-Schwarz inequality, \eqref{Lipschitzianuniformlycondu},    {Lemma \ref{lemma3.3limits} (ii)} and the dominated convergence theorem, we have
  \begin{align*}
  u^g_{\ep_n}(t,x)=&\,1+\int_0^t\langle p_{t-s}(x-\cdot)\left(\sigma(s,\cdot,u^g_{\ep_n}(s,\cdot))-\sigma(s,\cdot, u^g(s,\cdot))\right),g(s,\cdot)\rangle_{\cH_{\ep_n}} ds\notag\\
 &+\int_0^t\langle p_{t-s}(x-\cdot)\sigma(s,\cdot,u^g(s,\cdot)),g(s,\cdot)\rangle_{\cH_{\ep_n}} ds\notag\\
  &\,\longrightarrow  1+\int_0^t\langle p_{t-s}(x-\cdot)\sigma(s,\cdot, u^g(s,\cdot)),g(s,\cdot)\rangle_{\cH} ds\ \ \ \text{as } \ep_n\downarrow 0.
  \end{align*}
 By the uniqueness of the limit of  $\{u^g_{\ep_n}\}_{n\ge1}$, we know that $u^g$ satisfies Eq.\,\eqref{eq skeleton}.
  {By Proposition  \ref{thm solu skeleton} and Lemma \ref{lemma3.3limits} (i),
 we have  that, for any $p\ge 2$,
\begin{align}\label{ugZlamdanormal}
\|u^g\|_{Z_{\lambda,T}^p}<\infty.
\end{align}}

{(\bf Uniqueness)}
Let $u^g$ and $v^g$ be two solutions of \eqref{eq skeleton}.
By using the same technique as that  in the proof of Lemma \ref{TimeSpaceRegBdd} (i),
we have that, for any   $p>\frac{6}{4H-1}$,
\begin{equation}\label{Nug}
      \sup\limits_{t\in[0,T],\, x\in \RR}\lambda^{\frac{1}{p}}(x)\cdot  \cN_{\frac 12-H}u^g(t,x) <\infty;
\end{equation}
\begin{equation}\label{Nvg}
      \sup\limits_{t\in[0,T],\, x\in \RR}\lambda^{\frac{1}{p}}(x) \cdot \cN_{\frac 12-H}v^g(t,x) <\infty.
\end{equation}
  Denote that
  \begin{align*}
    {S_1(t)=\,\int_{\mathbb{R}}|u^g(t,x)-v^g(t,x)|^2\lambda(x)dx},
  \end{align*}
and
   \begin{align*}
    {S_2(t)=\,\int_{\RR^2}\left|u^g(t,x)-v^g(t,x)-u^g(t,x+h)+v^g(t,x+h)\right|^2\cdot|h|^{2H-2}\lambda(x)dhdx.}
  \end{align*}
    {According to  \eqref{LPlambdaRNorm}, \eqref{N*pNorm}  and \eqref{ugZlamdanormal}, we know that
  \begin{align}\label{eq S12}
  \displaystyle\sup_{t\in[0,T]}S_1(t)<\infty, \,\,\, \displaystyle\sup_{t\in[0,T]}S_2(t)<\infty.
  \end{align}}
  Recall $D_{t}(x,h)$ defined by \eqref{TechDt} and denote that $$\triangle(t,x,y):=\sigma(t,x,u^g(t,y))-\sigma(t,x,v^g(t,y)).$$
 Since $\int_0^T \|g(s,\cdot)\|_{\HH}^2ds<\infty$, by using the Cauchy-Schwarz inequality, \eqref{eq H product}  and a change of variable, we have that, for any $t\in [0,T]$,
   { \begin{equation}\label{4.84}
 \begin{split}
  &\int_{\RR}|u^g(t,x)-v^g(t,x)|^2\lambda(x)dx\\
  =&\,\int_{\RR}\left|\int_0^t\langle p_{t-s}(x-\cdot)\triangle(s,\cdot,\cdot), g(s,\cdot)\rangle_{\mathcal H}ds\right|^2\lambda(x)dx \\
  \lesssim&\,\int_{\RR}\int_0^t\|g(s,\cdot)\|^2_{\mathcal H}ds\cdot
\int_0^t\|p_{t-s}(x-\cdot)\triangle(s,\cdot,\cdot)\|^2_{\mathcal H}ds\lambda(x)dx\\
  \lesssim&\,\int_0^t\int_{\RR^3}\big[p_{t-s}(x-y-h)\triangle(s,y+h,y+h)-p_{t-s}(x-y)\Delta(s,y,y)\big]^2\\
  &\,\,\,\,\,\,\, \cdot|h|^{2H-2}\lambda(x)dhdydxds\\
  \lesssim&\,\int_0^t\int_{\RR^3}p^2_{t-s}(x-y)\left|\triangle(s,y,y+h)-\triangle(s,y,y)\right|^2\cdot|h|^{2H-2}\lambda(x)dhdydxds\\
  &+\int_0^t\int_{\RR^3}\left|D_{t-s}(x-y,h)\right|^2\left|\triangle(s,y+h,y)\right|^2\cdot|h|^{2H-2}\lambda(x)dhdydxds\\
  &+\int_0^t\int_{\RR^3}p^2_{t-s}(x-y)\left|\triangle(s,y+h,y)-\triangle(s,y,y)\right|^2\cdot|h|^{2H-2}\lambda(x)dhdydxds\\
=:&\,V_1(t)+V_2(t)+V_3(t).
  \end{split}
  \end{equation}}
    {By the integral mean value theorem}, we have
  \begin{align*}
 &\left|\triangle(s,y,y+h)-\triangle(s,y,y)\right|^2\\
 =&\,\Big|\int_0^1\left[u^g(s,y+h)-v^g(s,y+h)\right]\sigma'_\xi\left(s,y,\theta u^g(s,y+h)+(1-\theta) v^g\right)d\theta\\
 &-\int_0^1\left[u^g(s,y)-v^g(s,y)\right]\sigma'_\xi\left(s,y,\theta u^g(s,y)+(1-\theta) v^g\right)d\theta\Big|^2.
 \end{align*}
 By  \eqref{DuSigam} and \eqref{DuSigamAdd}, we have
 \begin{align}\label{490-490}
 &\left|\triangle(s,y,y+h)-\triangle(s,y,y)\right|^2\notag\\
 \lesssim&\,\left|u^g(s,y+h)-v^g(s,y+h)-u^g(s,y)+v^g(s,y)\right|^2\\
 &+\lambda^{\frac{2}{p_0}}(y)|u^g(s,y)-v^g(s,y)|^2\cdot\left[|u^g(s,y+h)-u^g(s,y)|^2+|v^g(s,y+h)-v^g(s,y)|^2\right].\notag
 \end{align}
 where    {$p_0\in (\frac{6}{4H-1},\infty)$} is the constant  appeared in \eqref{DuSigamAdd}.

 By \eqref{scale} and Lemma \ref{TechLemma1}, we have
 \begin{align}\label{V1_firstterm}
 &\int_0^t\int_{\RR^3}p^2_{t-s}(x-y)\left|u^g(s,y+h)-v^g(s,y+h)-u^g(s,y)+v^g(s,y)\right|^2\cdot|h|^{2H-2}\lambda(x)dhdydxds\notag\\
 \simeq&\,\int_0^t(t-s)^{-\frac{1}{2}}\int_{\RR^2}\cdot\left(\int_{\RR}p_{\frac{t-s}{2}(y-x)}\lambda(x)dx\right)\notag\\
 &\qquad\qquad\qquad\cdot\left|u^g(s,y+h)-v^g(s,y+h)-u^g(s,y)+v^g(s,y)\right|^2\cdot|h|^{2H-2}dhdyds\\
 \lesssim&\,\int_0^t(t-s)^{-\frac{1}{2}}\int_{\RR^2}\lambda(y)\left|u^g(s,y+h)-v^g(s,y+h)-u^g(s,y)+v^g(s,y)\right|^2\cdot|h|^{2H-2}dhdyds\notag\\
 =&\,\int_0^t(t-s)^{-\frac{1}{2}}\cdot S_2(s)ds.\notag
 \end{align}
 By \eqref{scale}, \eqref{Nug}, \eqref{Nvg} and Lemma \ref{TechLemma1}, we have
 \begin{equation}\label{V1_secondterm}
 \begin{split}
 &\int_0^t\int_{\RR^3}p^2_{t-s}(x-y)\lambda^{\frac{2}{p_0}}(y)|u^g(s,y)-v^g(s,y)|^2\\
 &\qquad\cdot
 \left[|u^g(s,y+h)-u^g(s,y)|^2+|v^g(s,y+h)-v^g(s,y)|^2\right]\cdot|h|^{2H-2}\lambda(x)dhdydxds\\
 \lesssim&\,\int_0^t\int_{\RR^2}p^2_{t-s}(x-y)|u^g(s,y)-v^g(s,y)|^2\lambda(x)dydxds\\
 \lesssim&\,\int_0^t(t-s)^{-\frac{1}{2}}\int_{\RR}\cdot\left(\int_\RR p_{\frac{t-s}{2}}(y-x)\lambda(x)dx\right)|u^g(s,y)-v^g(s,y)|^2dyds\\
 \lesssim&\,\int_0^t(t-s)^{-\frac{1}{2}}\int_{\RR}\lambda(y)|u^g(s,y)-v^g(s,y)|^2dyds\\
  =&\,\int_0^t(t-s)^{-\frac{1}{2}}\cdot S_1(s)ds.
 \end{split}
 \end{equation}
 Combining \eqref{490-490},  \eqref{V1_firstterm} with \eqref{V1_secondterm}, we obtain that \
 \begin{align}\label{v1label}
  V_1(t)\lesssim\int_0^t(t-s)^{-\frac{1}{2}}\cdot \left(S_1(s)+S_2(s)\right)ds.
 \end{align}
 A change of variable, Lemma \ref{TechLemma5} and \eqref{Lipschitzianuniformlycondu} yield that
 \begin{equation}\label{v2label}
 \begin{split}
 V_2(t)\lesssim&\,\int_0^t\int_{\RR^3}|D_{t-s}(x-y,h)|^2\cdot|u^g(s,y)-v^g(s,y)|^2\cdot|h|^{2H-2}\lambda(x)dhdydxds\\
 =&\,\int_0^t\int_{\RR}\cdot\left(\int_{\RR^2}|D_{t-s}(x,h)|^2\lambda(y-x)dx\right)\cdot|u^g(s,y)-v^g(s,y)|^2\cdot|h|^{2H-2}dhdyds\\
 \lesssim&\,\int_0^t(t-s)^{H-1}\int_{\RR}\lambda(y)|u^g(s,y)-v^g(s,y)|^2dyds\\
 =&\,\int_0^t(t-s)^{H-1}\cdot S_1(s)ds.
   \end{split}
  \end{equation}
If  $h>1$, then we have that by   \eqref{DuSigam},
 \begin{equation}\label{h1dayu}
 \begin{split}
 \left(\triangle(s,y+h,y)-\triangle(s,y,y)\right)^2
 =&\,\left|\int^{v^g}_{u^g}\left(\sigma'_\xi(s,y+h,\xi)-\sigma'_\xi(s,y,\xi)\right)d\xi\right|^2\\
 \lesssim&\, \left|u^g(s,y)-v^g(s,y)\right|^2.
  \end{split}
  \end{equation}
 If  $h\leq1$, then we have that by   \eqref{DuxSigam},
 \begin{equation}\label{h1xiaoyu}
 \begin{split}
 \left(\triangle(s,y+h,y)-\triangle(s,y,y)\right)^2
 =&\,\left|\int^{v^g}_{u^g}\left(\sigma'_\xi(s,y+h,\xi)-\sigma'_\xi(s,y,\xi)\right)d\xi\right|^2\\
 \lesssim&\,\left|u^g(s,y)-v^g(s,y)\right|^2\cdot \left|h\right|^2.
   \end{split}
  \end{equation}
 Thus, by  \eqref{pqt-s12simeq}, \eqref{h1dayu}, \eqref{h1xiaoyu} and Lemma \ref{TechLemma1}, we have
 \begin{equation}\label{v3label}
 \begin{split}
 V_3(t)=&\,\int_0^t\int_{\RR^2}\int_{|h|>1}p^2_{t-s}(x-y)\cdot |u^g(s,y)-v^g(s,y)|^2\cdot |h|^{2H-2}\lambda(x)dhdydxds\\
 &+\int_0^t\int_{\RR^2}\int_{|h|\leq1}p^2_{t-s}(x-y-h)\cdot |u^g(s,y)-v^g(s,y)|^2\cdot |h|^{2H}\lambda(x)dhdydxds\\
 \lesssim&\,\int_0^t\int_{\RR^2}p^2_{t-s}(x-y)\cdot |u^g(s,y)-v^g(s,y)|^2\lambda(x)dydxds\\
 \simeq&\,\int_0^t(t-s)^{\frac{1}{2}}\int_{\RR}\cdot\left(\int_{\RR}p_{\frac{t-s}{2}}(y-x)
 \lambda(x)dx\right)\cdot |u^g(s,y)-v^g(s,y)|^2dyds\\
  \lesssim&\,\int_0^t(t-s)^{\frac{1}{2}}\int_{\RR}\lambda(y) |u^g(s,y)-v^g(s,y)|^2dyds\\
 =&\,\int_0^t(t-s)^{-\frac{1}{2}}\cdot S_1(s)ds.
 \end{split}
  \end{equation}

 By \eqref{4.84}, \eqref{v1label}, \eqref{v2label} and \eqref{v3label}, we have
 \begin{align}\label{s1label}
 S_1(t)\lesssim\, \int^t_0(t-s)^{H-1}\cdot [S_1(s)+S_2(s)]ds.
 \end{align}
  {Using the  similar procedure as in the proof of \eqref{s1label}, we  obtain that  }
 \begin{align}\label{s2label}
 S_2(t)\lesssim \,\int^t_0(t-s)^{2H-\frac{3}{2}}\cdot [S_1(s)+S_2(s)]ds.
 \end{align}
Therefore, by \eqref{s1label} and \eqref{s2label}, we have
\begin{align*}
 S_1(t)+ S_2(t)\lesssim\, \int^t_0\left((t-s)^{H-1}+(t-s)^{2H-\frac{3}{2}}\right)\cdot[S_1(s)+S_2(s)]ds.
 \end{align*}
  {Notice that $S_1(t)$ and $S_2(t)$ are uniformly bounded on $[0,T]$ as in \eqref{eq S12}. By the fractional Gr\"onwall lemma (\cite[Lemma 7.1.1]{Herry1981}), we have $$S_1(t)+S_2(t)=0 \ \ \ \text{for all } t\in[0,T].$$
 This, together with the continuities of $u^g$ and $v^g$,  implies that  $u^g=v^g$.}

 The proof is complete.
 \end{proof}

\section{Verification of Condition \ref{cond1} $(a)$}
We now verify Condition \ref{cond1} $(a)$. Recall that $\Gamma^{0}\left(\int_0^\cdot g(s)ds\right)=u^g$ for $g\in \mathbb{S}$, where $u^g$ is the solution of Eq.\,\eqref{eq skeleton}.
\begin{proposition}\label{thm continuity skeleton}
Assume  that   $\sigma$ satisfies the   hypothesis $(\mathbf{H})$.
For any   {$N\geq1$}, let $g_n, g\in S^N$ be such that $g_n\rightarrow g$ weakly as $n\rightarrow \infty$. Let $u^{g_n}$ denote the solution to Eq.\,\eqref{eq skeleton} replacing $g$ by $g_n$. Then, as $n\rightarrow\infty$,
$$ u^{g_n}\longrightarrow  u^{g} \ \ \text{
in } \mathcal{C}([0, T]\times\RR).
$$
\end{proposition}
\begin{proof}
The proof is divided into two steps.  In Step 1, we prove that  the family $\{u^{g_n}\}_{n\geq1}$ is relatively compact  in $\mathcal{C}([0,T]\times\RR)$, which implies that there exists a subsequence of $\{u^{g_n}\}_{n\geq1}$ (still denoted by $\{u^{g_n}\}_{n\geq1}$)   such that $u^{g_n}\rightarrow u$ as $n\rightarrow\infty$ in $\mathcal{C}([0,T]\times\RR)$ for some function $u\in\mathcal{C}([0,T]\times\RR)$. In Step 2, we show that $u=u^g$.

{\bf{Step 1.}} 
Recall that \begin{equation}\label{ugn}
u^{g_n}(t,x)=\,1+\int_0^t\langle p_{t-s}(x-\cdot)\sigma(s, \cdot,u^{g_n}(s, \cdot)), {g_n}(s,\cdot)\rangle_{\mathcal H}ds.
 \end{equation}
   { Since  $\{g_n\}_{n\ge1}\subset S^N$,
 by Proposition \ref{thm solu skeleton}, we know that  for any $p\geq 2$,
 \begin{equation}\label{ugn-z-normal}
 \sup_{n\geq1}\left\|u^{g_n}\right\|_{ Z_{\lambda,T}^p}<\infty.
\end{equation}}

As in \eqref{461-transform}, for any $\alpha\in(0,1)$, we have
\begin{align*}
& \int_0^t\langle p_{t-s}(x-\cdot)\sigma(s, \cdot,u^{g_n}(s, \cdot)), {g_n}(s,\cdot)\rangle_{\mathcal H}ds
\simeq\int_0^t\int_{\RR}(t-r)^{\alpha-1}\cdot p_{t-r}(x-z)\cdot I_\alpha(r,z)dzdr,
 \end{align*}
 where
 \begin{align*}
 I_\alpha(r,z):=\,\int_0^r\langle (r-s)^{-\alpha}p_{r-s}(z-\cdot)\sigma(s, \cdot,u^{g_n}(s, \cdot)), {g_n}(s,\cdot)\rangle_{\mathcal H}ds.
 \end{align*}
 If $\alpha<\frac{H}{2}$, then by using the same technique  as that in the proof of (\ref{J-12-Third}), we have
\begin{equation}\label{I-bound-first}
\begin{split}
 \big\|I_\alpha(r,\cdot)\big\|^p_{L^p_\lambda(\mathbb{R})}
\lesssim \,1+\|u^{g_n}\big\|^p_{{Z}^p_{\lambda,T}}.
\end{split}
\end{equation}
Fix $\gamma\in(0,1)$. By using the same method as that in the proof  of    \eqref{phi-estimate}, if $\alpha<\frac{H}{2}$, then we have
     \begin{equation}\label{faignthx1}
     \begin{split}
      & \sup\limits_{\substack{t,t+h\in[0,T], \\ x\in \RR}}\lambda^{\frac{1}{p}}(x)\cdot    \left|  {u^{g_n}(t+h,x)-u^{g_n}(t,x)}\right| \\
       \lesssim &\, |h|^{\gamma}\cdot \left(\int_{0}^{t}(t-r)^{q(\alpha-1-\gamma)+\frac{1- q}{2}}dr\right)^{\frac{1}{q}}\cdot \left(1+\|u^{g_n}\|_{Z_{\lambda,T}^p}\right).
       \end{split}
     \end{equation}
  {If  $\alpha<\frac{H}{2}$ and $q(\alpha-1-\gamma)+\frac{1-q}{2}>-1$, namely,} if $p>\frac{3}{H}$ and
     $0<\gamma<\frac H2-\frac{3}{2p}$, then there exists a positive constant $ C_{T,p,H,N,\gamma}$ such that
     \begin{align}\label{continous_time_ugn}
       \sup\limits_{\substack{t, t+h\in[0,T], \\ x\in \RR}}\lambda^{\frac{1}{p}}(x)\cdot    \left|u^{g_n}(t+h,x)-u^{g_n}(t,x)\right| \leq \, C_{T,p,H,N,\gamma}\cdot |h|^{\gamma}\cdot \left(1+\|u^{g_n}\|_{Z_{\lambda,T}^p}\right).
     \end{align}
     By using the same method as  that in the proof of \eqref{481-x-y}, if  $\alpha<\frac{H}{2}$, then we have
 \begin{equation}\label{faigntxyh2}
     \begin{split}
&\,\left|  {u^{g_n}(t,x)-u^{g_n}(t,y)}\right|\\
\lesssim&\,\big|x-y\big|^\gamma\cdot\left[\lambda^{-\frac{1}{p}}(x)+\lambda^{-\frac{1}{p}}(y)\right]
\cdot\left(\int_{0}^{t}(t-r)^{(\alpha q-\frac{3q}{2}+\frac{1}{2})-\frac{\gamma q}{2}}dr\right)^{\frac{1}{q}}
 \cdot\left(1+\|u^{g_n}\big\|_{{Z}^p_{\lambda,T}}\right).
 \end{split}
 \end{equation}
  {If $\alpha<\frac{H}{2}$ and $(\alpha q-\frac{3q}{2}+\frac{1}{2})-\frac{\gamma q}{2}$, namely,} if $p>\frac{3}{H}$ and
   $0<\gamma<H-\frac{3}{p}$,  then there exists a positive constant $ C_{T,p,H,N,\gamma}$ such that
     \begin{align}\label{Lemma-4.3-iii-u-gn}
       \sup\limits_{\substack{t\in[0,T], \\ x, h\in \RR}}\frac{\left|u^{g_n}(t,x)-u^{g_n}(t,y)\right|}{\lambda^{-\frac{1}{p}}(x)+\lambda^{-\frac{1}{p}}(y)}  \leq\, C_{T,p,H,N,\gamma}\cdot |x-y|^{\gamma}\cdot \left(1+\|u^{g_n}\|_{Z_{\lambda,T}^p}\right).
     \end{align}
  {By  \eqref{continous_time_ugn} and \eqref{Lemma-4.3-iii-u-gn}, we have that, for any   $R>0$ and $\gamma>0$, there exists $\theta>0$ such that
\begin{align*}
       \max\limits_{\substack{|t-s|+|x-y|\leq \theta, \\ 0\leq t,s \leq T;\,    -R\le x,y \le R }}|u^{g_n}(t,x)-u^{g_n}(s,y)|  \leq\gamma.
     \end{align*}
Combining this with the initial value $u^{g_n}(0,\cdot)\equiv1$, we know that $\{u^{g_n}\}_{n\geq1}$ is relatively compact on the space $({\mathcal C}([0,T]\times\RR), d_{\mathcal C})$ by the Arzel\`a-Ascoli theorem.} Thus,
 there exists a subsequence of $\{u^{g_n}\}_{n\geq1}$ (still denoted by $\{u^{g_n}\}_{n\geq1}$) and $u\in\mathcal{C}([0,T]\times\RR)$ such that $u^{g_n}\rightarrow u$ as $n\rightarrow \infty$.
  {By Lemma \ref{lemma3.3limits} (i) and \eqref{ugn-z-normal}}, we have   that, for any $p\ge2$,
\begin{align}\label{and}
\|u\|_{ Z_{\lambda,T}^p}<\infty.
\end{align}

  {\textbf{Step 2.}} In this step, we prove that $u=u^g$.
  Denote that
  \begin{align*}
  D_1(t)=\,\int_{\mathbb{R}}|u^{g_n}(t,x)-u^g(t,x)|^2\lambda(x)dx,
  \end{align*}
and
   \begin{align*}
  D_2(t)=\,\int_{\RR^2}\left|u^{g_n}(t,x)-u^g(t,x)-u^{g_n}(t,x+h)+u^g(t,x+h)\right|^2\cdot|h|^{2H-2}\lambda(x)dhdx.
  \end{align*}
  According to  \eqref{LPlambdaRNorm}, \eqref{N*pNorm}  and \eqref{ugZlamdanormal}, we know that $\displaystyle\sup_{t\in[0,T]}D_1(t)<\infty$ and $\displaystyle\sup_{t\in[0,T]}D_2(t)<\infty$.
  Recall $D_{t}(x,h)$ defined in \eqref{TechDt} and denote that $$\triangle_n(t,x,y):=\sigma(t,x,u^{g_n}(t,y))-\sigma(t,x,u^g(t,y)).$$
  By \eqref{eq skeleton} and \eqref{ugn}, we have
  \begin{equation}\label{D_1+D_2}
      \begin{split}
        &D_1(t)+D_2(t)\\
        \leq&\,2\int_{\RR}\left|\int^t_0\langle p_{t-s}(x-\cdot)\triangle_n(s, \cdot,\cdot), {g_n}(s,\cdot)\rangle_{\mathcal H}ds\right|^2\lambda(x)dx\\
        &+\,2\int_{\RR}\left|\int^t_0\langle p_{t-s}(x-\cdot)\sigma(s,\cdot,u^g(s,\cdot)), {g_n}(s,\cdot)-g(s,\cdot)\rangle_{\mathcal H}ds\right|^2\lambda(x)dx\\
        &+\,2\int_{\RR}\left|\int^t_0\langle D_{t-s}(x-\cdot,h)\triangle_n(s, \cdot,\cdot), {g_n}(s,\cdot)\rangle_{\mathcal H}ds\right|^2|h|^{2H-2}\lambda(x)dhdx\\
        &+\,2\int_{\RR}\left|\int^t_0\langle D_{t-s}(x-\cdot)\sigma(s,\cdot,u^g(s,\cdot)), {g_n}(s,\cdot)-g(s,\cdot)\rangle_{\mathcal H}ds\right|^2|h|^{2H-2}\lambda(x)dhdx\\
        =:&\,2\left (E_1(t)+E_2(t)+E_3(t)+E_4(t)\right).
      \end{split}
  \end{equation}
  By the similar technique as that  in the uniqueness of the proof of Proposition
   \ref{thm solu skeleton}, we have
   \begin{equation}\label{E_1+E_3}
      \begin{split}
        E_1(t)+E_3(t)
        \lesssim\,
   \int^t_0\left((t-s)^{H-1}+(t-s)^{2H-\frac{3}{2}}\right)\cdot\left[D_1(s)+D_2(s)\right]ds.
      \end{split}
  \end{equation}
 On the other hand, denote that
 \begin{align*}
   G_n(t,x):=\left|\int^t_0\langle p_{t-s}(x-\cdot)\sigma(s,\cdot,u^g(s,\cdot)), {g_n}(s,\cdot)-g(s,\cdot)\rangle_{\mathcal H}ds\right|^2.
 \end{align*}
  By Lemma \ref{lemma3.3limits} (ii),
 we know that for almost all $x\in \mathbb R$,
 \begin{align}\label{eq H G1}
    \int_{0}^{t}\left\| p_{t-s}(x-\cdot)\sigma(s,\cdot,u^{g}(s,\cdot))\right\|_{\HH}^2ds<\infty.
 \end{align}
 This, together with the weak convergence of $g_n$ to $g$, implies that  almost all $x\in \mathbb R$,
  \begin{align}\label{eq H G2}
     G_n(t,x) \rightarrow 0, \ \ \text{as }  n\rightarrow \infty
 \end{align}
 Since $g_n, g\in S^N$,  by using the Cauchy-Schwarz inequality, \eqref{eq H product}, and a change of variable, we have that, for any $p\geq 2$,
 \begin{align}
       &\int_{\RR} G_n(t,x) ^{\frac{p}{2}}\lambda(x)dx\notag\\
      \lesssim&\,\int_{\RR}\lc\int_{0}^{t}\left\| p_{t-s}(x-\cdot)\sigma(s,\cdot,u^{g}(s,\cdot))\right\|_{\HH}^2ds\rc^{\frac p2}\lambda(x)dx \notag\\
     \simeq&\,
       \int_{\RR}\bigg(\int_{0}^{t}\int_{\RR^2}  \Big|p_{t-s}(x-y-h)\sigma(s,y+h,u^{g}(s,y+h))\\
       &\qquad\qquad\quad-p_{t-s}(x-y)\sigma(s,y,u^{g}(s,y))\Big|^2 \cdot|h|^{2H-2}dhdyds\bigg)^{\frac p2}\lambda(x)dx.\notag
     \end{align}
  By the similar technique as that  in Step 2 in the proof of Proposition
   \ref{thm solu skeleton}, we have
  \begin{align*}
       &\sup_{n\geq1}\int_{\RR} G_n(t,x) ^{\frac{p}{2}}\lambda(x)dx
      \lesssim \|u^g\|^p_{Z^p_{\lambda,T}}<\infty.
   \end{align*}
 It follows from \cite[p. 105, Exercise 8]{Chung} that $\{G_n(t,x)\}_{n\geq 1}$ is $L^1$-uniformly integrable in $(\mathbb R, \lambda(x)dx)$, namely,
 \begin{align}\label{Mda}
     \lim_{M\rightarrow\infty}\sup_{n\geq1}
     \int_{ G_n(t,x) >M} G_n(t,x) \lambda(x)dx=0.
 \end{align}
By \eqref{eq H G2} and the dominated convergence theorem,  we have
 \begin{equation}\label{Mxiao}
     \begin{split}
  \lim_{n\rightarrow\infty}
  \int_{ G_n(t,x) \leq M} G_n(t,x) \lambda(x)dx
  =\,0.
      \end{split}
 \end{equation}
 By \eqref{Mda} and \eqref{Mxiao}, we have
 \begin{align}\label{E2lim0}
   \lim_{n\rightarrow\infty} E_2(t)=0.
 \end{align}
Using the same technique as that in the proof of \eqref{E2lim0}, we can prove that
  \begin{align}\label{E4lim0}
   \lim_{n\rightarrow\infty} E_4(t)=0.
 \end{align}
Since $D_1(t)$ and $D_2(t)$ are uniformly bounded on $[0,T]$, they are  integrable on $[0,T]$. Putting \eqref{D_1+D_2}, \eqref{E_1+E_3}, \eqref{E2lim0} and \eqref{E4lim0} together, by the fractional Gr\"onwall lemma (  {\cite[Lemma 7.1.1]{Herry1981}}), we have $$\displaystyle\lim_{n\rightarrow\infty}\left[D_1(t)+D_2(t)\right]=0,  \, \text{for all }\, t\in[0,T].$$
 In particular, $u^{g_n}(t,\cdot)\rightarrow u^g(t,\cdot)$ as $n\rightarrow \infty$ in the space $ L_{\lambda}^2(\RR)$ for all $t\in [0,T]$.
   Since
   $u^{g_n}$ also converges to $u$ as $n\rightarrow \infty$ in the space $\left(\mathcal{C}([0,T]\times\RR),d_{\mathcal{C}}\right)$, the uniqueness of the limit of $u^{g_n}$ implies that $u=u^g$.

 The proof is complete.
\end{proof}

\section{Verification of Condition \ref{cond1} $(b)$}

For any $\e>0$, define the  solution functional  $\Gamma^{\e}:    {C([0,T];\RR^{\infty})} \rightarrow \mathcal C([0,T]\times\RR) $ by
 \begin{align}\label{eq Gamma e}
 \Gamma^{\e}\left(W(\cdot)\right):=u^\e,
 \end{align}
  where $u^\e$ stands for the solution of Eq.\,\eqref{SHE} and $W$ can be regarded  as a  cylindrical Brownian motion on $\HH$ by \eqref{eq int 10}.

Let $\{g^{\e}\}_{\e>0}\subset \mathcal U^N$ be a given family of stochastic processes.
By the Girsanov theorem, it is easily to see that $\tilde{u}^{\e}:=\Gamma^{\e}\left(W({\cdot})+\frac{1}{\sqrt \e} \int_0^{\cdot}{g}^{\e}(s)ds \right)$ is the unique solution of the  equation
 \begin{equation}\label{eq u e1}
 \begin{split}
  \frac{\partial \tilde{u}^{\e}(t,x)}{\partial t}
  =&\,\frac{\partial^2 \tilde{u}^{\e}(t,x)}{\partial x^2}+\sqrt{\e}\sigma(t,x,\tilde{u}^{\e}(t,x))\dot{W}(t,x) \\
  &+\left\langle \sigma(t,\cdot,\tilde{u}^{\e}(t,\cdot)), g^{\e}(t,\cdot)\right\rangle_{\HH}, \ \ \  \,\,t>0,\,\,x\in\RR
\end{split}
\end{equation}
with the initial value $\tilde{u}^{\e}(0,\cdot)\equiv1$. 

Recall the map $ \Gamma^0$ defined by \eqref{eq Gamma0}. Then 
  $\bar{u}^{\e}:=\Gamma^0\left(\int_0^{\cdot} g^{\e}(s)ds\right)$    solves the equation
 \begin{align}\label{eq u e2}
  \frac{\partial \bar{u}^{\e}(t,x)}{\partial t}
  =&\,\frac{\partial^2 \bar{u}^{\e}(t,x)}{\partial x^2}+\left\langle \sigma(t,\cdot,\bar{u}^{\e}(t,\cdot)), g^{\e}(t,\cdot)\right\rangle_{\HH},  \,\,t>0,\,\,x\in\RR
\end{align}
with the initial value $\bar{u}^{\e}(0,\cdot)\equiv1$. 

Equivalently,  we have
 \begin{equation}\label{eq SPDE Y-1}
\begin{split}
   \tilde{u}^{\e}(t,x)
  =&\,1+\sqrt{\e}\int^t_0\int_{\RR}p_{t-s}(x-y)\sigma(s,y,\tilde{u}^{\e}(s,y)){W}(ds,dy) \\
  &+\int^t_0\left\langle p_{t-s}(x-\cdot)\sigma(s,\cdot,\tilde{u}^{\e}(s,\cdot)), g^{\e}(s,\cdot)\right\rangle_{\HH}ds,
\end{split}
\end{equation}
and
\begin{align}\label{eq SPDE Z-1}
   \bar{u}^{\e}(t,x)
  =&\,1+\int^t_0\left\langle p_{t-s}(x-\cdot)\sigma(s,\cdot,\bar{u}^{\e}(s,\cdot)), g^{\e}(s,\cdot)\right\rangle_{\HH}ds.
\end{align}
 \begin{proposition}\label{proposition-4-3} Assume that   $\sigma$ satisfies the hypothesis $(\mathbf{H})$ for some constant $p_0>\frac{6}{4H-1}$.  
For every $N<+\infty$ and $\{g^{\e}\}_{\e>0}\subset \mathcal U^N$, it holds that for any   $\delta>0$,
    $$\lim_{\e\rightarrow 0}\mathbb P\left(d_{\mathcal C}(\tilde{u}^{\e}, \bar{u}^{\e})>\delta\right)=0. $$ \end{proposition}

Before proving Proposition \ref{proposition-4-3}, we give the following lemmas.
 \begin{lemma}\label{lemma-5-1}  Assume that   $\sigma$ satisfies the hypothesis $(\mathbf{H})$ for some constant $p_0>\frac{6}{4H-1}$.   Then, it holds that  for any $p> p_0$,
   \begin{align} \label{4-95-bound}
   \sup_{0<\e<1}\|\tilde{u}^{\e}\|_{\mathcal{Z}^p_{\lambda,T}}<\infty,\,\,\,\,\sup_{0<\e<1}\|\bar{u}^{\e}\|_{\mathcal{Z}^p_{\lambda,T}}<\infty.
   \end{align}
   \end{lemma}

\begin{proof}
 We give the details of the proof for  $\tilde{u}^{\e}$,  while the proof for   $\bar{u}^{\e}$ is similar but simpler which is omitted. Here,  the proof is inspired by the proof of \cite[Lemma 4.5]{HW2019}. 
 
\textbf{Step 1.}  
As in  (4.38) of \cite{HW2019},  let    
\begin{align*}
\frac{\partial}{\partial x}W_{\eta}(t,x):=\int_{\RR} p_{\eta}(x-y)W(t,dy), \ \ \ \ \text{for  any }  \eta>0. 
\end{align*}
Consider 
\begin{equation}\label{tildeue_eta_Def}
\begin{split}
\tilde{u}^{\e}_{\eta}(t,x)
  =&\,1+\sqrt{\e}\int^t_0\int_{\RR}p_{t-s}(x-y)\sigma(s,y,\tilde{u}^{\e}_{\eta}(s,y))W_{\eta}(ds,dy) \\
  &+\int^t_0\left\langle p_{t-s}(x-\cdot)\sigma(s,\cdot,\tilde{u}^{\e}_{\eta}(s,\cdot)), g^{\e}(s,\cdot)\right\rangle_{\HH_{\eta}}ds\\
  =:&\,1+\sqrt{\e}\Phi_{1,{\eta}}^{\e}(t,x)+\Phi_{2,{\eta}}^{\e}(t,x).
\end{split}
\end{equation}
Let
$$
\tilde{u}^{\e,0}_{\eta}(t,x)=1,
$$
and recursively for $n=0,1,2,...,$
 \begin{equation}\label{eq SPDE Y-1_n_e1}
\begin{split}
   \tilde{u}^{\e,n+1}_{\eta}(t,x)
  =&\,1+\sqrt{\e}\int^t_0\int_{\RR}p_{t-s}(x-y)\sigma(s,y,\tilde{u}^{\e,n}_{\eta}(s,y)){W}_{\eta}(ds,dy) \\
  &+\int^t_0\left\langle p_{t-s}(x-\cdot)\sigma(s,\cdot,\tilde{u}^{\e,n}_{\eta}(s,\cdot)), g^{\e}(s,\cdot)\right\rangle_{\HH_{\eta}}ds\\
  =:&1+\sqrt{\e}\Phi_{1,{\eta}}^{\e,n}(t,x)+\Phi_{2,{\eta}}^{\e,n}(t,x).
\end{split}
\end{equation}

By using  \cite[Lemma 4.15]{HHLNT2017} and the similar argument as  that in Step 1 in the proof of Lemma \ref{UniBExist}, we know that for any fixed $t\in [0,T]$ and $\eta>0$, when $n$ goes to infinity, $\tilde{u}^{\e,n}_{\eta}(t,\cdot)$ converges to $\tilde{u}^{\e}_{\eta}(t,\cdot)$ in $L^p_{\lambda}(\Omega\times\mathbb{R})$.
 
 By using the same method as  that   in the proof of Lemma \ref{UniBExist}, we obtain that
\begin{equation}\label{eq u e52}
\begin{split}
&\|\Phi_{2,{\eta}}^{\e,n+1}(t,\cdot)\|^2_{L^p_{\lambda}(\Omega\times\mathbb{R})}
+\left[\mathcal{N}^*_{\frac{1}{2}-H,\,p}\Phi^{\e,n+1}_{2,{\eta}}(t)\right]^2\\
\lesssim&\,1+\int^t_0\left((t-s)^{2H-\frac{3}{2}}+(t-s)^{H-1}+(t-s)^{-\frac{1}{2}}\right)\cdot \|\tilde{u}^{\e,n}_{\eta}(s,\cdot)\|^2_{L^p_{\lambda}(\Omega\times\RR)}ds\\
&+\,\int^t_0\left((t-s)^{H-1}+(t-s)^{-\frac{1}{2}}\right)\cdot\left[\mathcal{N}^*_{\frac{1}{2}-H,\,p}\tilde{u}^{\e,n}_{\eta}(s)\right]^2ds.
\end{split}
\end{equation}

Next, we will give some estimates for  $\|\Phi_{1,{\eta}}^{\e,n}(t,\cdot)\|_{L^p_{\lambda}(\Omega\times\mathbb{R})}$\ and $\mathcal{N}^*_{\frac{1}{2}-H,\,p}\Phi^{\e,n}_{1,{\eta}}(t)$.
  
  \textbf{Step 2.} In this step, we estimate $\|\Phi_{1,{\eta}}^{\e,n}(t,\cdot)\|_{L^p_{\lambda}(\Omega\times\mathbb{R})}$. 
     By the Burkholder-Davis-Gundy inequality, we have that
\begin{equation}\label{5.9label}
\begin{split}
&\mathbb{E}\left[|\Phi_{1,{\eta}}^{\e,n+1}(t,x)|^p\right]\\
\lesssim&\,\mathbb{E}\bigg(\int^t_0\int_{\RR^2}\big|p_{t-s}(x-y-h)\sigma(s,y+h,\tilde{u}^{\e,n}_{\eta}(s,y+h))\\
&\qquad\qquad\quad-p_{t-s}(x-y)\sigma(s,y,\tilde{u}^{\e,n}_{\eta}(s,y))\big|^2\cdot|h|^{2H-2}dhdyds\bigg)^{\frac{p}{2}}\\
   \lesssim&\,\mathbb{E}\bigg(\int_{0}^{t}\int_{\RR^2} p^2_{t-s}(x-y-h)  \big|\sigma(s,y+h,\tilde{u}^{\e,n}_{\eta}(s,y+h))
    -\sigma(s,y,\tilde{u}^{\e,n}_{\eta}(s,y+h))\big|^2\\ &\qquad\quad\cdot|h|^{2H-2}dhdyds\bigg)^{\frac p2} \\
     &+\,\mathbb{E}\bigg(\int_{0}^{t}\int_{\RR^2} p^2_{t-s}(x-y-h)  \big|\sigma(s,y,\tilde{u}^{\e,n}_{\eta}(s,y+h))-\sigma(s,y,\tilde{u}^{\e,n}_{\eta}(s,y))\big|^2\\
     &\qquad\quad\cdot|h|^{2H-2}dhdyds\bigg)^{\frac p2}\\
     &+\,\mathbb{E}\bigg( \int_{0}^{t}\int_{\RR^2}|D_{t-s}(x-y,h)|^2\cdot\left|\sigma(s,y,\tilde{u}^{\e,n}_{\eta}(s,y))\right|^2
             \cdot|h|^{2H-2} dhdyds\bigg)^{\frac p2}\\
   =:&\,\mathcal{L}_1(t,x)+\mathcal{L}_2(t,x)+\mathcal{L}_3(t,x),
\end{split}
\end{equation}
where
 $D_{t-s}(x-y,h)$ is defined by \eqref{TechDt}.

By using the same arguments as to that in Step 3 of the proof of Lemma \ref{UniBExist}, we have 
\begin{equation}\label{mathcalL_1label}
\begin{split}
&\left(\int_{\mathbb{\RR}}\mathcal{L}_1(t,x)\lambda(x)dx\right)^{\frac{2}{p}}\\
\lesssim&\,\int^t_0\int_{\RR}p^2_{t-s}(y)\left(\int_\RR\left(1+\|\tilde{u}^{\e,n}_{\eta}(s,x+y)\|^p_{L^p(\Omega)}\right)\lambda(x)dx\right)^{\frac{2}{p}}dyds\\
\lesssim&\,\int^t_0(t-s)^{-\frac{1}{2}}\cdot \left(\int_{\RR^2} p_{\frac{t-s}{2}}(y)\left(1+\|\tilde{u}^{\e,n}_{\eta}(s,x)\|^p_{L^p(\Omega)}\right)\lambda(x-y)dxdy\right)^{\frac{2}{p}}ds\\
\lesssim&\,\int^t_0(t-s)^{-\frac{1}{2}}\cdot \left(1+\|\tilde{u}^{\e,n}_{\eta}(s,\cdot)\|^2_{L^p_{\lambda}(\Omega\times\mathbb{R})}\right)ds,
\end{split}
\end{equation}
\begin{align}\label{mathcalL_2label}
&\left(\int_{\mathbb{\RR}}\mathcal{L}_2(t,x)\lambda(x)dx\right)^{\frac{2}{p}}\notag\\
\lesssim&\,\int^t_0\int_{\RR^2}p^2_{t-s}(y)\left(\int_\RR\|\tilde{u}^{\e,n}_{\eta}(s,x+y+h)-\tilde{u}^{\e,n}_{\eta}(s,x+y)\|^p_{L^p(\Omega)}\lambda(x)dx\right)^{\frac{2}{p}}\cdot |h|^{2H-2}dhdyds\notag\\
\lesssim&\,\int^t_0\int_{\RR}(t-s)^{-\frac{1}{2}}\cdot \Bigg(\int_{\RR^2}p_{\frac{t-s}{2}}(y)\|\tilde{u}^{\e,n}_{\eta}(s,x+h)-\tilde{u}^{\e,n}_{\eta}(s,x)\|^p_{L^p(\Omega)}\lambda(x-y)dxdy\Bigg)^{\frac{2}{p}}\\
&\quad\quad\cdot |h|^{2H-2}dhds\notag\\
\lesssim&\,\int^t_0(t-s)^{-\frac{1}{2}}\cdot \left[\mathcal{N}^*_{\frac{1}{2}-H,\,p}\tilde{u}^{\e,n}_{\eta}(s)\right]^2ds,\notag
\end{align}
and
\begin{align}\label{mathcalL_3label}
&\left(\int_{\mathbb{\RR}}\mathcal{L}_3(t,x)\lambda(x)dx\right)^{\frac{2}{p}}\notag\\
\lesssim&\,\int^t_0\int_{\RR^2}|D_{t-s}(y,h)|^2\cdot \left(\int_\RR\left(1+\|\tilde{u}^{\e,n}_{\eta}(s,x+y)\|^p_{L^p(\Omega)}\right)\lambda(x)dx\right)^{\frac{2}{p}}\cdot |h|^{2H-2}dhdyds\notag\\
\lesssim&\,\int^t_0(t-s)^{H-1}\cdot \Bigg(\int_{\RR^3}(t-s)^{1-H}\cdot |D_{t-s}(y,h)|^2\cdot \left(1+\|\tilde{u}^{\e,n}_{\eta}(s,x)\|^p_{L^p(\Omega)}\right)\lambda(x-y)\\
&\qquad\qquad\qquad\quad\cdot |h|^{2H-2}dhdxdy\Bigg)^{\frac{2}{p}}ds\notag\\
\lesssim&\,\int^t_0(t-s)^{H-1}\cdot \left(1+\|\tilde{u}^{\e,n}_{\eta}(s,\cdot)\|^2_{L^p_{\lambda}(\Omega\times\mathbb{R})}\right)ds.\notag
\end{align}
Therefore, by \eqref{5.9label}, \eqref{mathcalL_1label}, \eqref{mathcalL_2label} and \eqref{mathcalL_3label}, we have
\begin{align}\label{511.511}
\|\Phi_{1,{\eta}}^{\e,n+1}(t,\cdot)\|^2_{L^p_{\lambda}(\Omega\times\mathbb{R})}
=&\,\left(\int_{\mathbb{R}}\mathbb{E}\left[\left|\Phi_{1,{\eta}}^{\e,n+1}(t,x)\right|^p\right]\lambda(x)dx\right)^{\frac{2}{p}}\nonumber\\
\lesssim&\,1+\int^t_0\left((t-s)^{-\frac{1}{2}}+(t-s)^{H-1}\right)\cdot \|\tilde{u}_{\eta}^{\e,n}(s,\cdot)\|^2_{L^p_{\lambda}(\Omega\times\mathbb{R})}ds\\
&+\int^t_0(t-s)^{-\frac{1}{2}}\cdot \left[\mathcal{N}^*_{\frac{1}{2}-H,\,p}\tilde{u}^{\e,n}_{\eta}(s)\right]^2ds.\nonumber
\end{align}

\textbf{Step 3.} In this step, we deal with $\mathcal{N}^*_{\frac{1}{2}-H,\,p}\Phi_{1,{\eta}}^{\e,n}(t)$. By the Burkholder-Davis-Gundy inequality and the hypothesis $(\mathbf{H})$, the similar calculation as that in Step 2 of the proof of Lemma \ref{UniBExist} implies that
\begin{align*}
&\mathbb{E}\big[\left|\Phi_{1,{\eta}}^{\e,n+1}(t,x)-\Phi_{1,{\eta}}^{\e,n+1}(t,x+h)\right|^p\big]\\
\simeq&\,\mathbb{E}\bigg(\int^t_0\int_{\RR^2}\big|D_{t-s}(x-y-z,h)\sigma(s,y+z,\tilde{u}^{\e,n}_{\eta}(s,y+z))\nonumber\\
&\qquad\qquad\,\,\,\,\,-D_{t-s}(x-z,h)\sigma(s,z,\tilde{u}^{\e,n}_{\eta}(s,z))\big|^2\cdot|y|^{2H-2}dzdyds\bigg)^{\frac{p}{2}}\nonumber\\
\lesssim&\,\mathbb{E}\Bigg(\int^t_0\int_{\RR}|D_{t-s}(x-z,h)|^2\cdot \left(1+|\tilde{u}^{\e,n}_{\eta}(s,z)|^2\right)dzds\Bigg)^{\frac{p}{2}}\\
&+\,\mathbb{E}\Bigg(\int^t_0\int_{\RR^2}|D_{t-s}(x-y-z,h)|^2\cdot \left|\tilde{u}^{\e,n}_{\eta}(s,y+z)
-\tilde{u}^{\e,n}_{\eta}(s,z)\right|^2\cdot |y|^{2H-2}dzdyds\Bigg)^{\frac{p}{2}}\\
&+\,\mathbb{E}\left(\int^t_0\int_{\RR^2}|\Box_{t-s}(x-z,y,h)|^2\cdot \left(1+|\tilde{u}^{\e,n}_{\eta}(s,z)|^2\right)\cdot |y|^{2H-2}dzdyds\right)^{\frac{p}{2}}\\
=:&\,\mathcal{M}_1(t,x,h)+\mathcal{M}_2(t,x,h)+\mathcal{M}_3(t,x,h),
\end{align*}
where $\Box_{t-s}(x-z,y,h)$ is defined by \eqref{TechBoxt}.

By   \eqref{LpNormOmegatimesRR} and \eqref{N*pNorm}, we have
\begin{align}\label{5.100label}
\left[\mathcal{N}^*_{\frac{1}{2}-H,\,p}\Phi_{1,{\eta}}^{\e,n+1}(t)\right]^2
\lesssim\sum_{i=1}^3\int_{\RR}\left(\int_{\RR}\mathcal{M}_i(t,x,h)\lambda(x)dx\right)^{\frac{2}{p}}\cdot |h|^{2H-2}dh.
\end{align}
Applying a change of variable, Minkowski's inequality, Jensen's inequality and Lemma \ref{TechLemma5}, we have
\begin{align}
\int_{\RR}\left(\int_{\RR}\mathcal{M}_1(t,x,h)\lambda(x)dx\right)^{\frac{2}{p}}|h|^{2H-2}dh
\lesssim&\,\int^t_0(t-s)^{H-1}\cdot \left(1+\|\tilde{u}^{\e,n}_{\eta}(s,\cdot)\|^2_{L^p_{\lambda}(\Omega\times\RR)}\right)ds;\label{5.101label}\\
\int_{\RR}\left(\int_{\RR}\mathcal{M}_2(t,x,h)\lambda(x)dx\right)^{\frac{2}{p}}|h|^{2H-2}dh
\lesssim&\, \int^t_0(t-s)^{H-1}\cdot \left[\mathcal{N}^*_{\frac{1}{2}-H,\,p}\tilde{u}^{\e,n}_{\eta}(s)\right]^2ds;\label{5.102label}
\end{align}
\begin{equation}\label{J222-bound1}
\begin{split}
&\int_{\RR}\left|\int_{\RR}\mathcal{M}_{3}(t,x,h)\lambda(x)dx\right|^{\frac{2}{p}}\cdot |h|^{2H-2}dh \\
\lesssim&\,\int^t_0(t-s)^{2H-\frac{3}{2}}\Bigg((t-s)^{\frac{3}{2}-2H}\int_{\RR^4}\big| \Box_{t-s}(z,y,h)\big|^2\cdot |y|^{2H-2} \cdot |h|^{2H-2}\\
&\quad \cdot\left(1+\mathbb{E}\left[\left|\tilde{u}^{\e,n}_{\eta}(s,x)\right|^p\right]\right)\lambda(x-z)dxdydhdz\Bigg)^{\frac{2}{p}}ds \\
\lesssim&\,\int^t_0(t-s)^{2H-\frac{3}{2}}\cdot \left(1+\|\tilde{u}^{\e,n}_{\eta}(s,\cdot)\|^2_{L^p_{\lambda}(\Omega\times\RR)}\right)ds.
\end{split}
\end{equation}
Therefore, by \eqref{5.100label}, \eqref{5.101label}, \eqref{5.102label} and \eqref{J222-bound1}, we have
\begin{equation}\label{512.512}
\begin{split}
\left[\mathcal{N}^*_{\frac{1}{2}-H,\,p}\Phi^{\e,n+1}_{1,{\eta}}(t)\right]^2
\lesssim&\,1+\int^t_0\left((t-s)^{H-1}+(t-s)^{2H-\frac{3}{2}}\right)\cdot \|\tilde{u}^{\e,n}_{\eta}(s,\cdot)\|^2_{L^p_{\lambda}(\Omega\times\RR)}ds\\
&+\,\int^t_0(t-s)^{H-1}\cdot \left[\mathcal{N}^*_{\frac{1}{2}-H,\,p}\tilde{u}^{\e,n}_{\eta}(s)\right]^2ds.
\end{split}
\end{equation}

 \textbf{Step 4.} 
By \eqref{511.511} and \eqref{512.512}, we obtain that
\begin{equation}\label{eq u e51}
\begin{split}
&\|\Phi_{1,{\eta}}^{\e,n+1}(t,\cdot)\|^2_{L^p_{\lambda}(\Omega\times\mathbb{R})}
+\left[\mathcal{N}^*_{\frac{1}{2}-H,\,p}\Phi^{\e,n+1}_{1,{\eta}}(t)\right]^2\\
\lesssim&\,1+\int^t_0\left((t-s)^{H-1}+(t-s)^{-\frac{1}{2}}\right)\cdot \|\tilde{u}^{\e,n}_{\eta}(s,\cdot)\|^2_{L^p_{\lambda}(\Omega\times\RR)}ds\\
&+\,\int^t_0\left((t-s)^{H-1}+(t-s)^{-\frac{1}{2}}\right)\cdot\left[\mathcal{N}^*_{\frac{1}{2}-H,\,p}\tilde{u}^{\e,n}_{\eta}(s)\right]^2ds.
\end{split}
\end{equation}
 For any $t\ge0$, let
 $$\tilde{\Psi}^{\e,n}_{\eta}(t):=\|\tilde{u}^{\e,n}_{\eta}(t,\cdot)\|^2_{L^p_{\lambda}(\Omega\times\RR)}+\left[\mathcal{N}^*_{\frac{1}{2}-H,\,p}\tilde{u}^{\e,n}_{\eta}(t)\right]^2.$$
By \eqref{eq SPDE Y-1_n_e1},    \eqref{eq u e52} amd \eqref{eq u e51}, there exists a constant
  $C_{T,p,H,N}>0$ such that
\begin{align*}
\tilde{\Psi}^{\e,n+1}_{\eta}(t)\leq \,C_{T,p,H,N}\left(1+\int^t_0(t-s)^{2H-\frac{3}{2}}\cdot \tilde{\Psi}^{\e,n}_{\eta}(s)ds\right).
\end{align*}
By the extension of Gr\"onwall's lemma   {\cite[lemma 15]{Dalang1999}}, we have
$$
\sup_{n\geq1}\sup_{t\in[0,T]}\tilde{\Psi}^{\e,n}_{\eta}(t)\leq\, C_g,
  $$
  where $C_g$ is a constant independent of  $\eta\in (0,\infty)$  and $\e\in (0,1)$.

By using  the same argument as that in Step 3 of  the proof of \cite[Lemma 4.5]{HW2019} (or  Step 4 in the proof of Lemma \ref{UniBExist}), we have
\begin{align}\label{eq Zp sec 5}
\displaystyle\sup_{{\eta}>0}\|\tilde{u}^{\e}_{\eta}\|_{\mathcal{Z}^p_{\lambda,T}}\leq C_g\ \ \ \ \text {for any } p\ge p_0,
\end{align}
where $C_g$ is a constant independent of $\e\in(0,1)$.

By using the same methods as that in the proof of \cite[Lemma 4.7 (ii), (iii)]{HW2019} and Lemma \ref{TimeSpaceRegBdd} (ii), (iii),
we have that, for any   $R>0$ and $\gamma>0$, there exists $\theta>0$ such that for each $i=1,2$,
\begin{align*}
\lim_{\theta\downarrow 0}\mathbb{P}\left(\left\{\Phi_{i,{\eta}}^{\e}\in{\mathcal C}([0,T]\times\RR): m^{T,R}\left(\Phi_{i,{\eta}}^{\e},\theta\right)>\gamma\right\}\right)=0,
\end{align*}
where
\begin{align*}
     m^{T,R}\left(\Phi_{i,{\eta}}^{\e},\theta\right):= \max\limits_{\substack{|t-s|+|x-y|\leq \theta, \\ 0\leq t,s \leq T,\,-R\leq x, y\leq R }}|\Phi_{i,{\eta}}^{\e}(t,x)-\Phi_{i,{\eta}}^{\e}(s,y)|.
     \end{align*}

By Lemma \ref{TechLemma7}, the family $\{\tilde{u}^{\e}_{\eta}\}_{\eta>0}$ is tight on the space ${\mathcal C}([0,T]\times\RR)$. Thus,
$\tilde{u}^{\e}_{\eta}\rightarrow \tilde{u}^{\e}$ almost surely in the space $\left(\mathcal{C}([0,T]\times\RR),d_\mathcal{C}\right)$ as $\eta\rightarrow 0$. By \eqref{eq Zp sec 5} and \cite[Lemma 4.6]{HW2019}, we have 
$$\displaystyle\sup_{{\e}\in(0,1)}\|\tilde{u}^{\e}\|_{\mathcal{Z}^p_{\lambda,T}}<\infty, \,\,\,   \text{for any  } p\ge p_0.$$
The proof is complete.
\end{proof}

 For  any $u\in \mathcal{Z}^p_{\lambda,T}$ and $g\in \mathcal U^N$, let 
\begin{align}\label{Phi2definition}
Y(t,x):=\,\int^t_0\left\langle p_{t-s}(x-\cdot)\sigma(s,\cdot,u(s,\cdot)), g(s,\cdot)\right\rangle_{\HH}ds.
\end{align}
 
By using Lemma \ref{TimeSpaceRegBdd} and Minkowski's
inequality, we have the following results.
  {\begin{lemma}\label{lemma-5.2}
 Assume that   $\sigma$ satisfies the hypothesis $(\mathbf{H})$ for some constant $p_0>\frac{6}{4H-1}$. Then we have the following results: 
\begin{itemize}
\item[(i).]for any $p>p_0$,  there exists a constant $ C_{T, p,H,N}>0$  such that
\begin{align} \label{4-96-bound}
\left\|\sup_{t\in[0,T],\,x\in\mathbb{R}}
\lambda^{\frac{1}{p}}(x)\mathcal{N}_{\frac{1}{2}-H}Y(t,x)\right\|_{L^p(\Omega)}\leq \, C_{T,p,H,N}\left(1+\left\|u\right\|_{\mathcal{Z}^p_{\lambda,T}}\right);
\end{align}
    \item[(ii).] if {\bf$p>\frac{3}{H}$} and {\bf$0<\gamma<\frac{H}{2}-\frac{3}{2p}$}, then there exists a positive constant {\bf$C_{T,p,H,N,\gamma}$} such that
  \begin{equation}\label{eq lem51}
  \begin{split}
    \left\|\sup\limits_{\substack{t,\,t+h\in[0,T], \\ x\in \RR}}\lambda^{\frac{1}{p}}(x)\big[Y(t+h,x)  -Y(t,x) \big]\right\|_{L^p(\Omega)}  
   \leq  \,C_{T,p,H,N,\gamma}|h|^\gamma \cdot\left(1+ \left\|u \right\|_{\mathcal{Z}^p_{\lambda,T}}\right);
   \end{split}
   \end{equation}
   \item[(iii).]if {\bf$p>\frac{3}{H}$} and {\bf$0<\gamma<H-\frac{3}{p}$}, then  there exists a positive constant {\bf$C_{T,p,H,N,\gamma}$} such that
   \begin{equation}\label{eq lem52}
   \begin{split}
   \left\|\sup\limits_{\substack{t\in[0,T], \\ x,y\in \RR}}\frac{Y(t,x)- Y(t,y) }{\lambda^{-\frac{1}{p}}(x)+\lambda^{-\frac{1}{p}}(y)}\right\|_{L^p(\Omega)} 
   \leq   \,C_{ T,p,H,N,\gamma}|x-y|^\gamma \cdot \left(1+ \left\|u\right\|_{\mathcal{Z}^p_{\lambda,T}}\right).
   \end{split}
   \end{equation}
\end{itemize}
\end{lemma}}

 Recall  $\tilde u^{\e}$ and $\bar u^{\e}$ defined by \eqref{eq SPDE Y-1} and \eqref{eq SPDE Z-1}, respectively. For any $k\geq 1$ and any $p\geq p_0$, where $p_0$ is defined by \eqref{DuSigamAdd}, define the  stopping time
\begin{equation}\label{Stoptime1}
\begin{split}
\tau_k:=\inf\bigg\{&r\ge 0:\,\,\sup_{0\leq s\leq r,\,x\in\mathbb{R}}
\lambda^{\frac{1}{p}}(x)\mathcal{N}_{\frac{1}{2}-H}\tilde{u}^{\e}(s,x)\geq k,\\
&\,  \text{or }\,\,\sup_{0\leq s\leq r,\,x\in\mathbb{R}}
\lambda^{\frac{1}{p}}(x)\mathcal{N}_{\frac{1}{2}-H}\bar{u}^{\e}(s,x)\geq k\bigg\}.
\end{split}
\end{equation}
By \cite[Lemma 4.7]{HW2019},  Lemmas  \ref{lemma-5-1} and \ref{lemma-5.2}, we know that 
$$
\sup_{\e\in (0,1)}\left\|\sup_{t\in[0,T],\,x\in\mathbb{R}}
\lambda^{\frac{1}{p}}(x)\mathcal{N}_{\frac{1}{2}-H}\tilde u^{\e}(t,x)\right\|_{L^p(\Omega)}<\infty, \  \ \ 
\sup_{\e\in (0,1)}\left\|\sup_{t\in[0,T],\,x\in\mathbb{R}}
\lambda^{\frac{1}{p}}(x)\mathcal{N}_{\frac{1}{2}-H}\bar u^{\e}(t,x)\right\|_{L^p(\Omega)}<\infty.
$$
Those, together with Chebychev's inequality,  imply that  \begin{equation}\label{stop-infinit}
\tau_k\uparrow \infty,\text{ a.s., as } k\rightarrow \infty.
\end{equation}  For any $t\in [0,T]$, let 
\begin{align}
\tilde{u}^{\e}_k(t,\cdot):=\tilde{u}^{\e}(t\wedge \tau_k,\cdot),\ \ \ \ \ \bar{u}^{\e}_k(t,\cdot):=\bar{u}^{\e}(t\wedge \tau_k,\cdot).
\end{align} Obviously,  when $\tau_k> T$,
$\tilde{u}^{\e}_k(t,\cdot)=\tilde{u}^{\e}(t,\cdot)$ and  $\bar{u}^{\e}_k(t,\cdot)=\bar{u}^{\e}(t,\cdot)$   for any $t\in [0,T]$.

\begin{lemma}\label{utildebar}
  {Assume that   $\sigma$ satisfies the hypothesis $(\mathbf{H})$ for some constant  $p_0>\frac{6}{4H-1}$. Then for  any $p\ge p_0$,}
\begin{align}\label{u-d-c-E}
\lim_{\e\rightarrow0}\big\|\tilde{u}_k^{\e}-\bar{u}_k^{\e}\big\|_{\mathcal{Z}_{\lambda,T}^p}=0.
\end{align}
\end{lemma}
\begin{proof} 
Let
\begin{align}\label{Phi1definition}
\Phi_{1}^{\e}(t,x):=\,\int^{t }_0\int_{\RR}p_{t-s}(x-y)\sigma(s,y,\tilde{u}_k^{\e}(s,y)){W}(ds,dy), 
\end{align}
 and 
\begin{equation}\label{Phi3definition}
\begin{split}
 \Phi_2^{\e}(t,x):=\int^{t}_0\left\langle p_{t-s}(x-\cdot)\triangle(s,\cdot,\cdot). g^{\e}(s,\cdot)\right\rangle_{\HH}ds.
\end{split}
\end{equation} 
where  $\triangle(t,x,y):=\sigma(t,x,\tilde{u}^{\e}_k(t,y))-\sigma(t,x,\bar{u}^{\e}_k(t,y))$.
Then, by \eqref{eq SPDE Y-1} and \eqref{eq SPDE Z-1}, we have
\begin{align}\label{5128-inequality}
\big\|\tilde{u}_k^{\e}-\bar{u}_k^{\e}\big\|_{\mathcal{Z}_{\lambda,T}^p}\leq\,\sqrt{\e}
\big\|\Phi_1^{\e}\big\|_{\mathcal{Z}_{\lambda,T}^p}+\big\|\Phi_2^{\e}\big\|_{\mathcal{Z}_{\lambda,T}^p}.
\end{align}
According to Lemmas 4.5 and  4.6   in \cite{HW2019}, we have $\big\|\Phi_1^{\e}\big\|_{\mathcal{Z}_{\lambda,T}^p}<\infty$ for any $p\ge p_0$. Now, it remains to give an estimate for $\big\|\Phi_2^{\e}\big\|_{\mathcal{Z}_{\lambda,T}^p}$.
 
By using the same technique as that  in the proof of  \eqref{511.511}, we have
\begin{equation}\label{5139}
\begin{split}
&\left\| \Phi_2^{\e}(t,\cdot)\right\|^2_{L^p_{\lambda}(\Omega\times\RR)}\\
\lesssim&\,
\int^t_0\left((t-s)^{H-1}+(t-s)^{-\frac{1}{2}}\right)\cdot
\big\|  \tilde{u}_k^{\e}(s,\cdot)-\bar{u}_k^{\e}(s,\cdot) \big\|_{L^p_{\lambda}(\Omega\times\RR)}^2ds\\
&+\int^t_0(t-s)^{-\frac{1}{2}}\cdot \left[\mathcal{N}^*_{\frac{1}{2}-H,\,p}
\left( \tilde{u}_k^{\e}(s)-\bar{u}_k^{\e}(s)\right)\right]^2ds.
\end{split}
\end{equation}

Next, we deal with the term $\mathcal{N}^*_{\frac{1}{2}-H,p}\left( \Phi_2^{\e}(t)\right)$.
Since $g^{\e}\in \mathcal{U}^N$, by the Cauchy-Schwarz inequality and \eqref{eq H product}, we have
\begin{align}\label{Phi_3ePhi_3etxh}
&\mathbb{E}\left[ \big|\Phi_2^{\e}(t,x)-\Phi_2^{\e}(t,x+h)\big|^p\right] \notag\\
=&\,\mathbb{E}\left[ \Big|\int^t_0\big\langle D_{t-s}(x-\cdot,h)\triangle(s,\cdot,\cdot), g^{\e}(s,\cdot)\big\rangle_{\HH}ds\Big|^p\right] \notag\\
\lesssim&\,\mathbb{E}\left[\int^t_0 \big\| D_{t-s}(x-\cdot,h)\triangle(s,\cdot,\cdot)\big\|_{\HH}^2ds \right]^{\frac{p}{2}} \\
\simeq&\,\mathbb{E}\Bigg[\int^t_0\int_{\RR^2} \big| D_{t-s}(x-y-l,h)\triangle(s,y+l,y+l)\notag\\
&\qquad\qquad\qquad\qquad-D_{t-s}(x-y,h)\triangle(s,y,y)\big|^2\cdot |l|^{2H-2}dldyds\Bigg]^{\frac{p}{2}} \notag\\
\lesssim&\,\mathcal{R}_1^{\e}(t,x,h)+\mathcal{R}_2^{\e}(t,x,h)+\mathcal{R}_3^{\e}(t,x,h).\notag
\end{align}
where
\begin{align*}
\mathcal{R}_1^{\e}(t,x,h):=&\,\mathbb{E}\Bigg[\int^t_0\int_{\RR^2} \big| D_{t-s}(x-y-l,h)\big|^2\\
&\qquad\qquad \cdot\big|\triangle(s,y+l,y+l)
 -\triangle(s,y+l,y)\big|^2\cdot | l |^{2H-2}dldyds\Bigg]^{\frac{p}{2}};\\
\mathcal{R}_2^{\e}(t,x,h):=&\,\mathbb{E}\Bigg[\int^t_0\int_{\RR^2} \big| \Box_{t-s}(x-y,-l,h)\big|^2\cdot \big|\triangle(s,y+l,y)\big|^2\cdot|l|^{2H-2}dldyds\Bigg]^{\frac{p}{2}};\\
\mathcal{R}_3^{\e}(t,x,h):=&\,\mathbb{E}\Bigg[\int^t_0\int_{\RR^2} \big| D_{t-s}(x-y,h)\big|^2\cdot \big|\triangle(s,y+l,y)-\triangle(s,y,y)\big|^2 \cdot|l|^{2H-2}dldyds\Bigg]^{\frac{p}{2}}.
\end{align*}
 By  a change of variable, we have
 \begin{align*}
\mathcal{R}_{1}^{\e}(t,x,h)\lesssim&\,
 \mathbb{E}\Bigg[\int^t_0\int_{\RR^2} \big| D_{t-s}(x-y,h)\big|^2\notag \\
 &\qquad\qquad\cdot\big|\tilde{u}^{\e}_k(s,y+l)-\bar{u}^{\e}_k(s,y+l)
 -\tilde{u}_k^{\e}(s,y)+\bar{u}_k^{\e}(s,y)\big|^2\cdot|l|^{2H-2}dldyds\Bigg]^{\frac{p}{2}}\\
&+\,\mathbb{E}\Bigg[\int^t_0\int_{\RR} \big| D_{t-s}(x-y,h)\big|^2\cdot \big|\tilde{u}_k^{\e}(s,y)-\bar{u}_k^{\e}(s,y)\big|^2 \notag \\
&\qquad\qquad\cdot\left(\int_{\RR}\lambda^{\frac{2}{p_0}}(y)\big|\tilde{u}^{\e}_k(s,y+l)-\tilde{u}_k^{\e}(s,y)\big|^2\cdot |l|^{2H-2}dl\right)dyds\Bigg]^{\frac{p}{2}} \\
\lesssim&\,k^p\mathbb{E}\Bigg[\int^t_0\int_{\RR} \big| D_{t-s}(x-y,h)\big|^2\cdot\big|\tilde{u}_k^{\e}(s,y)-\bar{u}_k^{\e}(s,y)\big|^2dyds\Bigg]^{\frac{p}{2}}\\
&+\,\mathbb{E}\Bigg[\int^t_0\int_{\RR} \big| D_{t-s}(x-y,h)\big|^2\cdot \big|\tilde{u}^{\e}_k(s,y)-\bar{u}_k^{\e}(s,y)\big|^2 \notag \\
&\qquad\qquad\cdot\left(\int_{\RR}\lambda^{\frac{2}{p_0}}(y)\big|\bar{u}_k^{\e}(s,y+l)-\bar{u}_k^{\e}(s,y)\big|^2\cdot |l|^{2H-2}dl\right)dyds\Bigg]^{\frac{p}{2}} \\
\lesssim&\,k^p\mathbb{E}\Bigg[\int^t_0\int_{\RR} \big| D_{t-s}(x-y,h)\big|^2\cdot\big|\tilde{u}_k^{\e}(s,y)-\bar{u}_k^{\e}(s,y)\big|^2dyds\Bigg]^{\frac{p}{2}}\\
=:&\,\mathcal{R}_{11}^{\e}(t,x,h)+\mathcal{R}_{12}^{\e}(t,x,h)+\mathcal{R}_{13}^{\e}(t,x,h),
\end{align*}
where $p_0$ is the constant  given by  \eqref{DuSigamAdd}.

By using a change of variable, the Minkowski inequality, Jensen's inequality and Lemma \ref{TechLemma5}, we have
\begin{equation}\label{J21-bound}
\begin{split}
&\int_{\RR}\left(\int_{\RR}\mathcal{R}_{11}^{\e}(t,x,h)\lambda(x)dx\right)^{\frac{2}{p}}\cdot |h|^{2H-2}dh \\
\lesssim&\,\int^t_0\int_{\RR}(t-s)^{H-1}\cdot \Bigg[\int_{\RR^3}(t-s)^{1-H}\big| D_{t-s}(y,h)\big|^2\cdot |h|^{2H-2}\\
&\, \, \cdot \mathbb{E}\bigg[ \big|\tilde{u}_k^{\e}(s,x+l)
 -\bar{u}^{\e}_k(s,x+l) -\tilde{u}_k^{\e}(s,x)+\bar{u}^{\e}_k(s,x)\big|^p\bigg]
\cdot\lambda(x-y)dxdydh\Bigg]^{\frac{2}{p}}\cdot|l|^{2H-2}dlds\\
\lesssim&\,\int^t_0(t-s)^{H-1}\cdot \left[\mathcal{N}^*_{\frac{1}{2}-H,\,p}
\left(  \tilde{u}_k^{\e}(s)-\bar{u}_k^{\e}(s) \right)\right]^2ds,
\end{split}
\end{equation}
and
\begin{align}\label{J2223-bound}
&\int_{\RR}\left(\int_{\RR}\left(\mathcal{R}_{12}^{\e}(t,x,h)+\mathcal{R}_{13}^{\e}(t,x,h)\right)\lambda(x)dx\right)^{\frac{2}{p}}\cdot|h|^{2H-2}dh \notag\\
\lesssim&\,  {k^2}\int^t_0(t-s)^{H-1}\cdot \Bigg[\int_{\RR^3}(t-s)^{1-H}\cdot \big| D_{t-s}(y,h)\big|^2\cdot|h|^{2H-2} \notag\\
&\qquad\cdot\mathbb{E}\left[ \big|\tilde{u}_k^{\e}(s,x)-\bar{u}_k^{\e}(s,x)\big|^p\right]\cdot\lambda(x-y)dxdydh\Bigg]^{\frac{2}{p}}ds \\
\lesssim&\,  {k^2}\int^t_0(t-s)^{H-1}\cdot \left\| \tilde{u}_k^{\e}(s,\cdot)-\bar{u}_k^{\e}(s,\cdot)\right\|^2_{L^p_\lambda(\Omega\times\RR)}ds.\notag
\end{align}
By   \eqref{Lipschitzianuniformlycondu}, we have
\begin{align*}
\mathcal{R}_2^{\e}(t,x,h)\lesssim\,\mathbb{E}\left[\int^t_0\int_{\RR^2} \big| \Box_{t-s}(x-y,-l,h)\big|^2\cdot\big|\tilde{u}_k^{\e}(s,y)-\bar{u}_k^{\e}(s,y)\big|^2 \cdot |l|^{2H-2}dldyds\right]^{\frac{p}{2}}.
\end{align*}
Hence, by a change of variable, Minkowski's inequality, Jensen's inequality and Lemma \ref{TechLemma5}, we have
\begin{equation}\label{J222-bound}
\begin{split}
&\int_{\RR}\left|\int_{\RR}\mathcal{R}_{2}^{\e}(t,x,h)\lambda(x)dx\right|^{\frac{2}{p}}\cdot |h|^{2H-2}dh \\
\lesssim&\,\int^t_0(t-s)^{2H-\frac{3}{2}}\cdot \Bigg((t-s)^{\frac{3}{2}-2H}\cdot\int_{\RR^4}\big| \Box_{t-s}(y,l,h)\big|^2\cdot|l|^{2H-2} |h|^{2H-2} \\
&\cdot\mathbb{E}\big[ \big|\tilde{u}_k^{\e}(s,x)-\bar{u}_k^{\e}(s,x)\big|^p\big]\lambda(x-y)dxdldydh\Bigg)^{\frac{2}{p}}ds \\
\lesssim&\,\int^t_0(t-s)^{2H-\frac{3}{2}}\cdot \left\| \tilde{u}_k^{\e}(s,\cdot)-\bar{u}_k^{\e}(s,\cdot)\right\|^2_{L^p_\lambda(\Omega\times\RR)}ds.
\end{split}
\end{equation}
By   Lemma \ref{TechLemma5}, we have 
\begin{equation}\label{5-155R1}
\begin{split}
&\int_{\RR}\left(\int_{\RR}\mathcal{R}_{3}^{\e}(t,x,h)\lambda(x)dx\right)^{\frac{2}{p}}\cdot |h|^{2H-2}dh \\
\lesssim&\,\int^t_0(t-s)^{H-1}\cdot \left\|\tilde{u}_k^{\e}(s,\cdot)-\bar{u}_k^{\e}(s,\cdot)\right\|^2_{L^p_\lambda(\Omega\times\RR)}ds.
\end{split}
\end{equation}
Putting \eqref{Phi_3ePhi_3etxh}, \eqref{J21-bound}, \eqref{J2223-bound}, (\ref{J222-bound}) and \eqref{5-155R1} together, we have
\begin{equation}\label{5147}
\begin{split}
&\left[\mathcal{N}^*_{\frac{1}{2}-H,p}\left( 
\Phi_2^{\e}(t)\right)\right]^2\\
:=&\int_{\RR}\left[\int_{\RR}\mathbb{E}\left[ \left|\Phi_2^{\e}(t,x)-\Phi_2^{\e}(t,x+h)\right|^p\right]\lambda(x)dx\right]^{\frac{2}{p}}\cdot|h|^{2H-2}dh \\
\lesssim&\,  {k^2}\int^t_0\left((t-s)^{H-1}+(t-s)^{2H-\frac{3}{2}}\right)\cdot \left\| \tilde{u}_k^{\e}(s,\cdot)-\bar{u}_k^{\e}(s,\cdot)\right\|^2_{L^p_\lambda(\Omega\times\RR)}ds\\
&+\int^t_0(t-s)^{H-1}\cdot \left[\mathcal{N}^*_{\frac{1}{2}-H,\,p}
\left( \tilde{u}_k^{\e}(s)-\bar{u}_k^{\e}(s)\right)\right]^2ds .
\end{split}
\end{equation}
  {It follows from Lemma \ref{proposition-4-3} that  $\left\| \tilde{u}^{\e}_k-\bar{u}_k^{\e}\right\|_{\mathcal{Z}_{\lambda,T}^p}$ is finite.}
By \eqref{5128-inequality}, \eqref{5139}, \eqref{5147}  and the fractional Gr\"onwall lemma (  {\cite[Lemma 7.1.1]{Herry1981}}), we have   {that, for any fixed  $k\ge1$ and  $p>p_0$,}
 \begin{align*}
  \lim_{\e\rightarrow0}\big\| \tilde{u}^{\e}_k-\bar{u}^{\e}_k\|_{\mathcal{Z}_{\lambda,T}^p}=0.
\end{align*} 
The proof is complete.
\end{proof}

We now give the proof of Proposition \ref{proposition-4-3}.
\begin{proof}[Proof of Proposition \ref{proposition-4-3}]  {
 By \cite[Lemma 4.7]{HW2019}, Lemmas  \ref{lemma-5-1}, \ref{lemma-5.2}    and 
  \ref{TechLemma7}, we know that  the probability measures on the space $(\mathcal{C}([0,T]\times\RR),\mathcal{B}(\mathcal{C}([0,T]\times\RR)), d_{\mathcal{C}})$ corresponding to the processes    $\{\tilde{u}^{\e}-\bar{u}^{\e}\}_{\e>0}$   are tight. Thus, there is a subsequence $\e_n\downarrow 0$ such that $\tilde{u}^{\e_n}-\bar{u}^{\e_n}$ converges weakly  
  to some stochastic process $Z=\{Z(t,x),  t\in[0,T], x\in \mathbb R\}$ in  $(\mathcal{C}([0,T]\times\RR),\mathcal{B}(\mathcal{C}([0,T]\times\RR)), d_{\mathcal{C}})$. }
  
    {
 On the other hand, for any $k\ge1, p\ge p_0, \gamma>0$,
\begin{equation}\label{eq P}
\begin{split}
& \mathbb P\left(\sup_{0\le t\le T}\int_{\mathbb  R}\left|\tilde{u}^{\e}(t,x)-\bar{u}^{\e}(t,x)\right|^p\lambda(x) dx >\gamma\right)\\
\le\, & \mathbb P\left(\sup_{0\le t\le T}\int_{\mathbb  R}\left|\tilde{u}^{\e}(t,x)-\bar{u}^{\e}(t,x)\right|^p\lambda(x) dx >\gamma, \tau_k>T\right) 
 +\mathbb P\left(\tau_k>T\right)\\
 \le \, &\mathbb P\left(\sup_{0\le t\le T}\int_{\mathbb  R}\left|\tilde{u}_k^{\e}(t,x)-\bar{u}_k^{\e}(t,x)\right|^p\lambda(x) dx >\gamma\right) 
 +\mathbb P\left(\tau_k>T\right).
\end{split}
\end{equation}
First   letting $\e\rightarrow0$ and then letting  $k\rightarrow\infty$,    by  Lemma \ref{utildebar} and \eqref{stop-infinit}, we have     
 $$
\sup_{0\le t\le T}\int_{\mathbb  R}\left|\tilde{u}^{\e}(t,x)-\bar{u}^{\e}(t,x)\right|^p\lambda(x) dx   \rightarrow 0 \text{ in probability,\,\,\,\,  as }\,  \e\rightarrow0.
 $$
 Thus, for any fixed $t\in [0,T]$, the processes  $\{\tilde{u}^{\e}(t, x)(\omega)-\bar{u}^{\e}(t,x)(\omega); (\omega,x)\in \Omega\times\mathbb R\}$ converges in probability in the product probability space $( \Omega\otimes\mathbb R, \mathbb P\otimes\lambda(x)dx )$. This, together with the weak convergence  of $\{\tilde{u}^{\e_n}-\bar{u}^{\e}\}_{n\ge1}$ and the uniqueness of the limit distribution  of  $\tilde{u}^{\e_n}-\bar{u}^{\e_n}$, implies that   $Z(t,x)\equiv0, t\in [0,T], x\in \mathbb R$, almost surely. 
Thus,  as $\e\rightarrow0$,  $\tilde{u}^{\e}-\bar{u}^{\e}$ converges weakly to  $0$ in $(\mathcal{C}([0,T]\times\RR),\mathcal{B}(\mathcal{C}([0,T]\times\RR)), d_{\mathcal{C}})$. Equivalently,    }
       {the sequence of real-valued random variables
$   d_\mathcal{C}(\tilde{u}^{\e},\bar{u}^{\e})$ converges to $0$  in distribution as $\e$ goes to $0$.}
This  implies that as $\e\rightarrow0$,
$$  d_\mathcal{C}(\tilde{u}^{\e},\bar{u}^{\e})\longrightarrow 0  \text{ in probability.} $$
  {See, e.g.,   \cite[p. 98, Exercise 4]{Chung}.}
The proof is complete.
\end{proof}

\section{Appendix}
In this section, we  give some lemmas related to the heat kernel $p_t(x)$. Recall  $\lambda(x)$ defined by \eqref{lamd}, $D_{t}(x,h)$, $\Box_{t}(x,y,h)$   defined by \eqref{TechDt} and \eqref{TechBoxt}, respectively.

\begin{lemma}(\cite[Lemma 2.5]{HW2019})\label{TechLemma1} For any   $T>0$,
   \begin{equation}\label{Glamd}
     \sup_{t\in[0,T]}\sup_{x\in\RR} \frac{1}{\lambda(x)}\int_{\RR} p_{t}(x-y) \lambda(y) dy<\infty.
   \end{equation}
 \end{lemma}

\begin{lemma}(\cite[Lemma 2.8]{HW2019})\label{TechLemma2}
For any $H\in (\frac{1}{4},\frac{1}{2})$, there exists some constant $C_H$ such that
\begin{align}\label{Hu-2-17}
\int_{\RR^2}|D_{t}(x,h)|^2\cdot |h|^{2H-2}dhdx= C_Ht^{H-1},
\end{align}
and
\begin{align}\label{Hu-2-18}
\int_{\RR^3}|\Box_{t}(x,y,h)|^2\cdot |h|^{2H-2}\cdot |y|^{2H-2}dydhdx= C_H t^{2H-\frac{3}{2}}.
\end{align}
\end{lemma}

\begin{lemma}(\cite[Lemma 2.10]{HW2019})\label{TechLemma3}
For any $t>0$, there exists some constant $C_H$ such that
\begin{align}\label{Hu-2-22}
\int_{\RR}|D_{t}(x,h)|^2\cdot |h|^{2H-2}dh\leq C_H\left(t^{H-\frac{3}{2}}\wedge\frac{|x|^{2H-2}}{\sqrt{t}}\right),
\end{align}
where $0<H<\frac{1}{2}$.
\end{lemma}

\begin{lemma}(\cite[Lemma 2.11]{HW2019})\label{TechLemma4}
For any $t>0$, there exists some constant $C_H$ such that
\begin{align}\label{Hu-2-25}
\int_{\RR^2}|\Box_{t}(x,y,h)|^2\cdot |h|^{2H-2}\cdot |y|^{2H-2}dydh\leq C_H\left(t^{2H-2}\wedge\frac{|x|^{2H-2}}{t^{1-H}}\right).
\end{align}
\end{lemma}

\begin{lemma}(\cite[Lemma 2.12]{HW2019})\label{TechLemma5}
For any $t>0$, there exists some constant $C_{T,H}$ such that
\begin{align}\label{Hu-2-26}
\int_{\RR^2}|D_{t}(x,h)|^2\cdot |h|^{2H-2}\lambda(z-x)dxdh\leq C_{T,H}t^{H-1}\lambda(z),
\end{align}
and
\begin{align}\label{Hu-2-27}
\int_{\RR^3}|\Box_{t}(x,y,h)|^2\cdot |h|^{2H-2}\cdot |y|^{2H-2}\lambda(z-x)dxdydh\leq C_{T,H}t^{2H-\frac{3}{2}}\lambda(z).
\end{align}
\end{lemma}

\begin{lemma}(\cite[(4.29)]{HW2019},  \cite[(4.32)]{HW2019})
For some fixed $\gamma\in(0,1)$ and  $\alpha\in(0,1)$, the following two inequalities hold:
\begin{align}\label{6-154}
|(t+h)^{\alpha-1}-t^{\alpha-1}|\lesssim|t|^{\alpha-1-\gamma}h^\gamma,
\end{align}
and
\begin{align}\label{6-155}
 |p_{t+h}(x)-p_{t}(x)|
\lesssim  h^\gamma t^{-\gamma}\left[p_{\frac{2}{\gamma}(t+h)}(x)+p_{\frac{2t}{\gamma}}(x)\right].
\end{align}	
\end{lemma}
\begin{lemma}(\cite[Theorem 4.4]{HW2019})\label{TechLemma7}
A sequence $\{\mathbb{P}_n\}^\infty_{n=1}$ of probability measures on $(\mathcal{C}([0,T]\times\RR),\mathcal{B}(\mathcal{C}([0,T]\times\RR)))$ is tight if and only if the following conditions hold:
 \begin{itemize}
     \item[(i).] $\displaystyle\lim_{\lambda\uparrow\infty} \sup_{n\geq1} \mathbb{P}_n\left(\{\omega\in\mathcal{C}([0,T]\times\RR):|\omega(0,0)|>\lambda\}\right)=0$.
 \item[(ii).] for any   $R>0$ and $\gamma>0$,
 $$
 \lim_{\delta\downarrow0}\sup_{n\geq1}\mathbb{P}_n\left(\left\{\omega\in\mathcal{C}([0,T]\times\RR):m^{T,R}(\omega,\theta)>\gamma\right\}\right)=0  {,}
 $$
 where
 $$
 m^{T,R}(\omega,\theta):=\max\limits_{\substack{|t-s|+|x-y|\leq\theta, \\ 0\leq t,s\leq T;\,-R\leq x, y\leq R}}\left|\omega(t,x)-\omega(s,y)\right|
 $$
 is the modulus of continuity on $[0,T]\times[-R,R]$.
\end{itemize}
\end{lemma}

\vskip0.5cm
\noindent{\bf Acknowledgments }The authors are grateful to the anonymous referees for their constructive comments and corrections which have lead to significant improvement of this paper.    The research of R. Li is partially supported by Shanghai Sailing Program grant 21YF1415300 and NNSFC grant 12101392.
  The research of  R. Wang  is partially supported by  NNSFC grants 11871382 and 12071361. The research of B. Zhang is
partially supported by NNSFC grants 11971361 and 11731012.

 \vskip0.5cm

\noindent{\bf Data Availability}  Data sharing is not applicable to this article as no datasets were generated or analyzed during the current study.
\vskip0.5cm
\noindent{\large{\bf Declarations}}
\vskip0.5cm
\noindent{\bf  Conflict of interest} The authors declare  that they have no conflict of interest.
\vskip0.5cm

\noindent{\bf Authors' contribution } \\
Ruinan Li: Writing - Review \& Editing;\\
 Ran Wang: Writing - Review \& Editing; \\
 Beibei Zhang:  Writing - Review \& Editing.

\end{document}